\documentclass{article}
\usepackage[utf8]{inputenc}
\usepackage{mathtools}
\usepackage{amssymb}
\usepackage{amsthm}
\usepackage{tcolorbox}
\usepackage{xcolor}
\usepackage{graphicx, overpic}
\usepackage{subcaption}
\usepackage{pgfplots}
\usepackage{float}
\usepackage{algorithm}
\usepackage{algpseudocode}
\usepackage{tikz}
\usepackage{bm}

\usepackage{geometry}

\usepackage{hyperref}
\hypersetup{
	colorlinks=true,
	linkcolor=blue,
	filecolor=magenta,      
	urlcolor=cyan,
}

\newcommand{\bI}{\mathbf I}

\newcommand{\bP}{\mathbf P}

\newcommand{\bn}{\mathbf n}

\newcommand{\bx}{\mathbf x}

\newcommand{\T}{\mathcal T}

\newcommand{\divG}{{\mathop{\,\rm div}}_{\Gamma}}
\newcommand{\gradG}{\nabla_{\Gamma}}

\newcommand{\nablaGh}{\nabla_{\hspace{-0.3ex}\Gamma_{\hspace{-0.3ex} h}}}
\newcommand{\laplG}{\Delta_{\Gamma}}

\newcommand{\OGamma}{\Omega^\Gamma_h}

\DeclareGraphicsExtensions{.pdf,.eps,.ps,.eps.gz,.ps.gz,.eps.Y}

\def\cl {\nonumber \\}
\def\el {\nonumber }

\newtheorem{theorem}{Theorem}[section]

\begin{document}
	\title{A scalar auxiliary variable unfitted FEM for the surface Cahn--Hilliard equation}
	
	\author{Maxim Olshanskii \and Yerbol Palzhanov \and Annalisa Quaini}

	\date{\small Department of Mathematics, University of Houston, 3551 Cullen Blvd, Houston, Texas, 77204; \texttt{\{maolshanskiy\}\{ypalzhanov\}\{aquaini\}@uh.edu}}
	
	\maketitle
	
	\begin{abstract}
The paper studies a scalar auxiliary variable (SAV) method to solve the Cahn--Hilliard equation with degenerate mobility posed
on a smooth closed surface $\Gamma$. The SAV formulation is combined with adaptive time stepping and
a geometrically unfitted trace finite element method (TraceFEM), which embeds  $\Gamma$ in $\mathbb{R}^3$.
The stability is proven to hold in an appropriate sense for both first- and second-order in time variants of the method.
The performance of our SAV method is illustrated through a series of numerical experiments, which include
systematic comparison with a stabilized semi-explicit method.	
\end{abstract}

\section{Introduction}

Many physical systems are described by PDEs in the form of gradient flow 
of a free energy functional $E(\phi)$, $\phi$ being the problem unknown. 
Assuming that  $E(\phi)$ is bounded from below, the  gradient flow in a Hilbert space $H$ is defined by the identity
\begin{align}
\Big\langle \frac{\partial \phi}{\partial t},\eta \Big\rangle = -\frac{\delta E}{\delta \phi}[\eta], \label{eq:gradflow}
\end{align}
which holds for all test functions $\eta\in H$,  $\langle \cdot ,\cdot \rangle$ being the inner product in $H$ and $\frac{\delta E}{\delta \phi}$ the functional derivative of $E$ with respect to $\phi$. 
PDEs in the form of gradient flow are often derived from the second law of thermodynamics.
Examples include models for thin films (see, e.g., \cite{Giacomelli2001}), polymers (see, e.g., \cite{Fraaije2003}), and liquid crystals
(see, e.g., \cite{PhysRevE.58.7475}). In this paper, we focus on one particular gradient flow problem: 
the Cahn--Hilliard equation posed on a surface. Our interest in this problem stems from its application in biomembrane modeling 
\cite{hausser2013thermodynamically,garcke2016coupled,Yushutin_IJNMBE2019,zhiliakov2021experimental}. 

Constructing efficient, robust, and energy-stable numerical schemes for gradient flow problems is not a trivial task. If one is not careful in designing a numerical scheme that preserves the energy dissipation mechanism inherent in gradient flows, an extremely small time step might be required to dissipate energy, resulting in an inefficient scheme. For a comprehensive review of numerical schemes for gradient flows, we refer to \cite{Shen2019_SIAMrev}.
One effective numerical technique for a broad range of gradient flows is the scalar auxiliary variable (SAV) method \cite{shen2018scalar}, which enables the construction of efficient and accurate time discretization schemes. Since its introduction in \cite{shen2018scalar}, the SAV method has been developed and applied to various problems, including epitaxial thin film growth models \cite{cheng2019highly}, models for single- and multi-component Bose-Einstein condensates \cite{zhuang2019efficient}, and the square phase field crystal model \cite{wang2021second}.

SAV methods for our specific problem, namely the Cahn--Hilliard equation, have been extensively studied in volumetric domains. Convergence and error analysis for a first-order semi-discrete SAV scheme are conducted in \cite{Shen2018_SINUM}. An unconditionally energy-stable and second-order accurate SAV algorithm is presented in \cite{chen2019fast}. Error estimates for first and second-order fully discretized SAV schemes, utilizing a mixed finite element discretization for the spatial variables, are derived in \cite{chen2020optimal}. An improvement over the standard SAV method is represented by a class of extrapolated and linearized SAV methods based on Runge–Kutta time integration \cite{akrivis2019energy}. These methods can achieve arbitrarily high-order accuracy for the time discretization of the Cahn–Hilliard problem. 
To the best of our knowledge, it is the first time that the SAV method is applied to the \emph{surface} Cahn--Hilliard problem. 
Other surface problems are treated in \cite{sun2020numerical,sun2022modeling}. 

In early versions of the SAV method \cite{shen2018scalar,shen2019}, numerical efficiency for the Cahn--Hilliard equations is achieved through
the computationally cheap invertibility of the discrete biharmonic operator in simple geometric settings discretized by, e.g., a tensor product finite difference method. For the equations posed on surfaces, such fast solvers are not available, in general. 


Time adaptivity is a crucial feature of a numerical method for the Cahn--Hilliard equation in order to achieve efficiency. This is because the Cahn--Hilliard equation exhibits two distinct time scales. The first time scale is associated with the evolution of the order parameter ($\phi$ in problem \eqref{eq:gradflow}) due to interfacial effects. This time scale is typically fast and governs the short-term behavior of the system.
The second time scale is associated with the relaxation of the order parameter towards equilibrium. This time scale is typically very slow and governs the long-term behavior of the system. Using a constant time step, which is oblivious to these two time scales, would result in a highly inefficient simulation.
In \cite{huang2020highly}, high-order unconditionally stable SAV methods that utilize variable time step sizes were developed to address this issue.


In this paper, for the first time, SAV methods are combined with a geometrically \emph{unfitted} finite element method for the numerical solution of the Cahn-Hilliard problem posed on a surface. In \cite{sun2020numerical,sun2022modeling} instead, the authors have
opted for a fitted finite element method, combined with a exponential-type SAV scheme.
We consider both first-order and second-order backward differentiation formula schemes and prove their energy stability. Additionally, we present a time-adaptive version of the second-order scheme, drawing inspiration from \cite{gomez2008isogeometric}. Implementation details are provided for all the proposed schemes.
The unfitted finite element method we choose for spatial discretization is called the trace finite element method (TraceFEM) \cite{ORG09,olshanskii2017trace}. The selection of an unfitted finite element method is motivated by its flexibility in handling complex shapes, as demonstrated in this paper, and possibly evolving surfaces, as shown in \cite{olshanskii2014eulerian,lehrenfeld2018stabilized,Yushutin2019}. Among all the unfitted finite element methods, TraceFEM offers several advantages that make it appealing: (i) it employs a sharp surface representation, (ii) surfaces can be defined implicitly without the need for surface parametrization, (iii) the number of active degrees of freedom is asymptotically optimal, and (iv) the order of convergence is optimal. 

The paper outline is as follows. Sec.~\ref{sec:pd} states the surface Cahn-Hilliard problem under consideration
and rewrites it using the scalar auxiliary variable. Sec.~\ref{sec:st_disc} presents the first and second order
SAV schemes and provides proof of their energy stability. In Sec.~\ref{sec:at_scheme}, we discuss the 
time adaptive version of the second order SAV scheme. Several numerical results are presented in Sec.~\ref{sec:num_res}
and conclusions are drawn in Sec.~\ref{sec:concl}.

\section{Problem definition}\label{sec:pd}

Let $\Gamma$ be a closed sufficiently smooth surface in $\mathbb{R}^3$, with
the outward pointing unit normal $\bn$. The Cahn--Hilliard equation \cite{cahn1958free,CAHN1961}
posed on $\Gamma$ describes phase separation in a two component system on the surface $\Gamma$. In order to state
this equation, we denote the mass concentrations of the two components with $c_i = m_i/m$, $i = 1, 2$, 
where $m_i$ are their masses and $m$ is the total mass of the system.  Since $m = m_1 + m_2$,
we have $c_1 + c_2 = 1$. Let $c_1$ be the representative concentration $c$, i.e. $c = c_1$.
We note that we choose to work with concentration instead of order parameter $\phi$ like in the generic
problem \eqref{eq:gradflow}.
Let $\rho$ be the constant total density of the system $\rho = m/S$, where
$S$ is the surface area of $\Gamma$. Finally, let $ \divG$, $\gradG$, and $\laplG$
denote the surface divergence, the tangential gradient, and the Laplace--Beltrami operator respectively. 

The Cahn--Hilliard equation is a conservation law for concentration $c(\bx, t)$
that uses an empirical law, called Fick's law, for the diffusion flux:
\begin{align}\label{eq:sys_CH1}
\rho \frac{\partial c}{\partial t} -  \divG (M \gradG \mu) = 0 \quad \text{on}~\Gamma  \times (0, T],  
\end{align}
where $M$ is the so-called mobility coefficient  and
$\mu$ is the chemical potential, which is defined as the functional derivative of the total
specific free energy $f$ with respect to the concentration $c$. The total specific free energy is given by:
\begin{align}\label{eq:total_free_e}
f(c) = f_0(c) + \frac{1}{2} \epsilon^2 | \gradG c |^2,
\end{align}
where $f_0(c)$ is the free energy per unit surface and ${\epsilon}$ is thickness
of the interface layer between the two components. The second term at the right-hand side
in \eqref{eq:total_free_e} represents
the interfacial free energy based on the concentration gradient. By taking the functional derivative
of $f$ in \eqref{eq:total_free_e} with respect to $c$, we obtain: 
\begin{align}
\mu = \frac{\delta E}{\delta c}=f_0'(c) - \epsilon^2 \laplG c \quad \text{on}~\Gamma. \label{eq:sys_CH2}
\end{align}
The physical meaning of $\mu$ is chemical potential. 
Eq.~\eqref{eq:sys_CH1},\eqref{eq:sys_CH2} represent the Cahn--Hilliard problem
written in mixed form, i.e., as two coupled second-order equations.
Obviously, problem \eqref{eq:sys_CH1},\eqref{eq:sys_CH2}  needs to be supplemented with initial condition $c = c_0$
on $\Gamma\times \{0\}$, for a given $c_0$.

System \eqref{eq:sys_CH1},\eqref{eq:sys_CH2} needs to be supplemented with the
definitions of mobility $M$ and free energy per unit surface $f_0$. Out of all the possible definitions of 
$M$, we choose the so-called degenerate mobility:
\begin{align}\label{eq:M}
M = M(c) = c(1 - c).
\end{align}
A common choice for $f_0$ is given by
\begin{align}\label{eq:f0}
f_0(c) = \frac{1}{4} c^2(1 - c)^2.
\end{align}

Straightforward calculations show that for constant mobility $M$, the Cahn--Hilliard problem defines  gradient flows of the energy functional
	\begin{equation}\label{eq:E1}
		E(c) =\int_{\Gamma}f(c)\,ds = \int_{\Gamma} \frac{1}{2}\epsilon^2 |\gradG c|^2\ ds +E_1(c), \quad \text{with}~E_1(c) = \int_{\Gamma}f_0(c)\,ds,
	\end{equation} 
in  $H^{-1}(\Gamma)$ (a dual space to $H^1(\Gamma)$).  For the degenerate mobility, the Cahn--Hilliard problem is 
known to define gradient flows in a weighted-Wasserstein metric~\cite{lisini2012cahn}.
Incorporating various definitions of mobility $M(c)$ into the Cahn-Hilliard equations can exert a notable 
impact on the dynamic behavior of $c$, even without altering the energy landscape. 
Thanks to gradient structure mentioned above, the following energy dissipation property holds:
\begin{align}
  \frac{d}{dt}E(c) <0. \label{E_CH}
\end{align}

A time discretization scheme for problem \eqref{eq:sys_CH1},\eqref{eq:sys_CH2} is said to be energy stable 
if it satisfies a discrete energy dissipation law, i.e., it needs to adhere to fundamental property
\eqref{E_CH}. In this paper, we construct an energy stable scheme for \eqref{eq:sys_CH1},\eqref{eq:sys_CH2} 
using the scalar auxiliary variable (SAV) approach (see \cite{Shen2019_SIAMrev} for a review). As the name suggests, 
this method introduces a scalar auxiliary variable
\begin{align}
r(t) = \sqrt{E_1(c(t))+C}, \label{eq:aux_v}
\end{align}
where constant $C$ can be added to ensure that $r(t)$ is well defined. 
Without loss of generality, for the rest of the paper we will assume that $E_1(c)  >0$, i.e., $C = 0$. 
Then, system \eqref{eq:sys_CH1},\eqref{eq:sys_CH2}
can be rewritten as follows:
\begin{align} 
\rho \frac{\partial c}{\partial t}   &=  \divG \left(M(c) \gradG \mu \right)  &&\text{on}~\Gamma\times (0, T] , \label{saveq1} \\
\mu &=  \frac{r(t)}{\sqrt{E_1(c)}}f_0' - \epsilon^2 \laplG c && \text{on}~\Gamma\times (0, T], \label{saveq2}\\
\frac{d r}{d t}  &= \frac{1}{2\sqrt{E_1(c)}}\int_{\Gamma}f_0'(c) \frac{\partial c}{\partial t}  ds &&\text{on}~\Gamma\times (0, T]. \label{saveq3}
\end{align}
System \eqref{saveq1}-\eqref{saveq3} represents the starting point for the construction of 
our energy stable SAV scheme.

For the numerical method presented in the next section, we need a variational formulation of 
surface problem \eqref{saveq1}-\eqref{saveq3}.
To devise it, we multiply \eqref{saveq1} by $v\in H^1(\Gamma)$ and \eqref{saveq2} by $q\in H^1(\Gamma)$, 
integrate over $\Gamma$ and employ the integration by parts identity \cite{Yushutin_IJNMBE2019}.
This leads to the formulation: Find $(c,\mu) \in H^1(\Gamma) \times H^1(\Gamma)$ 
such that
\begin{align}
&\int_\Gamma \rho \frac{\partial c}{\partial t} \,v \, ds = - \int_\Gamma M(c) \gradG \mu \, \gradG v \, ds , \label{eq:sys_CH1_weak} \\
&\int_\Gamma  \mu \,q \, ds = \int_\Gamma \frac{r(t)}{\sqrt{E_1(c)}} f_0'(c) \,q \, ds + \int_\Gamma \epsilon^2 \gradG c \, \gradG q \, ds, \label{eq:sys_CH2_weak}
\end{align}
for all $ (v,q) \in H^1(\Gamma) \times H^1(\Gamma)$, while \eqref{saveq3} remains unchanged.

The majority of the papers available in the literature on SAV methods for the Cahn-Hilliard problem employs constant mobility.
Nevertheless, degenerate mobility has been considered in many practical applications (see, e.g., \cite{Yushutin_IJNMBE2019})
and is non-trivial to handle numerically (see, e.g., \cite{guillen2023energy} for recent advances).
The only paper that proposes a SAV approach for the Cahn--Hilliard equation with degenerate mobility
is \cite{huang2022upwind}. 

\section{Space and time discretization}\label{sec:st_disc}	

For the space discretization of the surface Cahn--Hilliard problem described in the previous section, 
we apply the trace finite element method (TraceFEM)~\cite{olshanskii2017trace,Yushutin_IJNMBE2019}. 
This is an unfitted method that allows to solve for scalar or vector fields on surface $\Gamma$ 
without the need for a parametrization or triangulation of $\Gamma$ itself. As typical 
of unfitted methods, TraceFEM relies on a triangulation of a bulk computational domain $\Omega$ 
($\Gamma\subset\Omega$ holds) into shape regular tetrahedra ``blind'' to the position of $\Gamma$. 
Such position is defined implicitly as the zero level set of a sufficiently smooth (at least Lipschitz continuous) 
function $\phi$, i.e., $\Gamma=\{\bx\in\Omega\,:\, \phi(\bx)=0\}$, such that $|\nabla\phi|\ge c_0>0$ in a 
3D neighborhood of the surface. 

Let $\T_h$ be  the collection of all tetrahedra, such that $\overline{\Omega}=\cup_{T\in\T_h}\overline{T}$.
Typically, we refine the grid  $\T_h$ near $\Gamma$.
The subset of tetrahedra that have a \emph{nonzero intersection} with $\Gamma$ is denoted by $\T_h^\Gamma$.
The domain formed by all tetrahedra in $\T_h^\Gamma$ is denoted by $\OGamma$.
On $\T_h^\Gamma$ we use a standard finite element space of continuous functions that are piecewise-polynomials
of degree $1$. Obviously, other choices of finite elements are possible (see, e.g., \cite{grande2018analysis}).
This bulk (volumetric) finite element space is denoted by $V_h$:
\[
V_h=\{v\in C(\OGamma)\,:\, v\in P_1(T)~\text{for any}~T\in\T_h^{\Gamma}\}.
\]
Finally, to define geometric quantities and for the purpose of numerical integration, we approximate $\Gamma$ with a ``discrete'' surface $\Gamma_h$, which is defined as the zero level set of a $P_1$ Lagrangian interpolant $\phi_h$ for level set function $\phi$ on the given mesh.
The $(\cdot,\cdot)$ inner product and $\|\cdot\|$ norm further denotes the $L^2(\Gamma_h)$ inner product and norm.
The approximate normal vector field  $\bn_h=\nabla \phi_h/|\nabla \phi_h|$ is piecewise smooth on $\Gamma_h$. The orthogonal projection into tangential space is given by $\bP_h(\bx)=\bI-\bn_h(\bx)\bn_h^T(\bx)$ for almost all $\bx\in\Gamma_h$. For $v\in V_h$ the surface gradient on $\Gamma_h$ is easy to compute from the bulk gradient $\nablaGh v=\bP_h\nabla v$.


Let us turn to time discretization. At time instance $t^n=n\Delta t$, with time step $\Delta t=\frac{T}{N}$,  
$c^n$ denotes the approximation of $c(t^n, \bx)$; similar notation is used for other quantities of interest. 
Variational problem \eqref{saveq3}--\eqref{eq:sys_CH2_weak} discretized in space by TraceFEM 
and in time by the implicit Euler (also called BDF1) scheme reads: Given $c_0$ and the associated $E_1(c_0)$ and $r_0$ \eqref{eq:aux_v}, for $n \geq 0$ at time step $t^{n+1}$
find $(c_h^{n+1},\mu_h^{n+1}, r_h^{n+1}) \in V_h \times V_h \times \mathbb{R}$ 
	such that
	\begin{align}
	&\rho(c_h^{n+1} - c_h^{n}, v_h) = -\Delta t (M(c^n)\nablaGh \mu_h^{n+1}, \nablaGh v_h) - h\Delta t \int_{\OGamma} (\textbf{n}_h \cdot \nabla \mu_h^{n+1})(\textbf{n}_h \cdot \nabla v_h)dx, \label{dissaveq1}\\
	& (\mu_h^{n+1}, q_h) = \frac{r_h^{n+1}}{\sqrt{E_1(c_h^n) }}(f_0'(c_h^n),q_h) + \epsilon^2(\nablaGh c_h^{n+1}, \nablaGh q_h) + h^{-1} \epsilon^2\int_{\OGamma} (\textbf{n}_h \cdot \nabla c_h^{n+1})(\textbf{n}_h \cdot \nabla q_h)dx, \label{dissaveq2} \\
	& r_h^{n+1} - r_h^n = \frac{1}{2\sqrt{E_1(c_h^n)}}(f_0'(c_h^n),c_h^{n+1} - c_h^n) \label{dissaveq3}  
	\end{align}
for all $(v_h,q_h) \in V_h \times V_h $.  The volumetric terms in \eqref{dissaveq1}--\eqref{dissaveq2} are
included to stabilize the resulting algebraic systems~\cite{burman2015stabilized,grande2018analysis}. Notice that the nonlinear terms in 
\eqref{dissaveq1}--\eqref{dissaveq3} have been linearized with a first order extrapolation. We will call this approach SAV-BDF1.

\begin{theorem}\label{energydecayBDF1}
	Let 
\begin{align}
\tilde{E}_h^{n+1} = \frac{\epsilon^2}{2} \left\lVert\nablaGh c_h^{n+1}\right\rVert^2 + \left|r_h^{n+1}\right|^2 + h\epsilon^2 \left\lVert \textbf{n}_h \cdot \nabla c_h^{n+1} \right\rVert^2_{L^2(\OGamma)}. \label{eq:mod_e_BDF1}
\end{align}
be the modified discrete energy.
Scheme \eqref{dissaveq1}--\eqref{dissaveq3} admits the following energy balance
\begin{align}
&\left(  \tilde{E}_h^{n+1}- \tilde{E}_h^{n} \right) +\frac{\epsilon^2}{2} \left\lVert\nablaGh c_h^{n+1}-\nablaGh c_h^{n-1}\right\rVert^2 +
\left|r_h^{n+1}-r_h^{n}\right|^2 +  \frac{h  \epsilon^2}{2} \left\lVert \textbf{n}_h \cdot \nabla (c_h^{n+1} - c_h^{n}) \right\rVert^2_{L^2(\OGamma)} \cl
&\quad =  - \frac{\Delta t}{\rho} (M(\widetilde c^n)\nablaGh \mu_h^{n+1}, \nablaGh \mu_h^{n+1})  - \frac{h\Delta t}{\rho} \left\lVert\textbf{n}_h \cdot \nabla \mu_h^{n+1}\right\rVert^2_{L^2(\OGamma)}. \label{eq:e_stab1}
\end{align}
In particular, this implies that the scheme \eqref{dissaveq1}--\eqref{dissaveq3} is energy stable in the sense that 
$ \tilde{E}_h^{n+1} \le \tilde{E}_h^{n}$ (the discrete analogue of \eqref{E_CH}) for all $n=0,1,2,\dots$.
\end{theorem}
\begin{proof}
Combine the equations obtained from taking $v_h = \mu_h^{n+1}/\rho$ in \eqref{dissaveq1} and $q_h = (c_h^{n+1} - c_h^{n})$
in \eqref{dissaveq2} to get
\begin{align}
\frac{r_h^{n+1}}{\sqrt{E_1(c_h^n) }}(f_0'(c_h^n),c_h^{n+1} &- c_h^{n}) + \epsilon^2(\nablaGh c_h^{n+1}, \nablaGh (c_h^{n+1} - c_h^{n}))\cl
\qquad &+ h^{-1} \epsilon^2\int_{\OGamma} (\textbf{n}_h \cdot \nabla c_h^{n+1})(\textbf{n}_h \cdot \nabla (c_h^{n+1} - c_h^{n}))dx \cl
&\quad  = -\frac{\Delta t}{\rho} (M(\widetilde c^n)\nablaGh \mu_h^{n+1}, \nablaGh \mu_h^{n+1}) - \frac{h\Delta t}{\rho} \left\lVert \textbf{n}_h \cdot \nabla \mu_h^{n+1}\right\rVert^2_{L^2(\OGamma)}. \label{eq:stab1_1} 
\end{align}
By plugging \eqref{dissaveq3} multiplied by $2r_h^{n+1}$ into \eqref{eq:stab1_1}, we obtain:
\begin{align}
&2r_h^{n+1}(r_h^{n+1} - r_h^n) + \epsilon^2 (\nablaGh c_h^{n+1}, \nablaGh (c_h^{n+1} - c_h^{n})) + h^{-1} \epsilon^2\int_{\OGamma} (\textbf{n}_h \cdot \nabla c_h^{n+1})(\textbf{n}_h \cdot \nabla (c_h^{n+1} - c_h^{n}))dx  \cl
&\quad  = -\frac{\Delta t}{\rho} (M(\widetilde c^n)\nablaGh \mu_h^{n+1}, \nablaGh \mu_h^{n+1}) - \frac{h\Delta t}{\rho} \left\lVert \textbf{n}_h \cdot \nabla \mu_h^{n+1}\right\rVert^2_{L^2(\OGamma)}. \label{eq:stab1_2} 
\end{align}
Using identity 
\begin{align}
2\left(a^{k+1}, a^{k+1}-a^{k}\right)= \left|a^{k+1}\right|^2-\left|a^{k}\right|^2 + \left(a^{k+1}-a^{k}\right)^2 \el
\end{align}
in \eqref{eq:stab1_2} leads to \eqref{eq:e_stab1}.
\end{proof}

For a second order scheme in time, we adopt Backward Differentiation Formula of order 2 (BDF2).
A second order approximation of a first time derivative and a linear extrapolation of second order at time $t^n$: 
 \begin{equation}\label{BDF2}
 \frac{\partial c}{\partial t} \approx \frac{3c^{n}-4c^{n-1}+c^{n-2}}{2\Delta t}, \quad
 \tilde{c}^n = 2 c^{n-1} - c^{n-2},
\end{equation}
respectively.
Then, the space and time discrete version of problem \eqref{saveq3}-\eqref{eq:sys_CH2_weak}
reads: Given $c_0$ and the associated $E_1(c_0)$ and $r_0$ \eqref{eq:aux_v}, find $(c_h^{1},\mu_h^{1}, r_h^{1}) \in V_h \times V_h \times \mathbb{R}$
such that \eqref{dissaveq1}-\eqref{dissaveq3} hold and for $n \geq 1$ at time step $t^{n+1}$
find $(c_h^{n+1},\mu_h^{n+1}, r_h^{n+1}) \in V_h \times V_h \times \mathbb{R}$ 
such that 
\begin{align}
	& \frac{\rho}{2\Delta t} \left({3c_h^{n+1} - 4c_h^{n}+c_h^{n-1}}, v_h\right) = - (M(\tilde{c}^n_h)\nablaGh \mu_h^{n+1}, \nablaGh v_h) - h \int_{\OGamma} (\textbf{n}_h \cdot \nabla \mu_h^{n+1})(\textbf{n}_h \cdot \nabla v_h)dx, \label{dissaveq21}\\
	& (\mu_h^{n+1}, q_h) = \frac{r_h^{n+1}}{\sqrt{E_1(\tilde{c}^n_h) }}(f_0'(\tilde{c}^n_h),q_h) + \epsilon^2(\nablaGh c_h^{n+1}, \nablaGh q_h) + h^{-1} \epsilon^2\int_{\OGamma} (\textbf{n}_h \cdot \nabla c_h^{n+1})(\textbf{n}_h \cdot \nabla q_h)dx, \label{dissaveq22} \\
	& 3r_h^{n+1} - 4r_h^n+r_h^{n-1} = \frac{1}{2\sqrt{E_1(\tilde{c}^n_h)}}(f_0'(\tilde{c}^n_h),3c_h^{n+1} - 4c_h^{n}+c_h^{n-1}), \label{dissaveq23}  
\end{align}
for all $(v_h,q_h) \in V_h \times V_h $. We will call this approach SAV-BDF2.

\begin{theorem}\label{energydecayBDF2}
Let 
\begin{align}
		\tilde{E}_h^{n+1} = &\frac{\epsilon^2}{2} \left\lVert\nablaGh c_h^{n+1}\right\rVert^2+\frac{\epsilon^2}{2} \left\lVert 2 \nablaGh c_h^{n+1}-\nablaGh c_h^n \right\rVert^2 +\left|r_h^{n+1}\right|^2+\left|2 r_h^{n+1}-r_h^n\right|^2 \cl
		  &+ \frac{h\epsilon^2}{2} \left\lVert \textbf{n}_h \cdot \nabla c_h^{n+1} \right\rVert^2_{L^2(\OGamma)} + \frac{h^{-1} \epsilon^2}{2}\left\lVert(\textbf{n}_h \cdot \nabla (2c_h^{n+1} - c_h^n)\right\rVert^2_{L^2(\OGamma)}   \label{eq:mod_e_BDF2}
	\end{align}
be the modified discrete energy.
Scheme \eqref{dissaveq1}--\eqref{dissaveq3} admits the following energy balance
	\begin{align}
		&\left(  \tilde{E}_h^{n+1}- \tilde{E}_h^{n} \right) +\frac{\epsilon^2}{2} \left\lVert\nablaGh c_h^{n+1}-2 \nablaGh c_h^n+\nablaGh c_h^{n-1}\right\rVert^2 +
		\left|r_h^{n+1}-2 r_h^n+r_h^{n-1}\right|^2 \cl
		&\quad 
		+  \frac{h  \epsilon^2}{2} \left\lVert \textbf{n}_h \cdot \nabla (c_h^{n+1} - 2c_h^n + c_h^{n-1}) \right\rVert^2_{L^2(\OGamma)} =  - \frac{2\Delta t}{\rho} (M(\widetilde c^n)\nablaGh \mu_h^{n+1}, \nablaGh \mu_h^{n+1})  \cl
		&\quad - \frac{ 2h\Delta t}{\rho} \left\lVert\textbf{n}_h \cdot \nabla \mu_h^{n+1} \right\rVert^2_{L^2(\OGamma)}. \label{eq:e_stab}
	\end{align}
In particular, this implies that the scheme \eqref{dissaveq1}--\eqref{dissaveq3} is energy stable in the sense that 
$ \tilde{E}_h^{n+1} \le \tilde{E}_h^{n}$ (the discrete analogue of \eqref{E_CH}) for all $n=0,1,2,\dots$.
\end{theorem}

\begin{proof}
Combine the equations obtained from taking $ v_h = 2 \Delta t \mu_h^{n+1}/\rho$ in \eqref{dissaveq21}, $q_h = (3c_h^{n+1} - 4c_h^n +c_h^{n-1})$ in \eqref{dissaveq22} to get
\begin{align}
&\frac{r_h^{n+1}}{\sqrt{E_1(\tilde{c}^n_h) }}(f_0'(\tilde{c}^n_h),3c_h^{n+1} - 4c_h^n +c_h^{n-1}) + \epsilon^2 (\nablaGh c_h^{n+1}, \nablaGh (3c_h^{n+1} - 4c_h^n +c_h^{n-1}))  \cl
&\quad + h^{-1} \epsilon^2 \int_{\OGamma} (\textbf{n}_h \cdot \nabla c_h^{n+1})(\textbf{n}_h \cdot \nabla ({3c_h^{n+1} - 4c_h^n +c_h^{n-1}}))dx =  - \frac{2\Delta t}{\rho} (M(\widetilde c^n)\nablaGh \mu_h^{n+1}, \nablaGh \mu_h^{n+1}) \cl
& \quad - \frac{2h\Delta t}{\rho} \left\lVert \textbf{n}_h \cdot \nabla \mu_h^{n+1} \right\rVert^2_{L^2(\OGamma)}. \label{eq:stab_t1}
\end{align}
By plugging \eqref{dissaveq23} multiplied by $2r_h^{n+1}$ into \eqref{eq:stab_t1}, we obtain
\begin{align}
&2r_h^{n+1}  (3r_h^{n+1} - 4r_h^n+r_h^{n-1}) + \epsilon^2 (\nablaGh c_h^{n+1}, \nablaGh (3c_h^{n+1} - 4c_h^n +c_h^{n-1}))  \cl
&\quad + h^{-1} \epsilon^2\int_{\OGamma} (\textbf{n}_h \cdot \nabla c_h^{n+1})(\textbf{n}_h \cdot \nabla ({3c_h^{n+1} - 4c_h^n +c_h^{n-1}}))dx = - \frac{2\Delta t}{\rho} (M(\widetilde c^n)\nablaGh \mu_h^{n+1}, \nablaGh \mu_h^{n+1}) \cl
& \quad - \frac{2h\Delta t}{\rho} \int_{\OGamma} (\textbf{n}_h \cdot \nabla \mu_h^{n+1})(\textbf{n}_h \cdot \nabla \mu_h^{n+1})dx . \el
\end{align}

Let us make use of identity 
\begin{align}
&2\left(a^{k+1}, 3 a^{k+1}-4 a^k+a^{k-1}\right)= \left|a^{k+1}\right|^2+\left|2 a^{k+1}-a^k\right|^2+\left|a^{k+1}-2 a^k+a^{k-1}\right|^2\cl
&\quad -\left|a^k\right|^2-\left|2 a^k-a^{k-1}\right|^2
\end{align}
to get:
\begin{align}
&\left|r_h^{n+1}\right|^2+\left|2 r_h^{n+1}-r_h^n\right|^2+\left|r_h^{n+1}-2 r_h^n+r_h^{n-1}\right|^2 -\left|r_h^n\right|^2-\left|2 r_h^n-r_h^{n-1}\right|^2   +\frac{\epsilon^2}{2} \left\lVert\nablaGh c_h^{n+1}\right\rVert^2 \cl
&\quad  +\frac{\epsilon^2}{2}\left\lVert2 \nablaGh c_h^{n+1}-\nablaGh c_h^n\right\rVert^2+\frac{\epsilon^2}{2}\left\lVert\nablaGh c_h^{n+1}-2 \nablaGh c_h^n+\nablaGh c_h^{n-1}\right\rVert^2 -\frac{\epsilon^2}{2}\left\lVert\nablaGh c_h^n\right\rVert^2-\frac{\epsilon^2}{2}\left\lVert2 \nablaGh c_h^n-\nablaGh c_h^{n-1}\right\rVert^2   \cl
&\quad + \frac{h  \epsilon^2}{2}\left\lVert \textbf{n}_h \cdot \nabla c_h^{n+1}\right\rVert^2_{L^2(\OGamma)} + \frac{h  \epsilon^2}{2} \left\lVert \textbf{n}_h \cdot \nabla (2c_h^{n+1} - c_h^n)\right\rVert^2_{L^2(\OGamma)} + \frac{h  \epsilon^2}{2} \left\lVert \textbf{n}_h \cdot \nabla (c_h^{n+1} - 2c_h^n + c_h^{n-1}\right\rVert^2_{L^2(\OGamma)} \cl
&\quad - \frac{h  \epsilon^2}{2}\left\lVert \textbf{n}_h \cdot \nabla c_h^{n} \right\rVert^2_{L^2(\OGamma)} - \frac{h  \epsilon^2}{2}\left\lVert \textbf{n}_h \cdot \nabla (2c_h^{n} - c_h^{n-1}) \right\rVert^2_{L^2(\OGamma)} = - \frac{2\Delta t}{\rho} (M(\widetilde c^n)\nablaGh \mu_h^{n+1}, \nablaGh \mu_h^{n+1})  \cl
&\quad - \frac{ 2h\Delta t}{\rho} \int_{\OGamma} (\textbf{n}_h \cdot \nabla \mu_h^{n+1})(\textbf{n}_h \cdot \nabla \mu_h^{n+1})dx, \el
\end{align}
which corresponds to \eqref{eq:e_stab}. 
\end{proof}

\subsection{Implementation}\label{sec:impl}

Schemes \eqref{dissaveq1}--\eqref{dissaveq3} and \eqref{dissaveq21}--\eqref{dissaveq23} can be conveniently rewritten
as relatively minor modifications of ``standard'' mixed TraceFEM for the surface Cahn--Hilliard problem.

By plugging $r^{n+1}_h$ obtained from  
\eqref{dissaveq3} into eq.~\eqref{dissaveq2}, we can rewrite problem \eqref{dissaveq1}--\eqref{dissaveq3}
as: Given $c_0$ and the associated $E_1(c_0)$ and $r_0$ \eqref{eq:aux_v}, for $n \geq 0$ at time step $t^{n+1}$
find $(c_h^{n+1},\mu_h^{n+1}) \in V_h \times V_h$
\begin{align}
\frac{\rho}{\Delta t}(c_h^{n+1}, v_h)&+ (M(\widetilde c^n)\nablaGh \mu_h^{n+1},\nablaGh v_h) +h \int_{\OGamma} (\textbf{n}_h \cdot \nabla \mu_h^{n+1})(\textbf{n}_h \cdot \nabla v_h)dx=\frac{\rho}{\Delta t}(c_h^{n}, v_h),  \label{eq:SAM_im3} \\
(\mu_h^{n+1}, q_h) & - \epsilon^2(\nabla c_h^{n+1}, \nabla q_h) - h^{-1} \epsilon^2\int_{\OGamma} (\textbf{n}_h \cdot \nabla c_h^{n+1})(\textbf{n}_h \cdot \nabla q_h)dx\cl  &-\frac{1}{2{E_1(c_h^n) }}(f_0'(c_h^n),c_h^{n+1} )(f_0'(c_h^n),q_h) \cl
&\quad = \frac{r_h^n}{\sqrt{E_1(c_h^n) }}(f_0'(c_h^n),q_h)  - \frac{1}{2{E_1(c_h^n) }}(f_0'(c_h^n), c_h^n)(f_0'(c_h^n),q_h) \label{eq:SAM_im4} 
\end{align}
for all $(v_h,q_h) \in V_h \times V_h $. The only differences between \eqref{eq:SAM_im3}--\eqref{eq:SAM_im4} and
a standard TraceFEM for the surface Cahn--Hilliard problem with the implicit Euler scheme for time discretization 
are the additional last term at the left-hand side in \eqref{eq:SAM_im4}, which corresponds to a rank-one matrix in the algebraic form of the problem,  and the modified  terms at the right-hand side in \eqref{eq:SAM_im4}.
At every time step $t^{n+1}$, the value of the auxiliary variable is computed with \eqref{dissaveq3}.

In a similar way, we plug $r^{n+1}_h$ obtained from  
\eqref{dissaveq23} into eq.~\eqref{dissaveq22} and rewrite problem \eqref{dissaveq21}--\eqref{dissaveq23}
as: Given $c_0$ and the associated $E_1(c_0)$ and $r_0$ \eqref{eq:aux_v}, find $(c_h^{1},\mu_h^{1}) \in V_h \times V_h$
such that \eqref{eq:SAM_im3}-\eqref{eq:SAM_im4} hold and get $r_h^{1}$ from \eqref{dissaveq3}, 
then for $n \geq 1$ at time step $t^{n+1}$ find $(c_h^{n+1},\mu_h^{n+1}) \in V_h \times V_h$ 
such that 
\begin{align}
&\frac{\rho}{2\Delta t}({3c_h^{n+1}}, v_h) + (M(\widetilde c^n)\nablaGh \mu_h^{n+1}, \nablaGh v_h) + h \int_{\OGamma} (\textbf{n}_h \cdot \nabla \mu_h^{n+1})(\textbf{n}_h \cdot \nabla v_h)dx = b_c^{n+1}, \label{eq:SAV_im1} \\
&(\mu_h^{n+1}, q_h) - \epsilon^2(\nablaGh c_h^{n+1}, \nablaGh q_h)  - h^{-1} \epsilon^2\int_{\OGamma} (\textbf{n}_h \cdot \nabla c_h^{n+1})(\textbf{n}_h \cdot \nabla q_h)dx  \cl
&\quad -\frac{1}{2{E_1(c_h^n) }}(f_0'(\tilde{c}^n_h),c_h^{n+1} )(f_0'(\tilde{c}^n_h),q_h) = b_\mu^{n+1} ,\label{eq:SAV_im2}    
\end{align}
for all $(v_h,q_h) \in V_h \times V_h $. The forcing terms in \eqref{eq:SAV_im1}--\eqref{eq:SAV_im2} are computed from known quantities
\begin{align}
&b_c^{n+1} = \frac{2\rho}{\Delta t}({c_h^{n}}, v_h) - \frac{\rho}{2\Delta t}({c_h^{n-1}}, v_h), \cl
&b_\mu^{n+1} = \frac{4r_h^n}{3\sqrt{E_1(c_h^n) }}(f_0'(\tilde{c}^n_h),q_h) - \frac{r_h^{n-1}}{3\sqrt{E_1(c_h^n) }}(f_0'(\tilde{c}^n_h),q_h)- \frac{2}{3{E_1(c_h^n) }}(f_0'(\tilde{c}^n_h), c_h^{n})(f_0'(\tilde{c}^n_h),q_h)\cl
& \quad + \frac{1}{6{E_1(c_h^n) }}(f_0'(\tilde{c}^n_h), c_h^{n-1})(f_0'(\tilde{c}^n_h),q_h). \el
\end{align}
These forcing terms and the last term at the left-hand side in \eqref{eq:SAV_im2} are the only differences
with respect to a standard TraceFEM for the surface Cahn--Hilliard problem with BDF2 for time discretization.
At every time step $t^{n+1}$, the value of the auxiliary variable is computed with \eqref{dissaveq23}.

For the numerical results in Sec.~\ref{sec:num_res}, we use the SAV-BDF2 scheme. 
In summary, we implement is as follows:
\begin{itemize}
\item[-] \emph{Step 0}: from $c_0$, get $E_1(c_0)$ as in \eqref{eq:E1} and $r_0$ from \eqref{eq:aux_v}.
\item[-] \emph{Step 1}: at $t^1 = \Delta t$, solve \eqref{eq:SAM_im3}-\eqref{eq:SAM_im4} to get  $(c_h^{1},\mu_h^{1})$ and compute $r_h^{1}$ from \eqref{dissaveq3}.
\item[-] \emph{Step 2}: at time $t^{n+1}$, $n \geq 1$, solve \eqref{eq:SAV_im1}-\eqref{eq:SAV_im2} to get $(c_h^{n+1},\mu_h^{n+1})$ and compute $r_h^{n+1}$ from \eqref{dissaveq23}.
\end{itemize}

The implementation described above differs from the one presented in the original papers on SAV schemes for gradient flows \cite{huang2020highly,LiShen2019,Shen2019_SIAMrev}. In those papers, the properties of the finite difference method on uniform grids were utilized to enhance computational efficiency. However, since we have chosen to work with finite elements for greater geometric flexibility, we cannot leverage the same properties. As a result, we decided to rewrite the SAV scheme as a minor modification of a standard finite element discretization to simplify the implementation process. Consequently, the additional terms introduced by the SAV method lead to dense matrices in the associated linear systems. 

\section{Adaptive time-stepping scheme}\label{sec:at_scheme}

The dynamic response of the Cahn--Hilliard equation exhibits significant temporal scale variations. Initially, a rapid phase of spinodal decomposition is observed, which can be adequately captured with a small time step (e.g., $\Delta t = \mathcal{O}(10^{-5})$). This phase is followed by a slower process of domain coarsening and growth, for which a larger time step can be employed (e.g., $\Delta t$ ranging from $10^{-1}$ to $10$). As the phase separation process approaches equilibrium, the time step can be further increased (e.g., up to $\Delta t = \mathcal{O}(10^{3})$).
In the literature, various approaches can be found where different time-step sizes are manually set during the simulation. See, e.g.,  \cite{Yushutin_IJNMBE2019}. However, a more intelligent approach to handle such a wide 
range of temporal scales is to employ an adaptive-in-time method that selects the time step based on an accuracy criterion.

We choose to apply the adaptive time stepping technique first presented in~\cite{gomez2008isogeometric}.
Before explaining the algorithm and how the time step is chosen, let us write the time discretization of the space-discrete 
version of problem \eqref{saveq3}-\eqref{eq:sys_CH2_weak} using the BDF2 scheme with a variable 
time step. Let $\Delta t^n = t^{n+1}- t^n$ be the variable time step and set $q^n = {\Delta t^n}/{\Delta t^{n-1}}$.
At time $t^{n+1}$, the time derivative is approximated as follows
 \begin{align*}
&  \frac{\partial c}{\partial t} \approx \frac{\alpha c^{n+1}- \beta c^{n}+ \gamma c^{n-1}}{\Delta t}, \quad \alpha = \frac{1+2q^n}{1+q^n}, \quad \beta = 1+q^n, \quad \gamma = \frac{(q^n)^2}{1+q^n}.
\end{align*}
Then, the fully discrete problem reads: for $n \geq 1$ at time step $t^{n+1}$
find $(c_h^{n+1},\mu_h^{n+1}, r_h^{n+1}) \in V_h \times V_h \times \mathbb{R}$ 
such that 
\begin{align}
	& \frac{\rho}{\Delta t} \left({\alpha c_h^{n+1} - \beta c_h^{n}+ \gamma c_h^{n-1}}, v_h\right) = - (M(\tilde{c}^n_h)\nablaGh \mu_h^{n+1}, \nablaGh v_h) - h \int_{\OGamma} (\textbf{n}_h \cdot \nabla \mu_h^{n+1})(\textbf{n}_h \cdot \nabla v_h)dx, \label{dissaveq21_ta}\\
	& \alpha r_h^{n+1} - \beta r_h^n+ \gamma r_h^{n-1} = \frac{1}{2\sqrt{E_1(\tilde{c}^n_h)}}(f_0'(\tilde{c}^n_h),\alpha c_h^{n+1} - \beta c_h^{n}+\gamma c_h^{n-1}), \label{dissaveq23_ta}  
\end{align}
and \eqref{dissaveq22} hold for all $(v_h,q_h) \in V_h \times V_h $. The formula to compute $\tilde{c}^n_h$ is \eqref{BDF2}.

For the implementation of \eqref{dissaveq22},\eqref{dissaveq21_ta},\eqref{dissaveq23_ta}, we proceed as explained in Sec.~\ref{sec:impl}, 
i.e., we plug $r^{n+1}_h$ obtained from  
\eqref{dissaveq23_ta} into eq.~\eqref{dissaveq22} and rewrite problem \eqref{dissaveq22},\eqref{dissaveq21_ta},\eqref{dissaveq23_ta}
as: for $n \geq 1$ at time step $t^{n+1}$ find $(c_h^{n+1},\mu_h^{n+1}) \in V_h \times V_h$ 
such that 
\begin{align}
&\frac{\rho}{\Delta t}({\alpha c_h^{n+1}}, v_h) + (M(\widetilde c^n)\nablaGh \mu_h^{n+1}, \nablaGh v_h) + h \int_{\OGamma} (\textbf{n}_h \cdot \nabla \mu_h^{n+1})(\textbf{n}_h \cdot \nabla v_h)dx = d_c^{n+1}, \label{eq:SAV_im1_ta} \\
&(\mu_h^{n+1}, q_h) - \epsilon^2(\nablaGh c_h^{n+1}, \nablaGh q_h)  - h^{-1} \epsilon^2\int_{\OGamma} (\textbf{n}_h \cdot \nabla c_h^{n+1})(\textbf{n}_h \cdot \nabla q_h)dx  \cl
&\quad -\frac{1}{2{E_1(c_h^n) }}(f_0'(\tilde{c}^n_h),c_h^{n+1} )(f_0'(\tilde{c}^n_h),q_h) = d_\mu^{n+1} ,\label{eq:SAV_im2_ta}    
\end{align}
for all $(v_h,q_h) \in V_h \times V_h $. The forcing terms in \eqref{eq:SAV_im1_ta}-\eqref{eq:SAV_im2_ta} are computed from known quantities
\begin{align}
&d_c^{n+1} = \frac{\rho}{\Delta t}(\beta {c_h^{n}}, v_h) - \frac{\rho}{\Delta t}(\gamma {c_h^{n-1}}, v_h), \cl
&d_\mu^{n+1} = \frac{\beta r_h^n}{\alpha \sqrt{E_1(c_h^n) }}(f_0'(\tilde{c}^n_h),q_h) - \frac{\gamma r_h^{n-1}}{\alpha \sqrt{E_1(c_h^n) }}(f_0'(\tilde{c}^n_h),q_h)- \frac{\beta}{2 \alpha{E_1(c_h^n) }}(f_0'(\tilde{c}^n_h), c_h^{n})(f_0'(\tilde{c}^n_h),q_h)\cl
& \quad + \frac{\gamma}{2 \alpha{E_1(c_h^n) }}(f_0'(\tilde{c}^n_h), c_h^{n-1})(f_0'(\tilde{c}^n_h),q_h). \el
\end{align}

Now, let us describe the adaptive time stepping technique. Let us call $c^{n+1}_{h,1}$ and  $c^{n+1}_{h,2}$ the solutions
at time $t^{n+1}$ of \eqref{eq:SAM_im3}-\eqref{eq:SAM_im4} and \eqref{eq:SAV_im1_ta}-\eqref{eq:SAV_im2_ta}, respectively.
We define
\begin{align}\label{eq:e}
e^{n+1} = \frac{\lVert c^{n+1}_{h,1} - c^{n+1}_{h,2}\rVert}{\lVert c^{n+1}_{h,2} \rVert},
\end{align}
which is taken as input to update the time step:
\begin{align}\label{eq:F}
\Delta t^{n+1} \gets F(e^{n+1}, \Delta t^{n+1}) =  \zeta \Big(\frac{{tol}}{e^{n+1}}\Big)^{1/2}\Delta t^{n+1},
\end{align}
where $\zeta$ is a ``safety'' coefficient and $tol$ is a user prescribed tolerance. Algorithm \ref{alg:cap} describes the steps to take at
time $t^{n+1}$ in order to adapt the time step. 

\begin{algorithm}
	\caption{Adaptive time-stepping algorithm at time $t^{n+1}$}\label{alg:cap}
\begin{algorithmic}[1]
	\Statex \textbf{Given $c_{n}$ and $\Delta t^n$}
	\State Solve \eqref{eq:SAM_im3}-\eqref{eq:SAM_im4} with $\Delta t^{n+1} = \Delta t^n$ to get $c^{n+1}_{h, 1}$
	\State Solve \eqref{eq:SAV_im1_ta}-\eqref{eq:SAV_im2_ta} with $\Delta t^{n+1} = \Delta t^n$ to get $c^{n+1}_{h,{2}}$ 
	\State Compute $e^{n+1}$ using \eqref{eq:e}
	\If{$e^{n+1} > tol$} 
	\State Update $\Delta t^{n+1}$ using \eqref{eq:F}
	\State \textbf{goto 1}
	\Else
	\State Set $\Delta t^{n+1} = F(e_{n+1}, \Delta t^{n+1})$
	\EndIf
	\Statex \textbf{Continue to $t^{n+2}$}
\end{algorithmic}
\end{algorithm}

Let $r^{n+1} = \Delta t^{n+1}/\Delta t^{n}$ be the time step ratio. Approximately 40 years ago, it was demonstrated that a variable step BDF2 method for ordinary initial-value problems is zero-stable if $r^{n+1} < 1+\sqrt{2}$ \cite{grigorieff1983stability}. Advancing beyond this classical result has proven to be a challenging task, which has recently gained attention. Through the utilization of techniques involving discrete orthogonal convolution kernels, it has been possible to establish that variable time step BDF2 methods are computationally robust, with $0 < r^{n+1} < 3.561$, for linear diffusion models \cite{liao2021analysisdiffusion}, a phase-field crystal model \cite{liao2022adaptive}, and the molecular beam epitaxial model without slope selection \cite{liao2021analysis}. These techniques have been extended to the Cahn-Hilliard model in \cite{liao2022mesh}.
The complexity associated with proving the energy stability of the scheme presented in this section is significant, to the extent that it could be the subject of a separate research paper. Therefore, we will not delve into it in this work.

\section{Numerical results}\label{sec:num_res}

After validating the accuracy of the numerical methods presented in Sec.~\ref{sec:impl}, 
we compare the numerical results obtained with our SAV methods against the results obtained
with a stabilized scheme inspired from~\cite{Shen_Yang2010} and presented in \cite{Yushutin_IJNMBE2019}.
We will start by comparing the numerical results produced by the different methods on a sphere in Sec.~\ref{sec:res_spehre}.
Then, in Sec.~\ref{sec:res_cell} we will present results on a
more complex surface that represents an idealized cell.

For implementation of the methods in Sec.~\ref{sec:st_disc} and \ref{sec:at_scheme}, we use
open source Finite Element package DROPS~\cite{DROPS}.

\subsection{Convergence test}\label{sec:conv_test}

To assess our implementation of the SAV schemes presented in Sec.~\ref{sec:impl}, 
we consider the following exact solution to the non-homogeneous surface Cahn--Hilliard equations 
on the unit sphere, centered at the origin:
\begin{align}
c^*(t, \bx) = \frac{1}{2} \left(1 + \tanh\frac{x_3} {2 \sqrt{2} \epsilon} \right), \quad t\in[0,1]. \label{eq:exact}
\end{align}
Here, $\bx = (x_1, x_2, x_3)^T$ denotes a point in $\mathbb{R}^3$.
The exact chemical potential $\mu^*$ can be readily computed from eq.~\eqref{eq:sys_CH2}
using the free energy per unit surface in \eqref{eq:f0} and the above $c^*$.
The non-zero forcing term is computed by plugging $c^*$ and $\mu^*$ into \eqref{eq:sys_CH1}.
We set $\rho = 1$ and mobility $M$ as in \eqref{eq:M}. In \eqref{eq:aux_v}, we take $C = 1$.
Since it is known that smaller values $\epsilon$ are numerically challenging (see, e.g., \cite{feng2004error,Shen_Yang2010}),
we consider decreasing values of $\epsilon$: $\epsilon = 1, 0.1, 0.05$.

We characterize the surface $\Gamma$ as the zero level set of function $\phi(\bx) = \|\bx\|_2 -1$,
and embed $\Gamma$ in an outer cubic domain $\Omega=[-5/3,5/3]^3$.
The initial triangulation $\T_{h_\ell}$ of $\Omega$ consists of 8 sub-cubes,
where each of the sub-cubes is further subdivided into 6 tetrahedra.
Further, the mesh is refined towards the surface, and $\ell\in\Bbb{N}$ denotes the level of refinement, with the associated 
mesh size $h_\ell= \frac{10/3}{2^{\ell+1}}$. Fig.~\ref{fig:conv_test_sol} shows the approximation of \eqref{eq:exact} with $\epsilon = 0.05$
computed with mesh level $\ell = 6$ and a magnified view of the interface thickness with the bulk mesh near the surface.
The time step is also refined with the mesh as specified below. Time step adaptivity is not used for this test.

\begin{figure}[htb!]
	\centering
	\begin{tikzpicture}
		\node at (9,0) {\includegraphics[width=.2\linewidth]{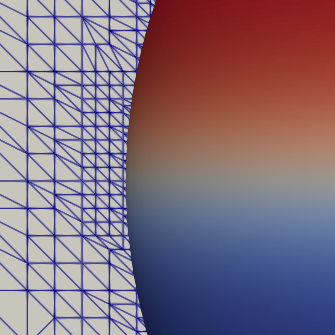}} ;
		\node at (15,-1) {\includegraphics[width=.3\linewidth]{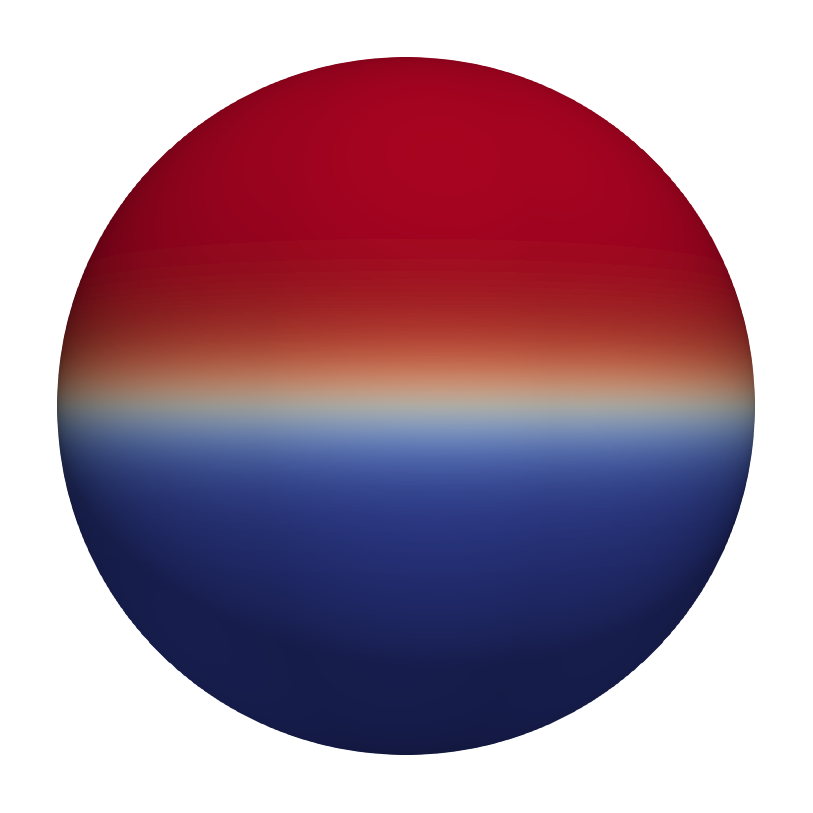}};
		\fill[opacity=0.4, color=gray] (12.55,-1.75) -- (14.05,-1.75) -- (14.05,-0.25) -- (12.55,-0.25) --cycle;
		\draw[color=green, line width=1.5] (12.55,-1.75) -- (14.05,-1.75) -- (14.05,-0.25) -- (12.55,-0.25) --cycle;
		\draw[color=green, line width=1.5] (12.55,-1.75) -- (7.5,-1.55);
		\draw[color=green, line width=1.5] (14.05,-0.25) -- (10.5, 1.55);
		\draw[color=green, line width=1.5] (7.5,1.55) -- (10.5, 1.55) -- (10.5, -1.55)--(7.5, -1.55) --cycle;		
	\end{tikzpicture}
	\caption {Approximation of exact solution \eqref{eq:exact} with $\epsilon = 0.05$
computed with mesh level $\ell = 6$ and a magnified view of the interface thickness 
with the bulk mesh near the surface.}\label{fig:conv_test_sol}
\end{figure}
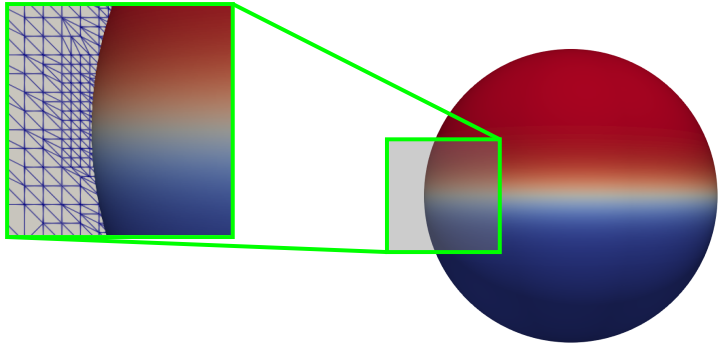

%



Fig.~\ref{fig:L2errors} shows the evolution of the $L_2$ errors of $c$ computed with the SAV-BDF1 and SAV-BDF2 methods
for $\epsilon = 0.05, 0.1, 1$. We used $\mathbb{P}^1$ elements and for each panel in Fig.~\ref{fig:L2errors}
we report the $L_2$ errors associated to four mesh refinement levels. We see that in all the cases the errors increase slightly at
the beginning of the time interval and then they tend to reach a plateau. The thinner the interface between phases is (i.e., the smaller
$\epsilon$), the faster the plateau is reached. In the case of the smallest $\epsilon$, i.e., $\epsilon = 0.05$,
Fig.~\ref{fig:mod_e} displays the evolution of modified energy \eqref{eq:mod_e_BDF1}, which is associated to the SAV-BDF1
method, and modified energy \eqref{eq:mod_e_BDF2}, which is associated to the SAV-BDF2 method, for mesh level $\ell = 5$. 
As expected, the modified energies decay in time.

\begin{figure}[htb!]
	\centering
	\begin{overpic}[width=.32\textwidth,grid=false]{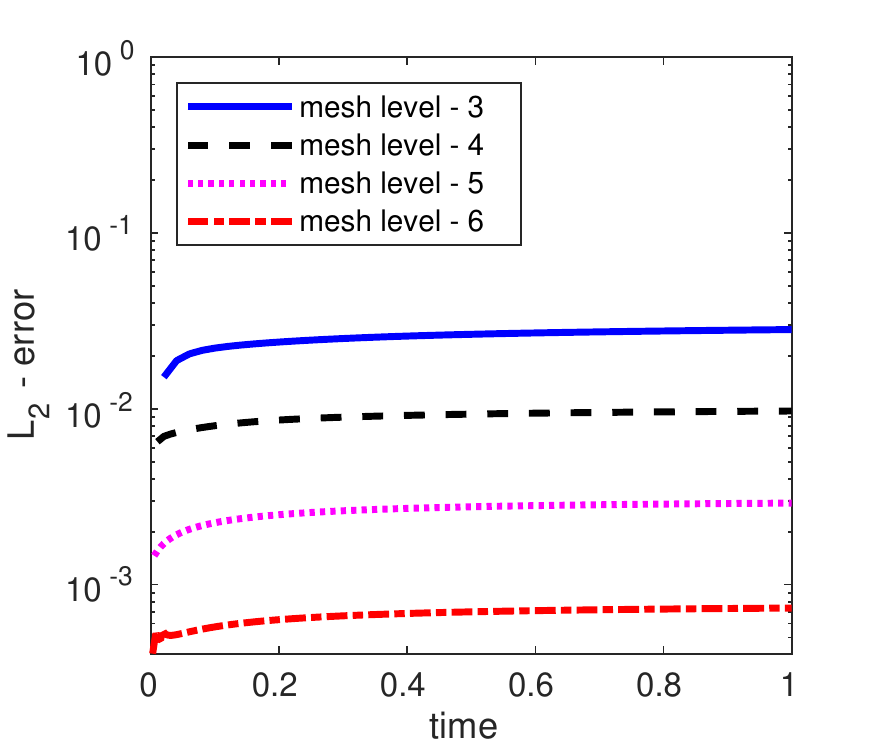}
			\put(30,82){\small{$\epsilon = 0.05$, BDF1}}
		\end{overpic}
	\begin{overpic}[width=.32\textwidth,grid=false]{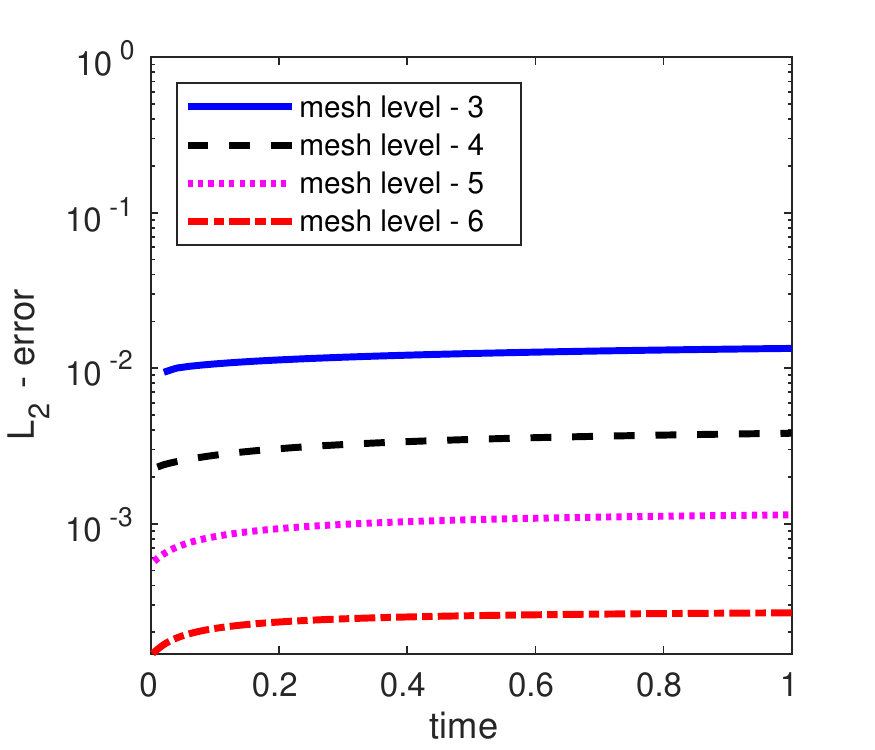}
			\put(32,82){\small{$\epsilon = 0.1$, BDF1}}
		\end{overpic}
	\begin{overpic}[width=.32\textwidth,grid=false]{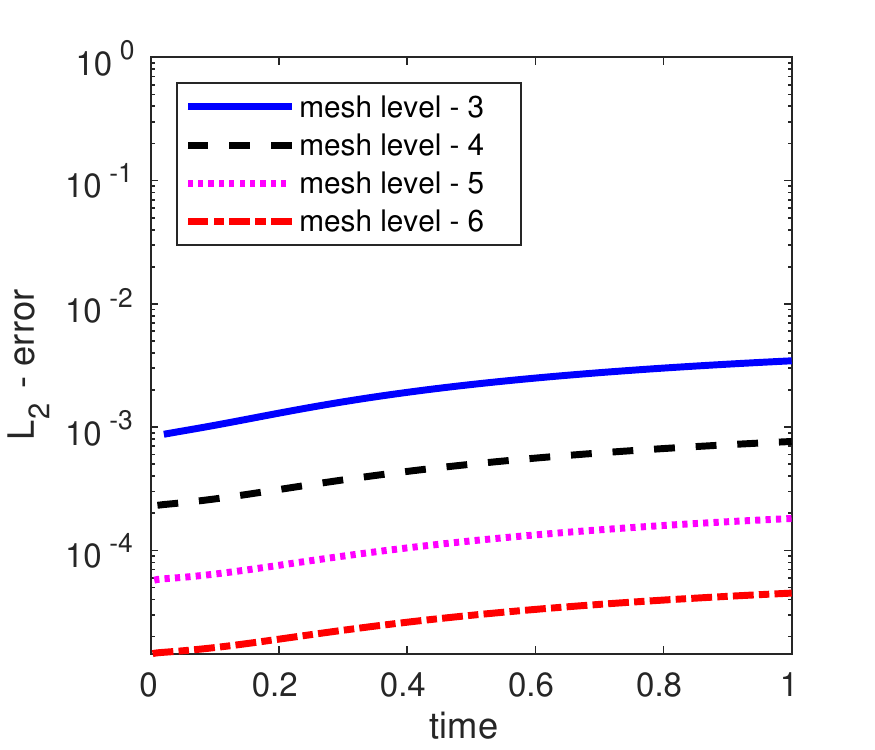} 
				\put(34,82){\small{$\epsilon = 1$, BDF1}}
		\end{overpic}
	\\
	\vskip .2cm
	\begin{overpic}[width=.32\textwidth,grid=false]{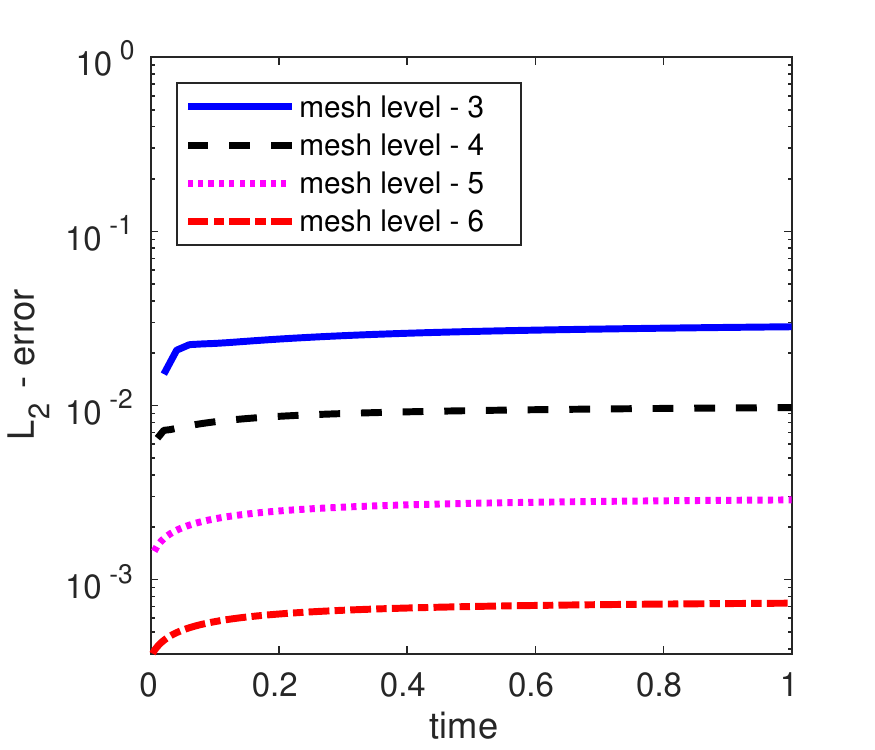}
				\put(30,82){\small{$\epsilon = 0.05$, BDF2}}
		\end{overpic}
	\begin{overpic}[width=.32\textwidth,grid=false]{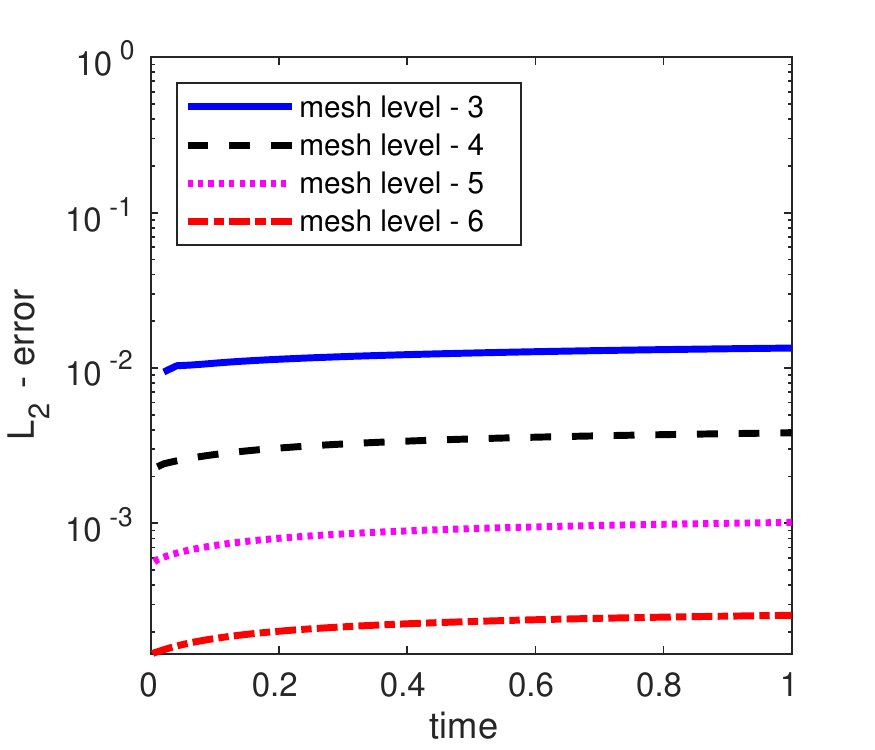}
				\put(32,82){\small{$\epsilon = 0.1$, BDF2}}
		\end{overpic}
	\begin{overpic}[width=.32\textwidth,grid=false]{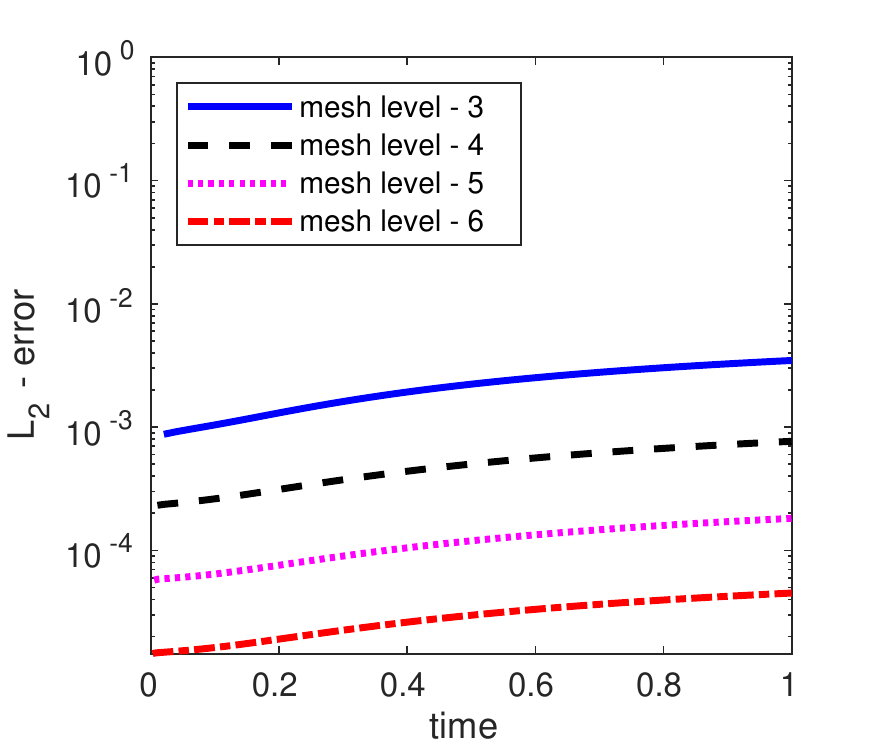}
					\put(34,82){\small{$\epsilon = 1$, BDF2}}
		\end{overpic}
	\caption{Convergence test: evolution of the $L_2$ errors of $c$ computed with the SAV-BDF1 method (top row)
	or SAV-BDF2 method (bottom row) for $\epsilon = 0.05$ (left), $\epsilon = 0.1$ (center), and $\epsilon = 1$ (right).}\label{fig:L2errors}
\end{figure}

\begin{figure}[htb!]
	\centering
	\begin{overpic}[width=.4\textwidth,grid=false]{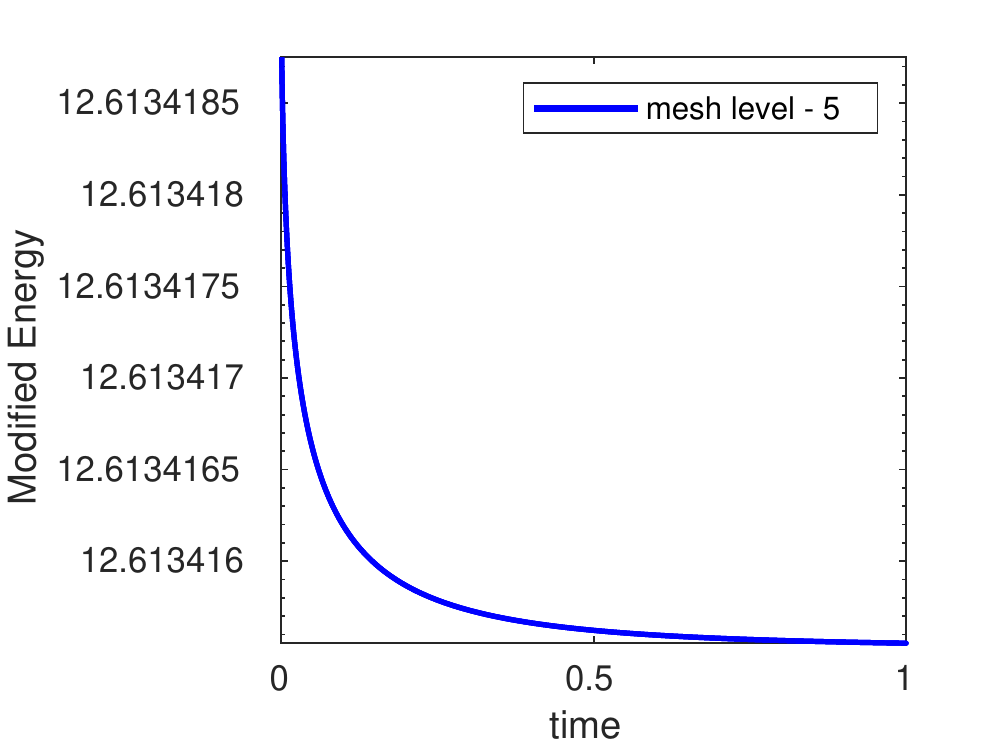}
	\put(54,72){\small{BDF1}}
		\end{overpic}
	\begin{overpic}[width=.4\textwidth,grid=false]{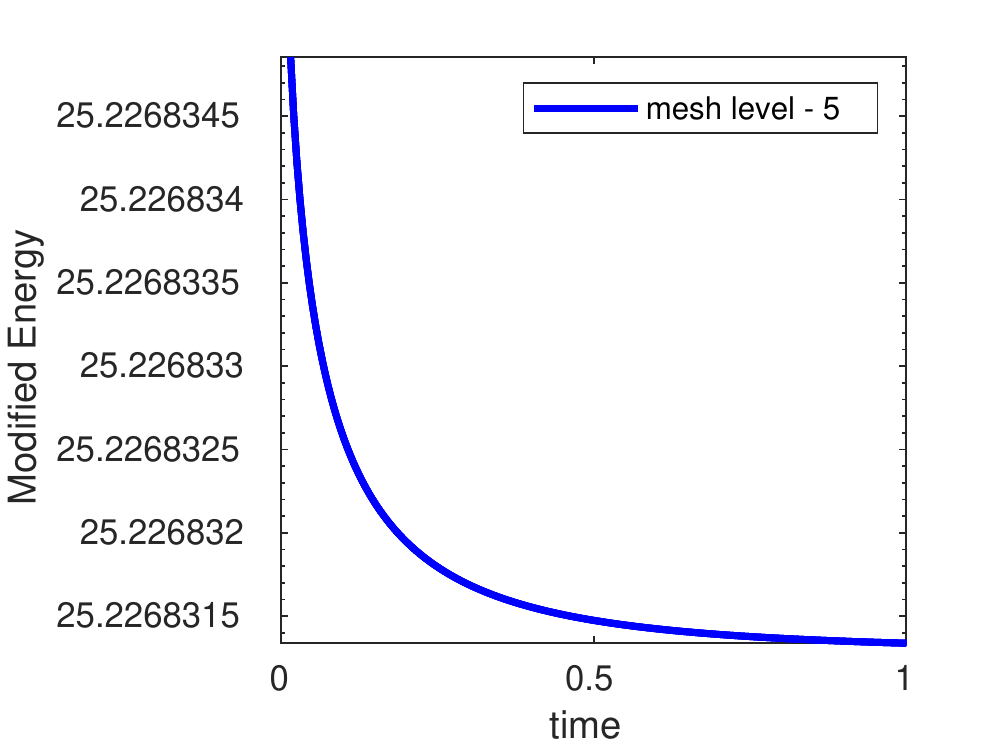}
		\put(54,72){\small{BDF2}}
		\end{overpic}
	\caption{Convergence test, $\epsilon = 0.05$: decay of modified energy  \eqref{eq:mod_e_BDF1} (left) and  \eqref{eq:mod_e_BDF2} 
	(right) for mesh level $\ell = 5$.}\label{fig:mod_e}
\end{figure}

Tables \ref{tab:1} and \ref{tab:2} report the $L_2$ errors of $c$ at the end of the time interval (i.e.,
$t = 1$) computed with the SAV-BDF1 and SAV-BDF2 method, respectively. Mesh refinement level and 
associated time steps are reported in the tables, which provide the order of convergence too. 
We see that while the $L_2$ errors are somewhat different, the order of convergence is the same.  
It is around 2, especially when going from $\ell = 5$ to $\ell = 6$, which is the optimal order of convergence
for $\mathbb{P}^1$ elements. We believe that the order of convergence is not spoilt when using BDF1 for time discretization
because the time step value is small enough to prevent the time discretization error from dominating over the space discretization
error.  
Table \ref{tab:2} can be compared with Table \ref{tab:3}, which provides $L_2$ errors of $c$ at $t = 1$ computed 
with the stabilized method in \cite{Yushutin_IJNMBE2019} and BDF2, together with the rates of convergence. Not just the convergence
rates are the same, but the errors are also very similar. We have highlighted in red
the digits in Table \ref{tab:3} that differ from Table \ref{tab:2}.

\begin{table}[htb!]
\centering
	\begin{tabular}{|c|c|c|c|c|c|c|c|}
		\hline
		& & \multicolumn{2}{c|}{$\epsilon = 0.05$} & \multicolumn{2}{c|}{$\epsilon = 0.1$} & \multicolumn{2}{c|}{$\epsilon = 1$}  \\
		\hline
		mesh level & $\Delta t$ & error & rate & error & rate & error & rate \\
		\hline
		3 & 0.02 & 2.8247 $\cdot 10^{-2}$  &  & 1.3409$\cdot 10^{-2}$ &  & 0.3453$\cdot 10^{-2}$ &  \\
		\hline
		4 & 0.01 & 0.9720 $\cdot 10^{-2}$  & 1.54 & 0.3816$\cdot 10^{-2}$ & 1.81 & 0.0765$\cdot 10^{-2}$  & 2.18 \\
		\hline
		5 & 0.005 & 0.2909$\cdot 10^{-2}$  &  1.76 & 0.1139$\cdot 10^{-2}$ & 1.91 & 0.0181$\cdot 10^{-2}$ & 2.07 \\
		\hline
		6 & 0.0025 & 0.0735$\cdot 10^{-2}$  &  1.96 & 0.0267$\cdot 10^{-2}$ & 1.98 & {0.0045$\cdot 10^{-2}$} & 2.01 \\
		\hline
	\end{tabular}\caption{Convergence test, $\epsilon = 0.05, 0.1, 1$: $L_2$ errors of $c$ at $t = 1$ computed with the SAV-BDF1 method and $\mathbb{P}^1$ elements for different meshes and time steps, together with the rates of convergence.}\label{tab:1}
\end{table}

\begin{table}[htb!]
\centering
	\begin{tabular}{|c|c|c|c|c|c|c|c|}
		\hline
		& & \multicolumn{2}{c|}{$\epsilon = 0.05$} & \multicolumn{2}{c|}{$\epsilon = 0.1$} & \multicolumn{2}{c|}{$\epsilon = 1$}  \\
		\hline
		mesh level & $\Delta t$ & error & rate & error & rate & error & rate \\
		\hline
		3 & 0.02 & 2.8338 $\cdot 10^{-2}$  &  & 1.3438$\cdot 10^{-2}$ &  & 0.3474$\cdot 10^{-2}$ &  \\
		\hline
		4 & 0.01 & 0.9727 $\cdot 10^{-2}$  & 1.54 & 0.3824$\cdot 10^{-2}$ & 1.81 & 0.0767$\cdot 10^{-2}$  & 2.18 \\
		\hline
		5 & 0.005 & 0.2869$\cdot 10^{-2}$  &  1.76 & 0.1013$\cdot 10^{-2}$ & 1.91 & 0.0181$\cdot 10^{-2}$ & 2.07 \\
		\hline
		6 & 0.0025 & 0.0732$\cdot 10^{-2}$  &  1.96 & 0.0255$\cdot 10^{-2}$ & 1.98 & 0.0045$\cdot 10^{-2}$ & 2.01 \\
		\hline
	\end{tabular}\caption{Convergence test, $\epsilon = 0.05, 0.1, 1$: $L_2$ errors of $c$ at $t = 1$ computed with the SAV-BDF1 method and $\mathbb{P}^1$ elements for different meshes and time steps, together with the rates of convergence.}\label{tab:2}
\end{table}

\begin{table}[htb!]
\centering
\begin{tabular}{|c|c|c|c|c|c|c|c|}
	\hline
	& & \multicolumn{2}{c|}{$\epsilon = 0.05$} & \multicolumn{2}{c|}{$\epsilon = 0.1$} & \multicolumn{2}{c|}{$\epsilon = 1$}  \\
	\hline
	mesh level & $\Delta t$ & error & rate & error & rate & error & rate \\
	\hline
	3 & 0.02 & 2.833{\color{red}5}$\cdot 10^{-2}$ &  & 1.343{\color{red}4}$\cdot 10^{-2}$ &  & 0.347{\color{red}5} $\cdot 10^{-2}$ &  \\
	\hline
	4 & 0.01 & 0.972{\color{red}5}$\cdot 10^{-2}$  & 1.54 & 0.382{\color{red}3}$\cdot 10^{-2}$ & 1.81 & 0.0767$\cdot 10^{-2}$  & 2.18 \\
	\hline
	5 & 0.005 & 0.2869$\cdot 10^{-2}$  & 1.76  & 0.1013$\cdot 10^{-2}$ & 1.91 & 0.018{\color{red}2} $\cdot 10^{-2}$ & 2.07 \\
	\hline
	6 & 0.0025 & 0.0732$\cdot 10^{-2}$  & 1.96  & 0.0255$\cdot 10^{-2}$ & 1.98 & 0.0045$\cdot 10^{-2}$ &  2.01\\ 
	\hline
\end{tabular}\caption{Convergence test, $\epsilon = 0.05, 0.1, 1$: $L_2$ errors of $c$ at $t = 1$ computed with the stabilized method in \cite{Yushutin_IJNMBE2019}, $\mathbb{P}^1$ elements, and BDF2 for different meshes and time steps, together with the rates of convergence.}\label{tab:3}
\end{table}

The results in this section give us confidence in our implementation of the SAV methods within DROPS. In addition,
they suggest that for the values of $\epsilon$ we consider $\ell = 5$ and $\Delta t = 0.005$ are appropriate levels of 
refinement for mesh size and time step
as they provide small discretization errors and are more computationally efficient than $\ell = 6$ and $\Delta t = 0.0025$. 
Hence, for the results in the next section we will use $\ell = 5$ and $\Delta t = 0.005$.



\subsection{Phase separation on the sphere}\label{sec:res_spehre}

Our interest in surface phase field problems, such as the Cahn--Hilliard equation \cite{Palzhanov2021, WANG2022183898, Yushutin_IJNMBE2019, Yushutin2019, zhiliakov2021experimental}, stems from their practical applications in targeted drug delivery. The phenomenon of lipid phase separation has been utilized to enhance the delivery performance of targeted lipid vesicles \cite{Bandekar2013, KARVE20104409}, as the formation of phase-separated patterns on the vesicle surface has been associated with increased target selectivity, cell uptake, and overall efficacy.
In our previous works \cite{WANG2022183898, zhiliakov2021experimental}, we validated our numerical results obtained using the approaches described in \cite{Palzhanov2021, Yushutin_IJNMBE2019} against laboratory experiments. We achieved good agreement between the numerical and experimental results for different lipid membrane compositions. 

In this paper, we consider 3 membrane compositions. Each membrane composition corresponds to a certain fraction $a$ 
of the sphere surface area (since these vesicles are spherical) covered by one representative phase. In this section, 
we present results for $a = 0.5, 0.3, 0.7$, which are experimentally relevant values.

In order to model an initially homogenous mix of components, 
the initial composition $c_0$ is defined  as a realization of Bernoulli random variable~$c_\text{rand} \sim \text{Bernoulli}(a)$
with mean value $a$, i.e. we set:
\begin{equation}\label{raftIC}
	c_0 \coloneqq c_\text{rand}(\bx)\quad\text{for active mesh nodes $\bx$}.
\end{equation}
As mentioned at the end of the previous section, the interface thickness $\epsilon$ is set to 0.05, which is a realistic value for lipid vesicles.

Let us start with the results obtained with the SAV-BDF2 method without time step adaptivity 
and compare them with the results obtained with the stabilized method in \cite{Yushutin_IJNMBE2019}.
Fig.~\ref{fig:a_50} shows the evolution of phases for $a=0.5$, which means that 50\% of the sphere
surface is covered by the representative phase (red in the figure) and the remaining 50\% is covered by 
the other phases (blue in the figure). There is no observable difference in the spinodal decomposition
and subsequent domain ripening given by the two methods. 

\begin{figure}[htb!]
	\begin{center}
		\begin{overpic}[width=.13\textwidth,grid=false]{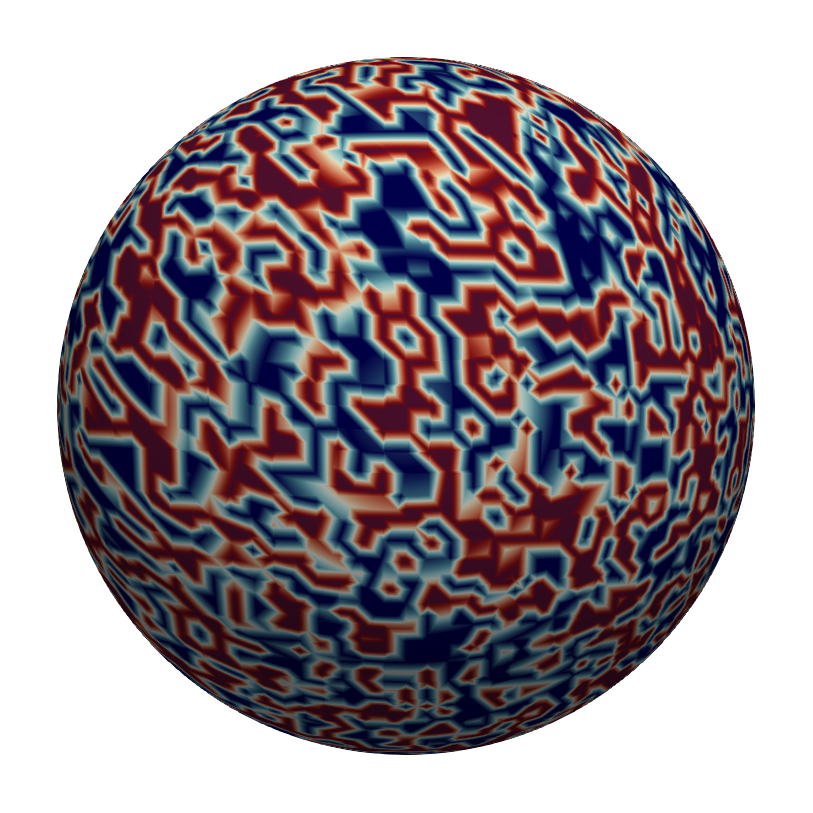}
			\put(30,102){\small{$t = 0$}}
			\put(-70,45){\small{Stabilized}}
		\end{overpic}
		\begin{overpic}[width=.13\textwidth,grid=false]{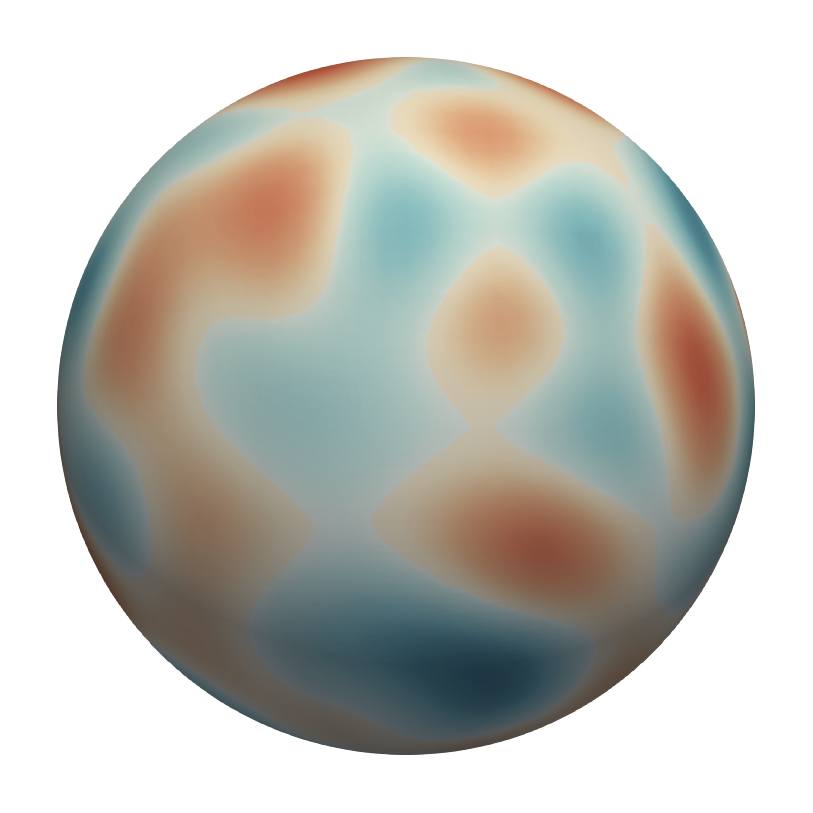}
			\put(30,100){\small{$t = 0.01$}}
		\end{overpic}
		\begin{overpic}[width=.13\textwidth,grid=false]{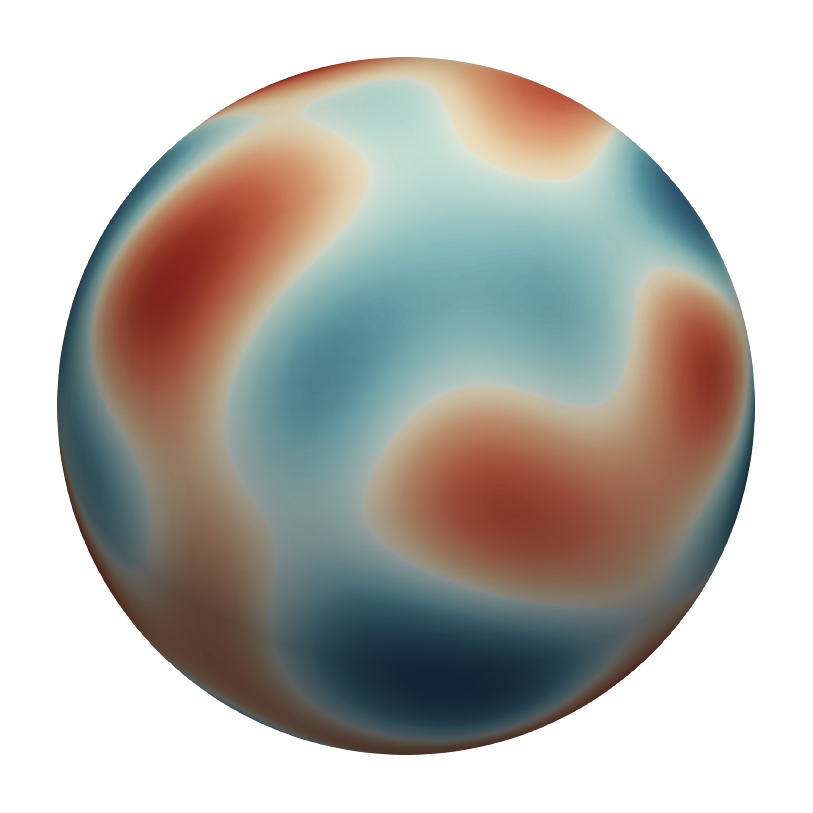}
			\put(30,102){\small{$t = 0.5$}}
		\end{overpic}
		\begin{overpic}[width=.13\textwidth,grid=false]{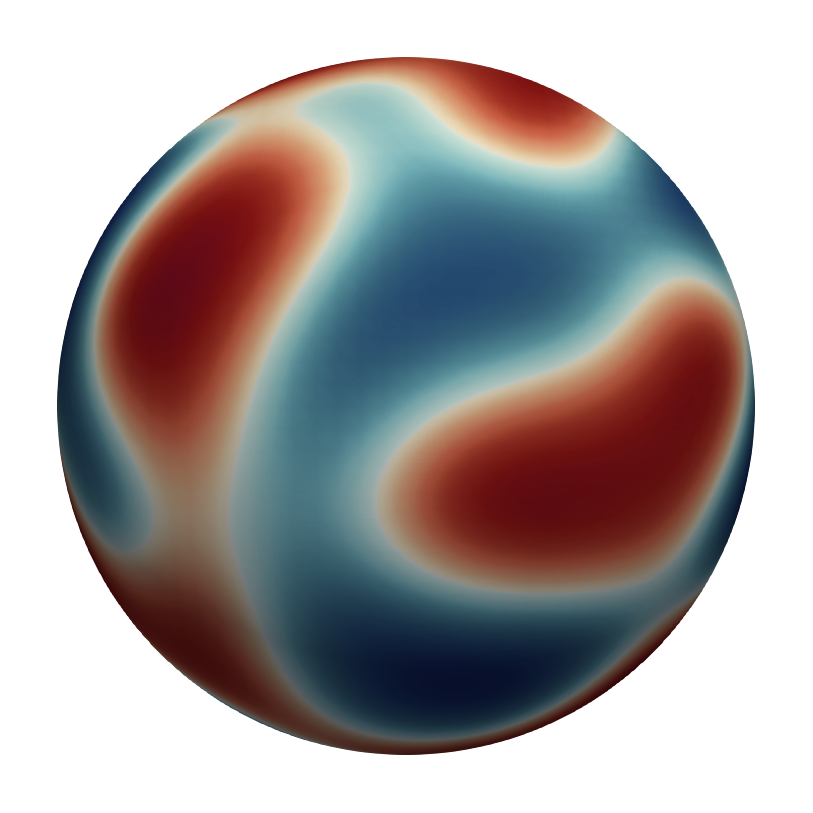}
			\put(30,102){\small{$t = 2$}}
		\end{overpic}
		\begin{overpic}[width=.13\textwidth,grid=false]{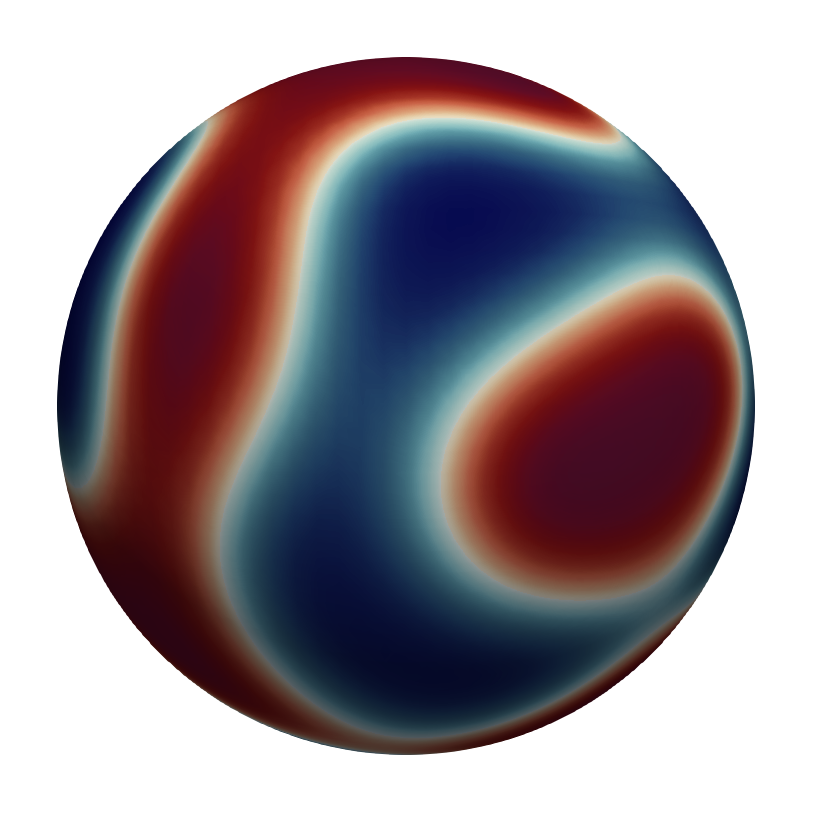}
			\put(30,102){\small{$t = 5$}}
		\end{overpic}
		\begin{overpic}[width=.13\textwidth,grid=false]{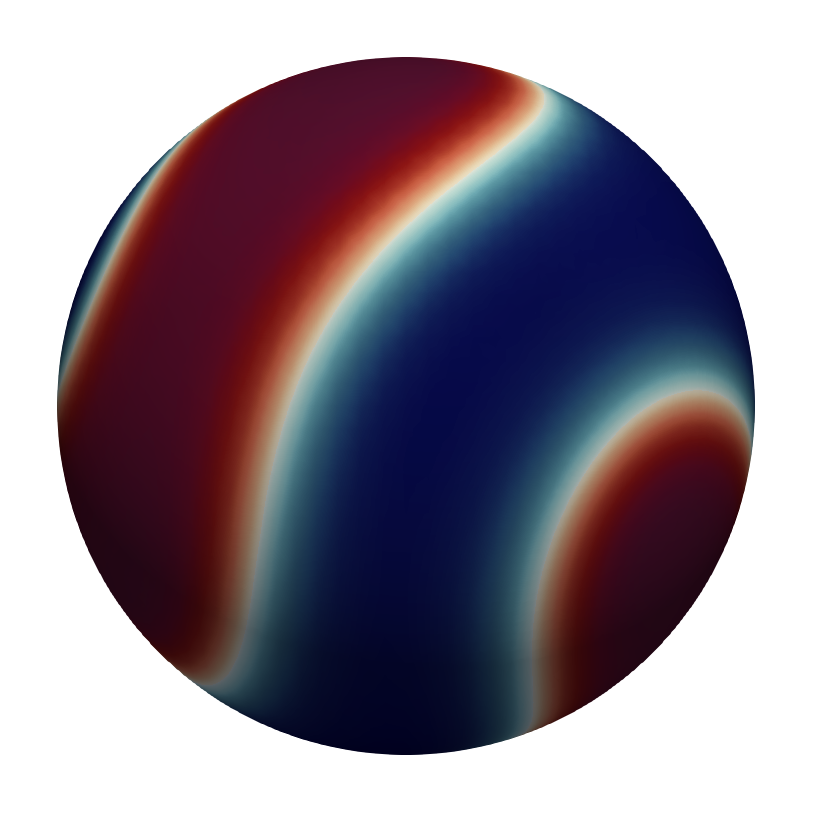}
			\put(30,102){\small{$t = 50$}}
		\end{overpic}\\
		\begin{overpic}[width=.13\textwidth,grid=false]{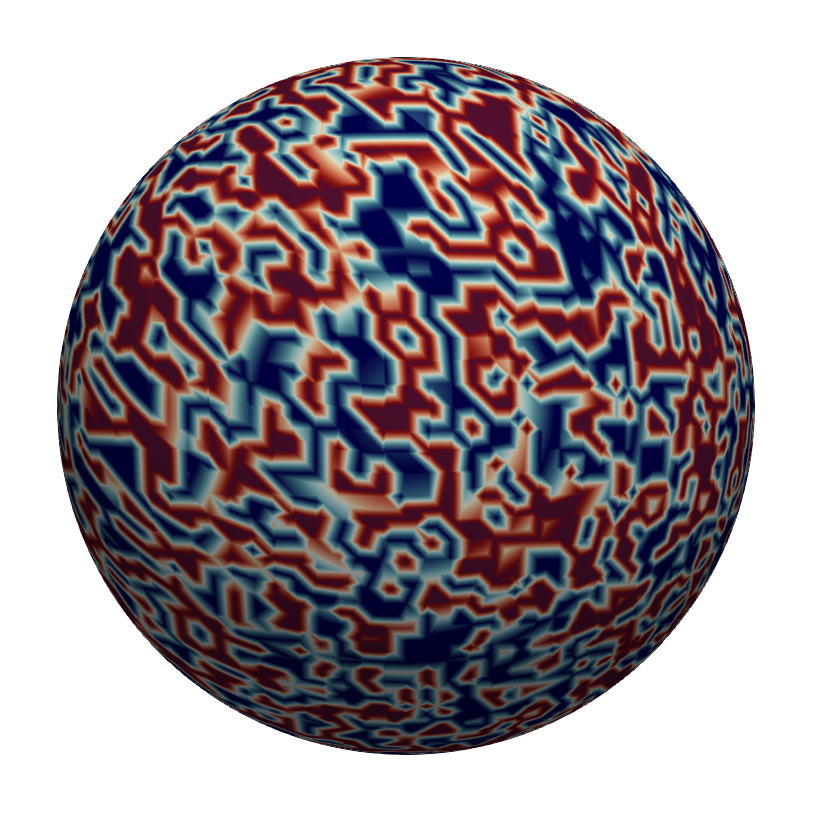}
			\put(-70,35){\small{BDF2}}
			\put(-70,55){\small{SAV}}
		\end{overpic}
		\includegraphics[width=0.13\columnwidth]{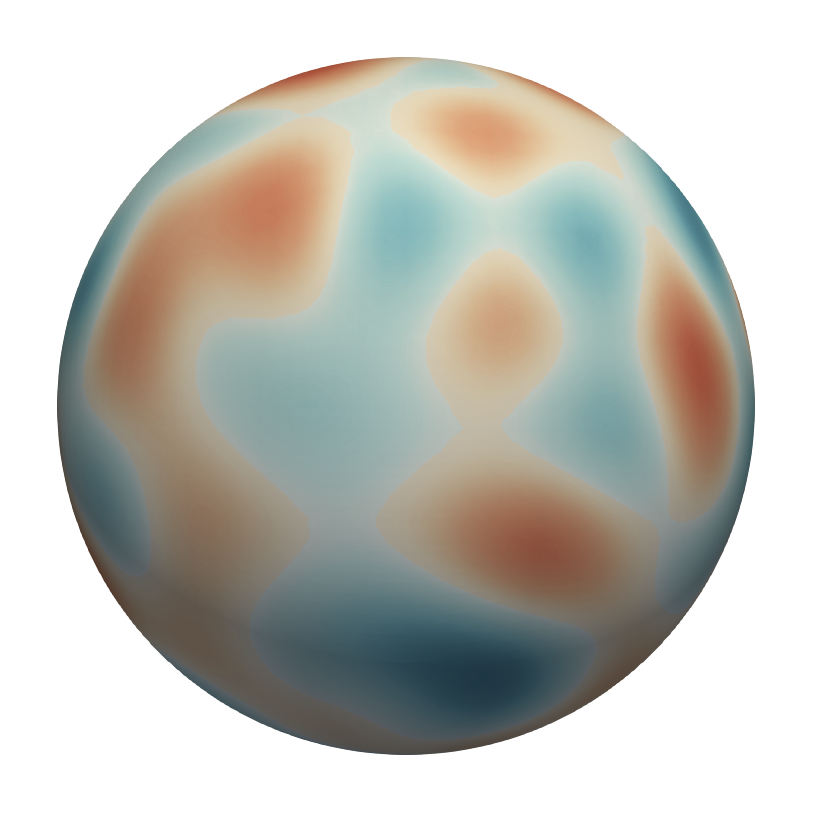}
		\includegraphics[width=0.13\columnwidth]{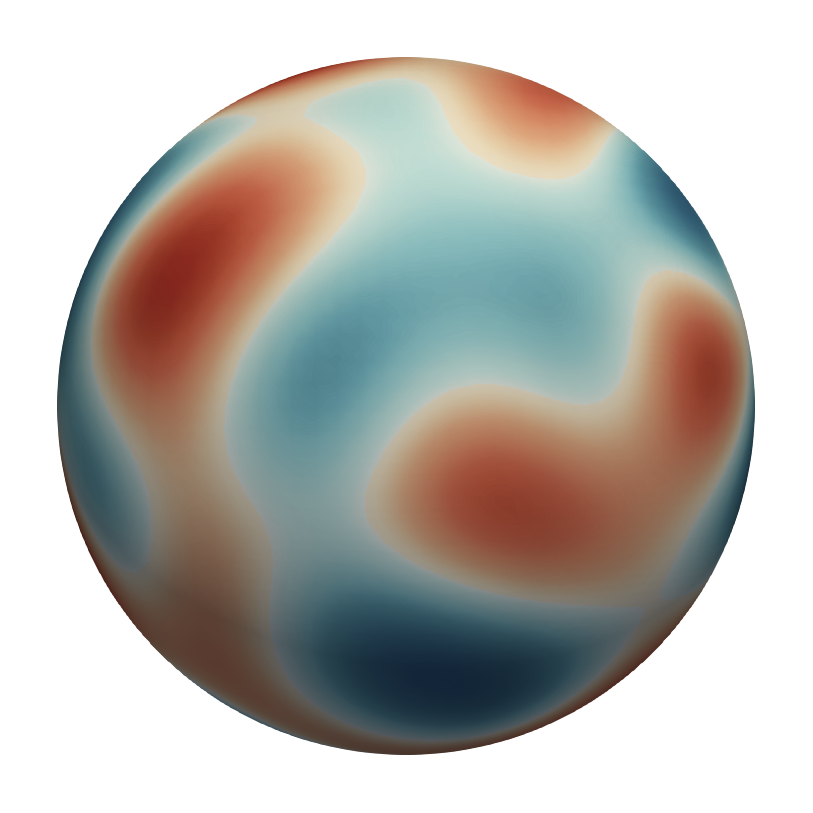}
		\includegraphics[width=0.13\columnwidth]{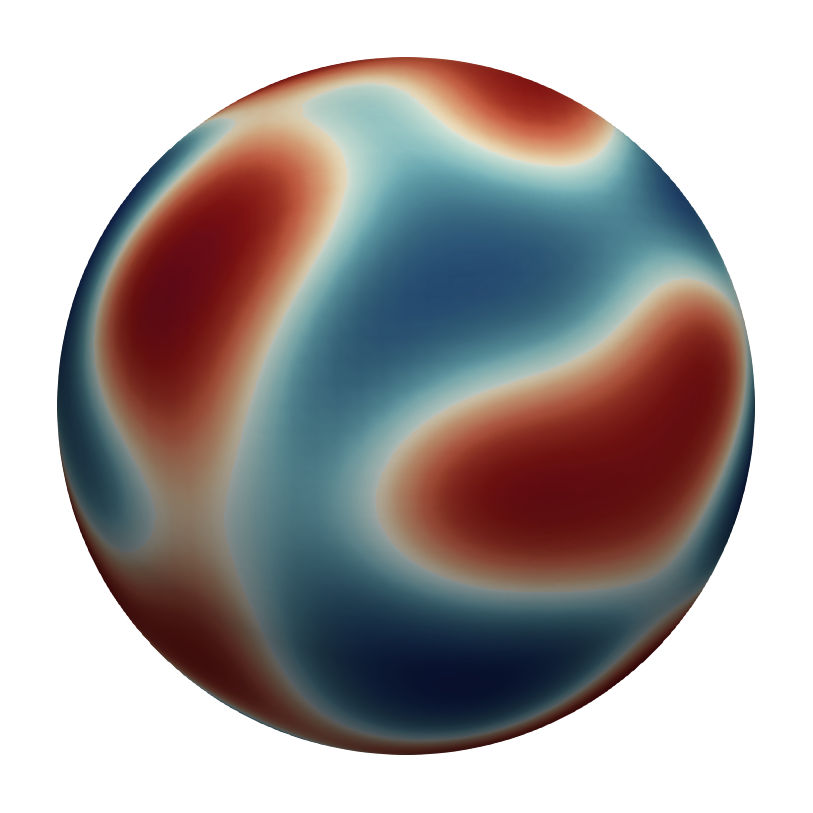}
		\includegraphics[width=0.13\columnwidth]{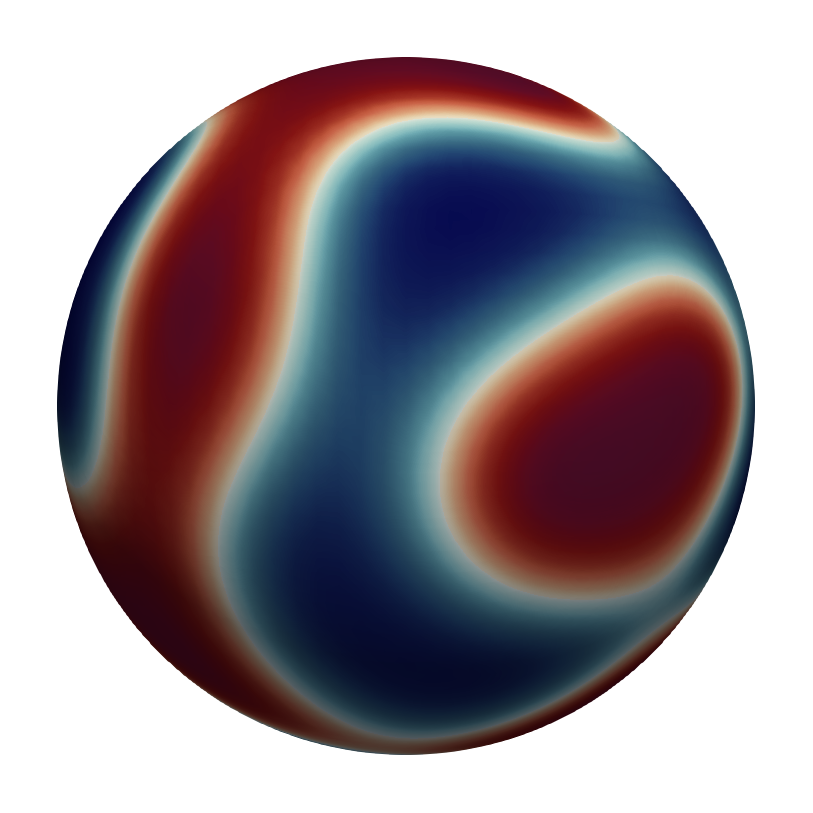}
		\includegraphics[width=0.13\columnwidth]{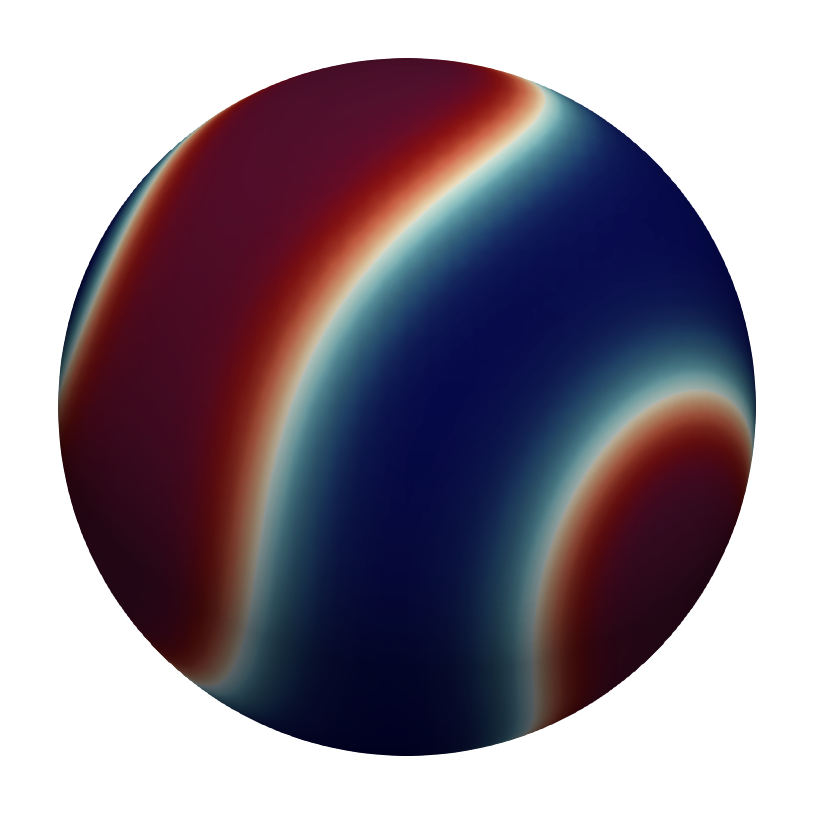}
		\\
		\includegraphics[width=0.4\columnwidth]{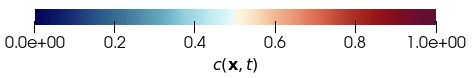}
		
	\end{center}
	\caption{Phase separation on the sphere, $a=0.5$: evolution of phases computed with the stabilized method in \cite{Yushutin_IJNMBE2019} (top) 
	and the SAV-BDF2 method without time step adaptivity (bottom).}
		\label{fig:a_50}
\end{figure}

Fig.~\ref{fig:a_30} and \ref{fig:a_70} display the evolution of phases for $a=0.3$ and $a=0.7$, respectively. 
Notice that there are opposite cases: 30\% of the sphere
surface is covered by the representative (red) phase for $a=0.3$, while
30\% of the sphere surface is covered by the opposite (blue) phase
for $a=0.7$.
If we were to use opposite initial conditions in these two cases, Fig.~\ref{fig:a_30} and \ref{fig:a_70} would
look identical just with inverted colors (red to blue and viceversa). However, the initial
conditions were generated randomly according to \eqref{raftIC} and so the evolution of the
red domains in Fig.~\ref{fig:a_30} looks similar (not identical) to the evolution of the 
blue domains in  Fig.~\ref{fig:a_70}. For both values of $a$ though, we see that 
again there is no observable difference in the solution computed with the 
stabilized method in \cite{Yushutin_IJNMBE2019} and the solution give by 
the SAV-BDF2 method without time step adaptivity.

\begin{figure}[htb!]
	\begin{center}
		\begin{overpic}[width=.13\textwidth,grid=false]{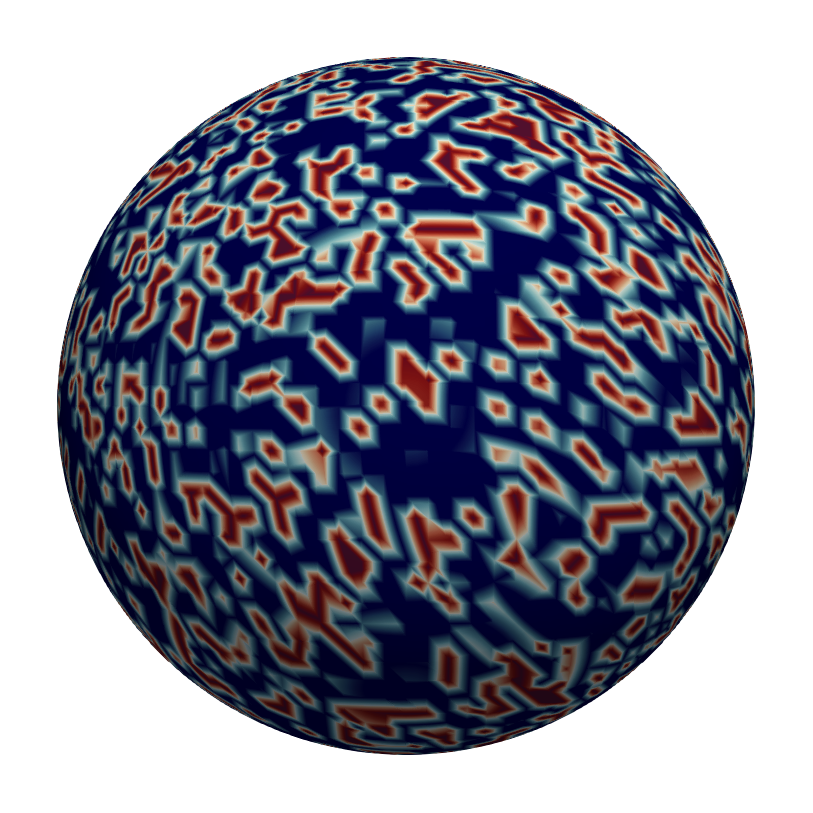}
			\put(30,102){\small{$t = 0$}}
			\put(-70,45){\small{Stabilized}}
		\end{overpic}
		\begin{overpic}[width=.13\textwidth,grid=false]{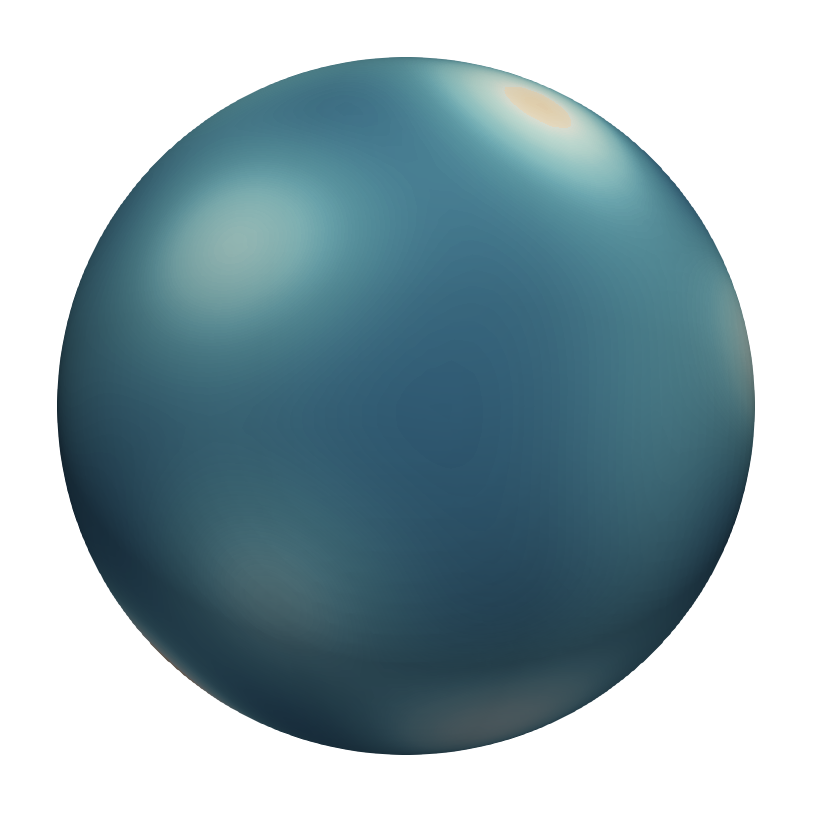}
			\put(30,102){\small{$t = 1$}}
		\end{overpic}
		\begin{overpic}[width=.13\textwidth,grid=false]{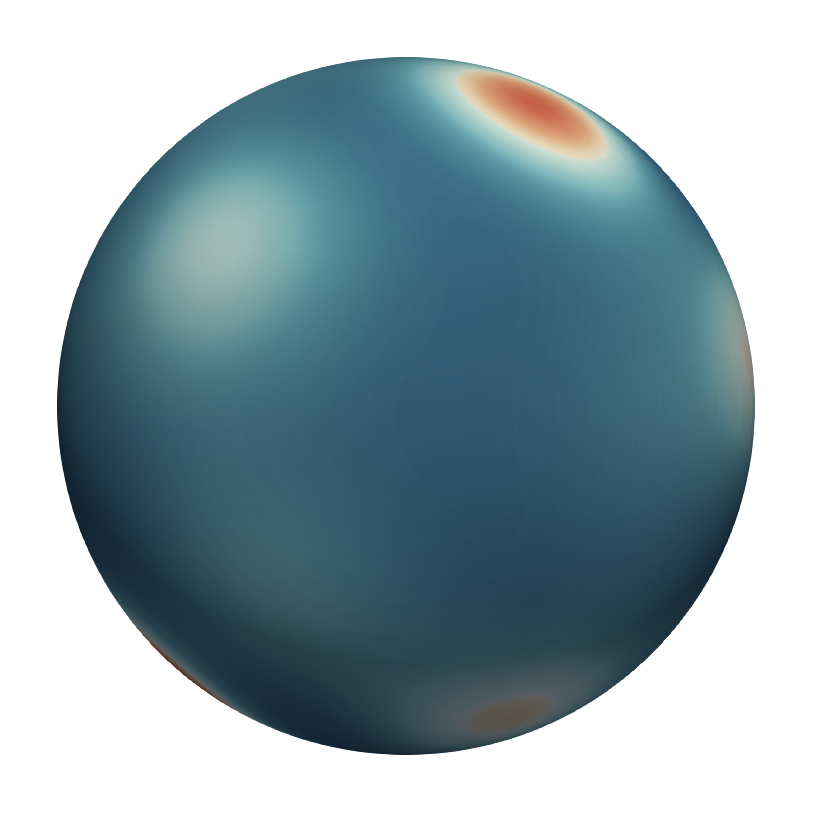}
			\put(30,102){\small{$t = 2$}}
		\end{overpic}
		\begin{overpic}[width=.13\textwidth,grid=false]{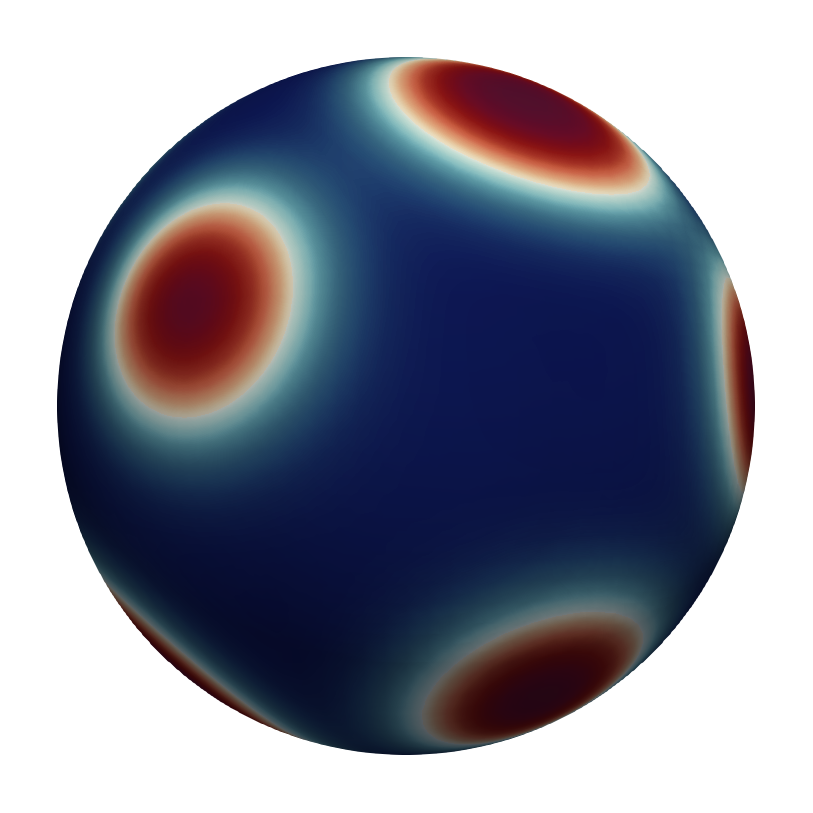}
			\put(30,102){\small{$t = 10$}}
		\end{overpic}
		\begin{overpic}[width=.13\textwidth,grid=false]{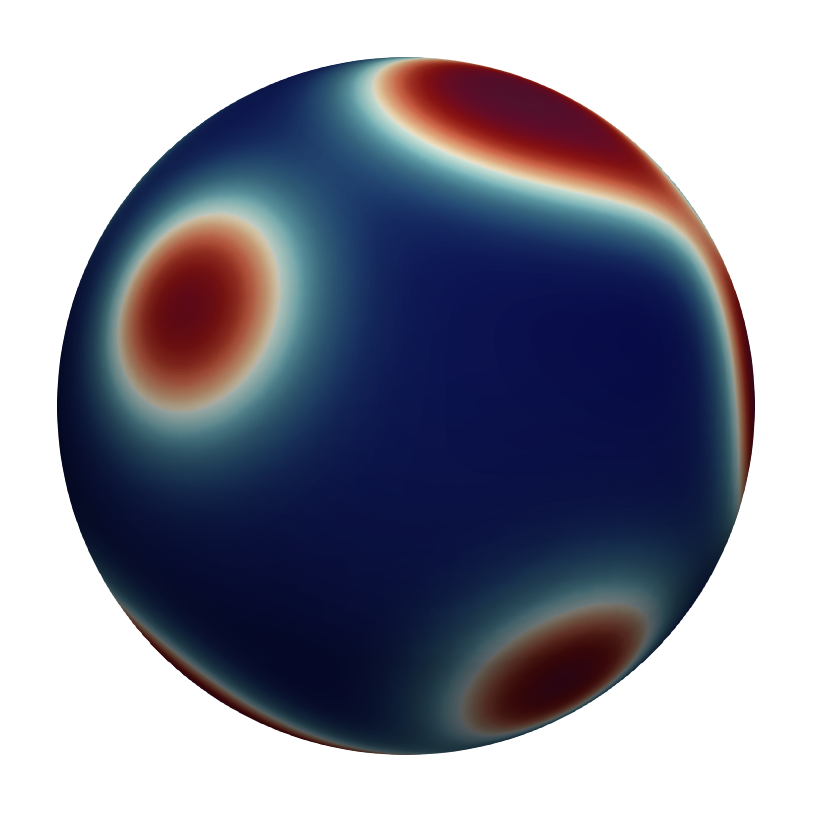}
			\put(30,102){\small{$t = 20$}}
		\end{overpic}
		\begin{overpic}[width=.13\textwidth,grid=false]{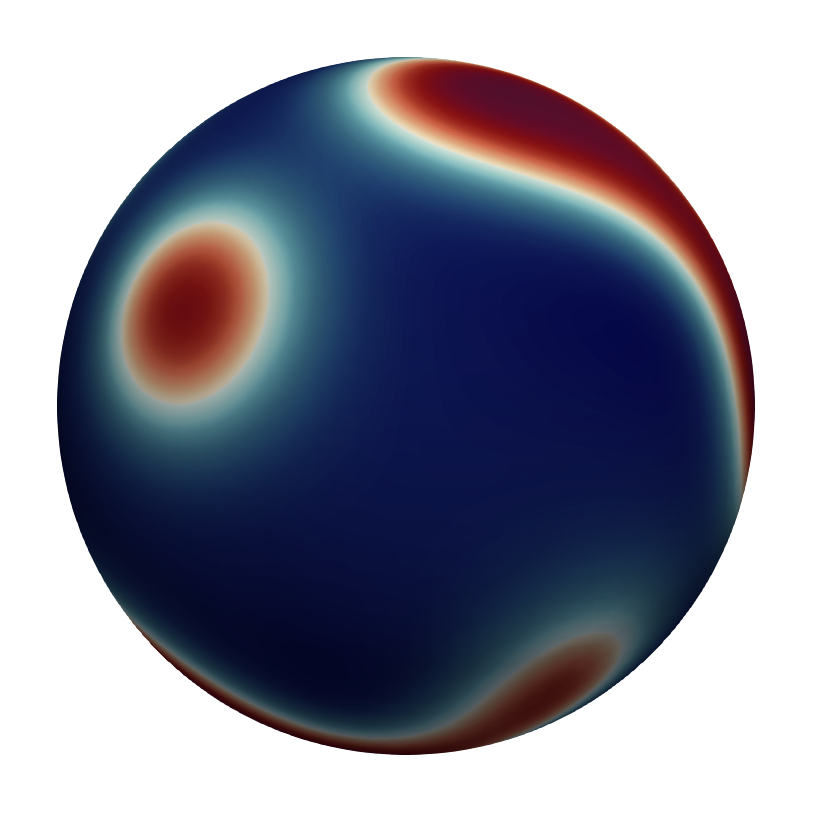}
			\put(30,102){\small{$t = 25$}}
		\end{overpic}\\
		\begin{overpic}[width=.13\textwidth,grid=false]{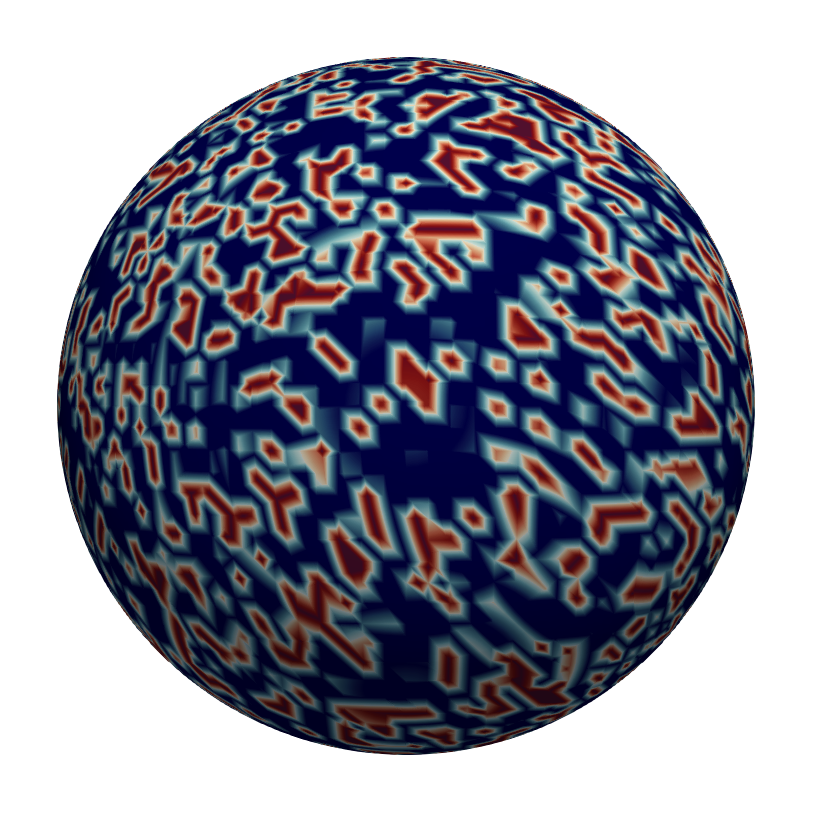}
			\put(-70,35){\small{BDF2}}
			\put(-70,55){\small{SAV}}
		\end{overpic}
		\includegraphics[width=0.13\columnwidth]{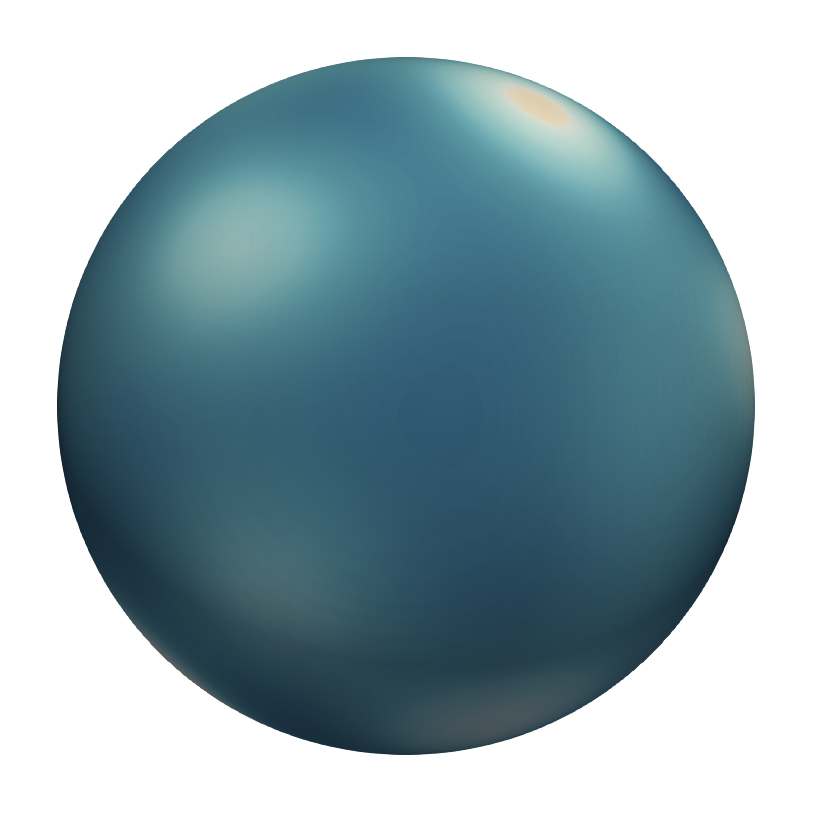}
		\includegraphics[width=0.13\columnwidth]{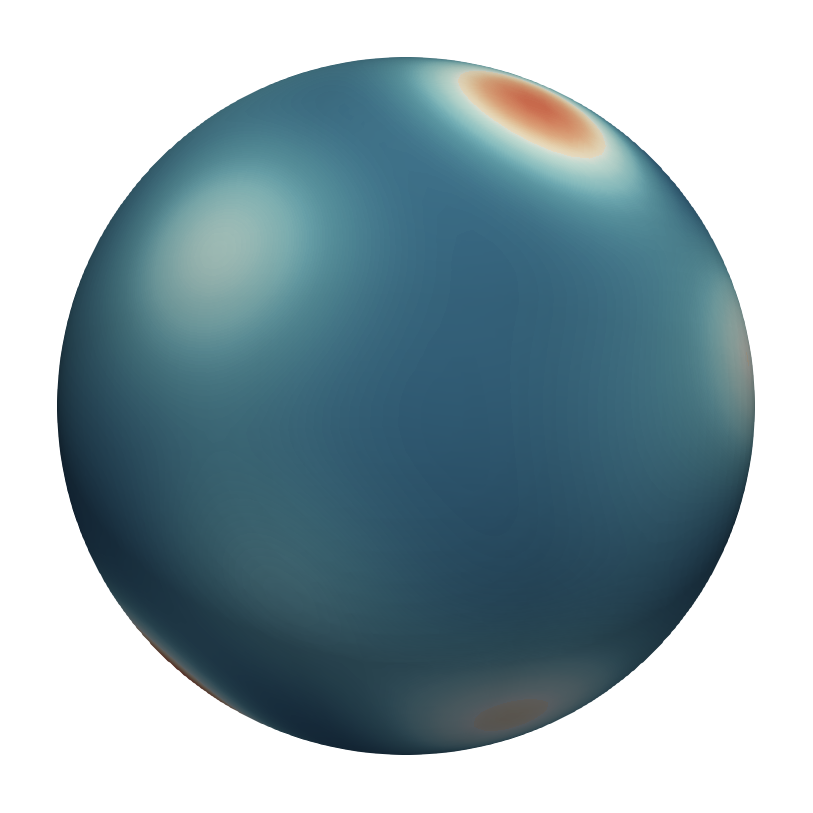}
		\includegraphics[width=0.13\columnwidth]{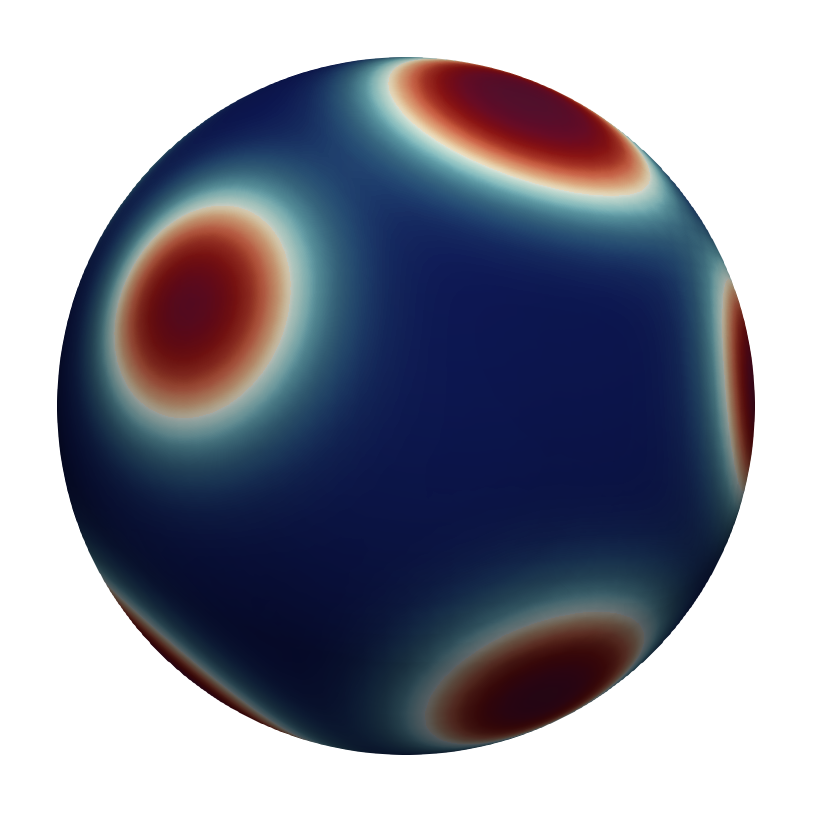}
		\includegraphics[width=0.13\columnwidth]{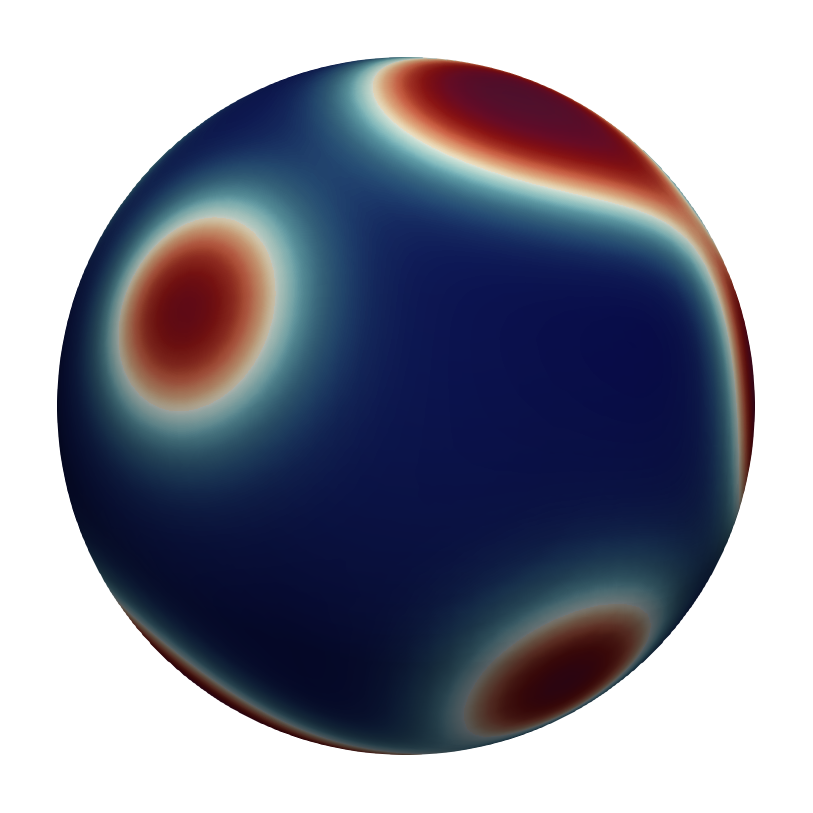}
		\includegraphics[width=0.13\columnwidth]{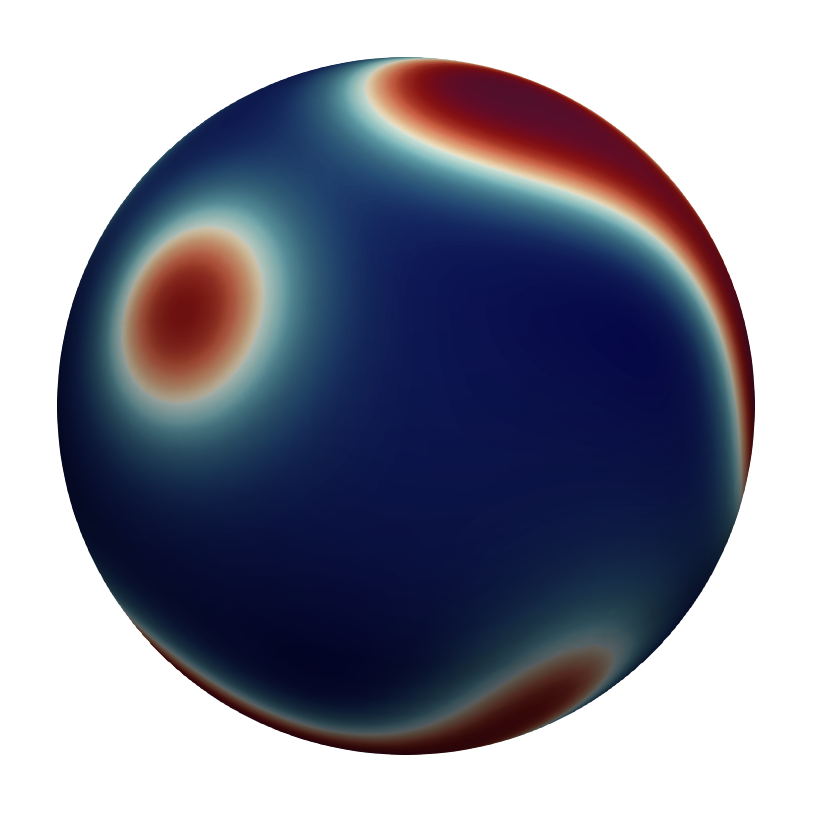}
		\\
		\includegraphics[width=0.4\columnwidth]{./img/colorbar.png}
		
	\end{center}
	\caption{Phase separation on the sphere, $a=0.3$: evolution of phases computed with the stabilized method in \cite{Yushutin_IJNMBE2019} (top) 
	and the SAV-BDF2 method without time step adaptivity (bottom).}
		\label{fig:a_30}
\end{figure}

\begin{figure}[htb!]
	\begin{center}
		\begin{overpic}[width=.13\textwidth,grid=false]{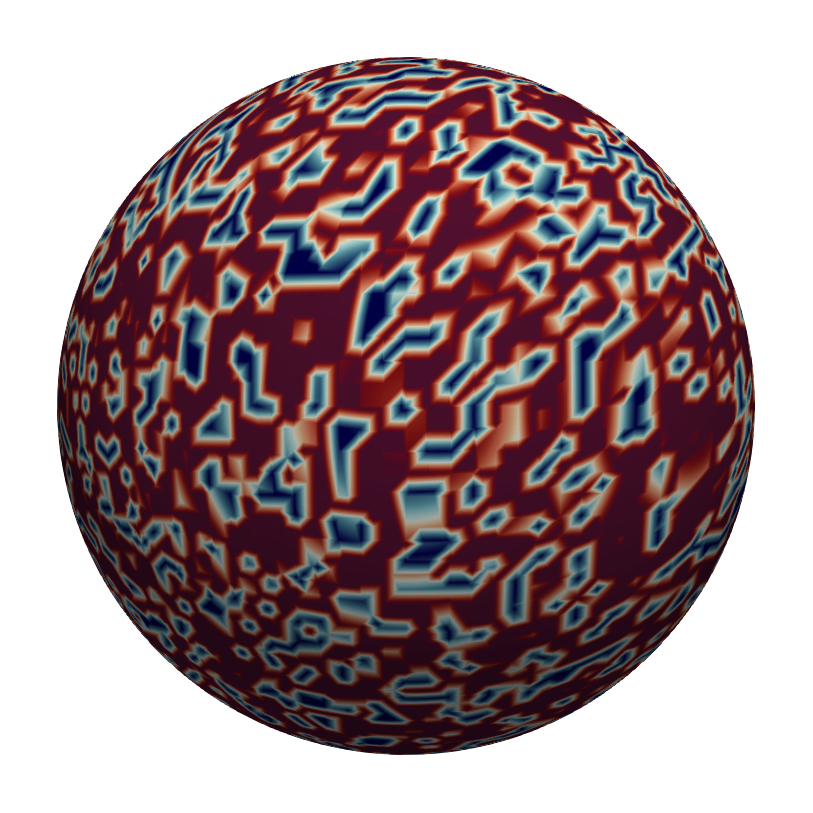}
			\put(30,102){\small{$t = 0$}}
			\put(-70,45){\small{Stabilized}}
		\end{overpic}
		\begin{overpic}[width=.13\textwidth,grid=false]{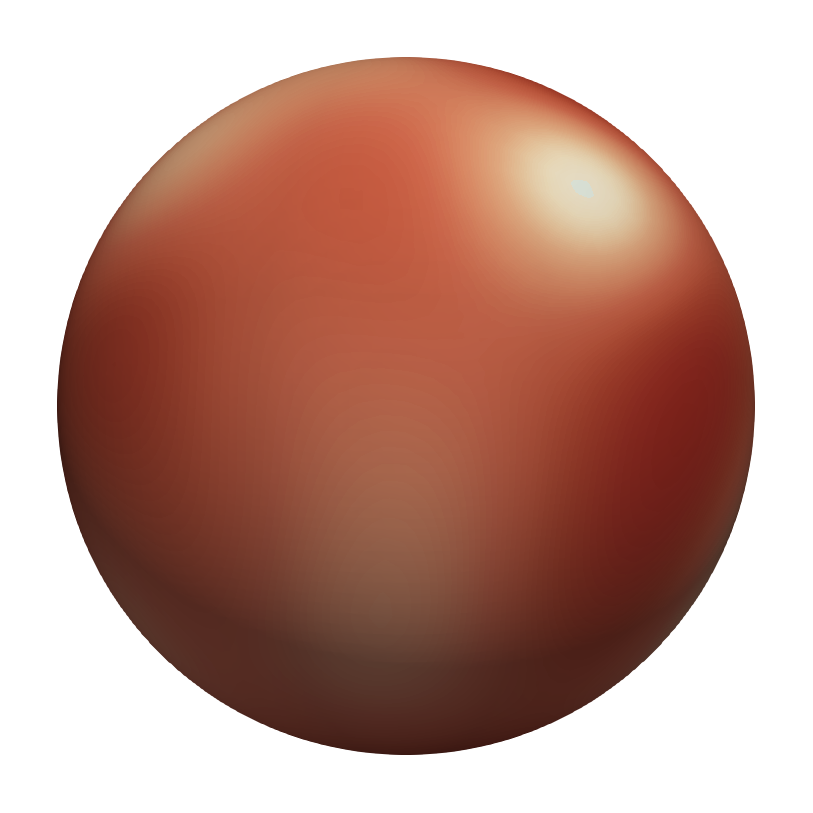}
			\put(30,102){\small{$t = 1$}}
		\end{overpic}
		\begin{overpic}[width=.13\textwidth,grid=false]{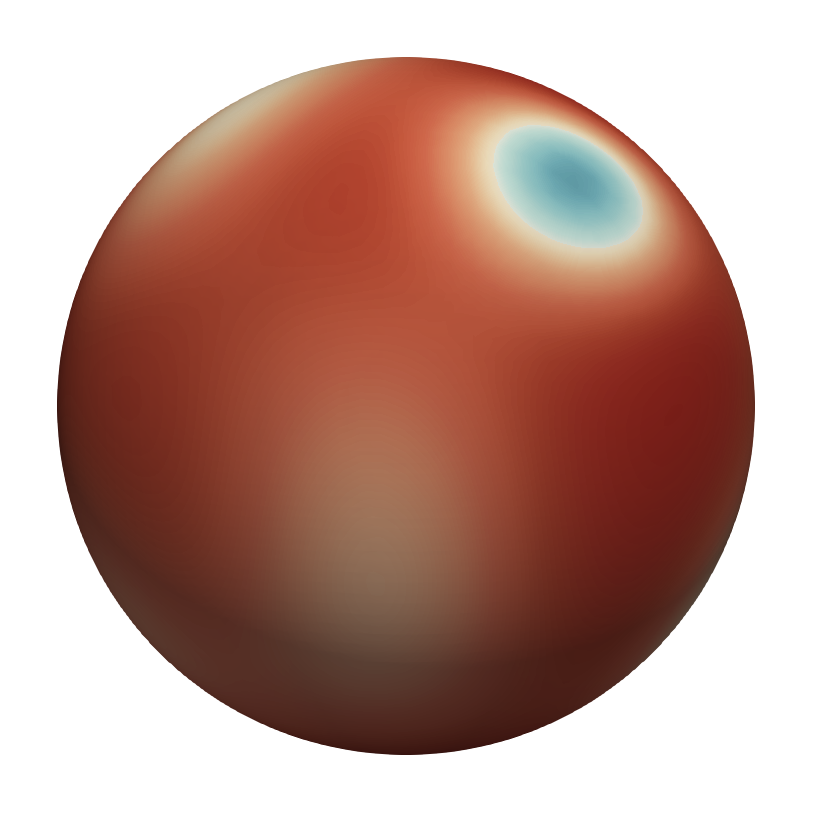}
			\put(30,102){\small{$t = 2$}}
		\end{overpic}
		\begin{overpic}[width=.13\textwidth,grid=false]{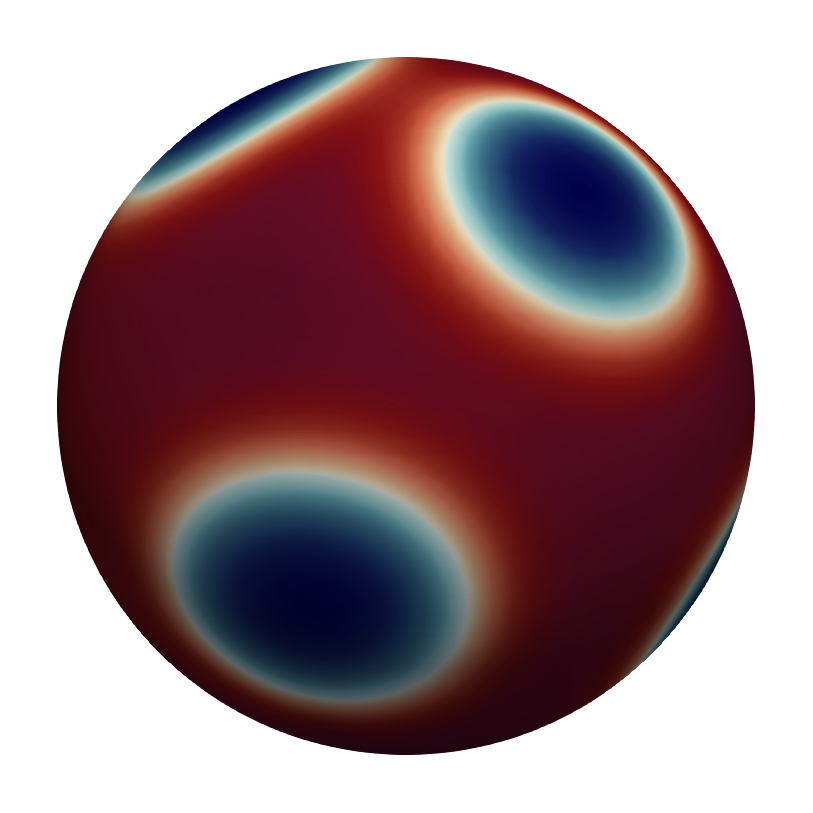}
			\put(30,102){\small{$t = 10$}}
		\end{overpic}
		\begin{overpic}[width=.13\textwidth,grid=false]{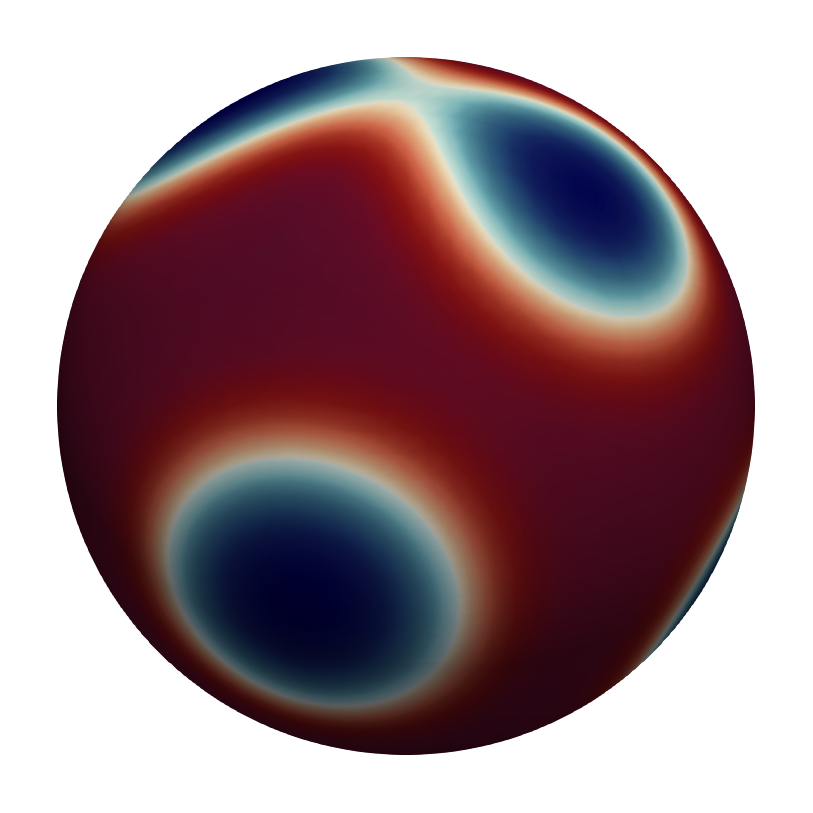}
			\put(30,102){\small{$t = 20$}}
		\end{overpic}
		\begin{overpic}[width=.13\textwidth,grid=false]{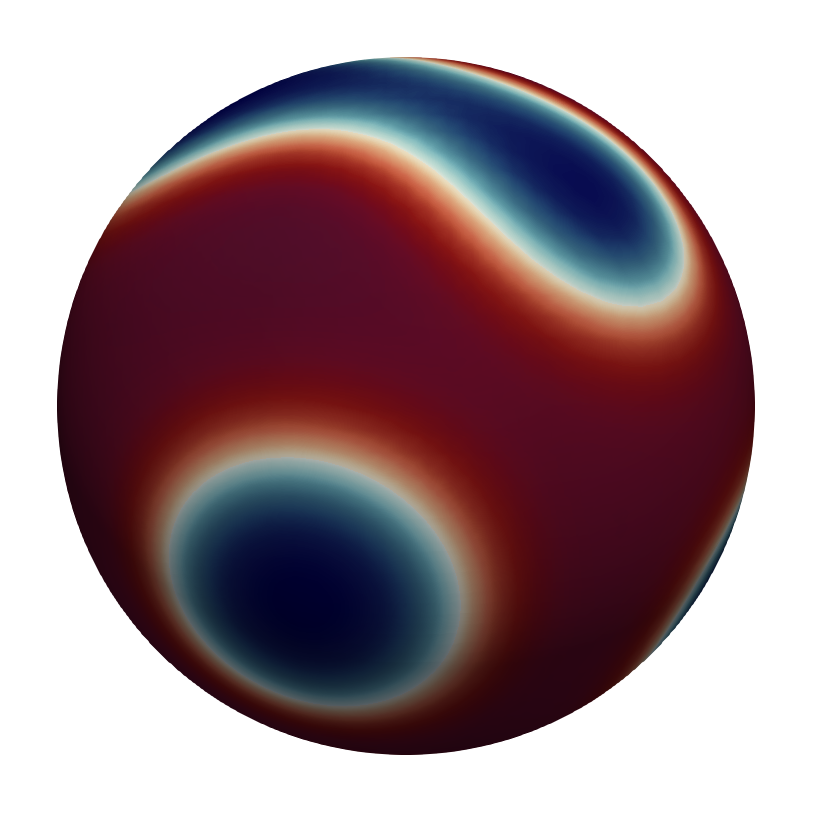}
			\put(30,102){\small{$t = 25$}}
		\end{overpic}\\
		\begin{overpic}[width=.13\textwidth,grid=false]{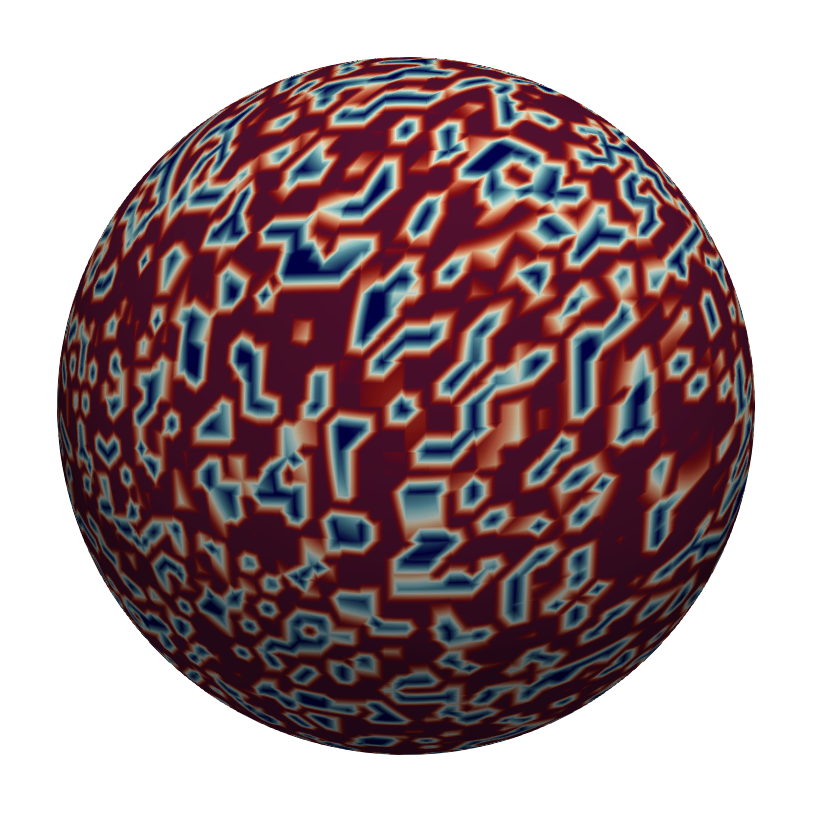}
			\put(-70,35){\small{BDF2}}
			\put(-70,55){\small{SAV}}
		\end{overpic}
		\includegraphics[width=0.13\columnwidth]{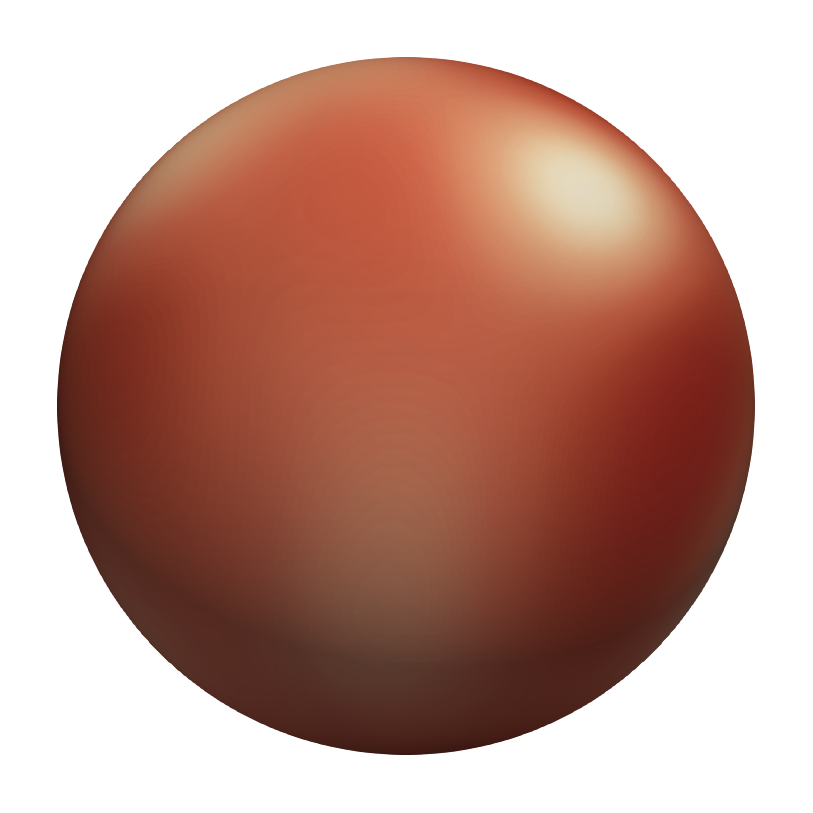}
		\includegraphics[width=0.13\columnwidth]{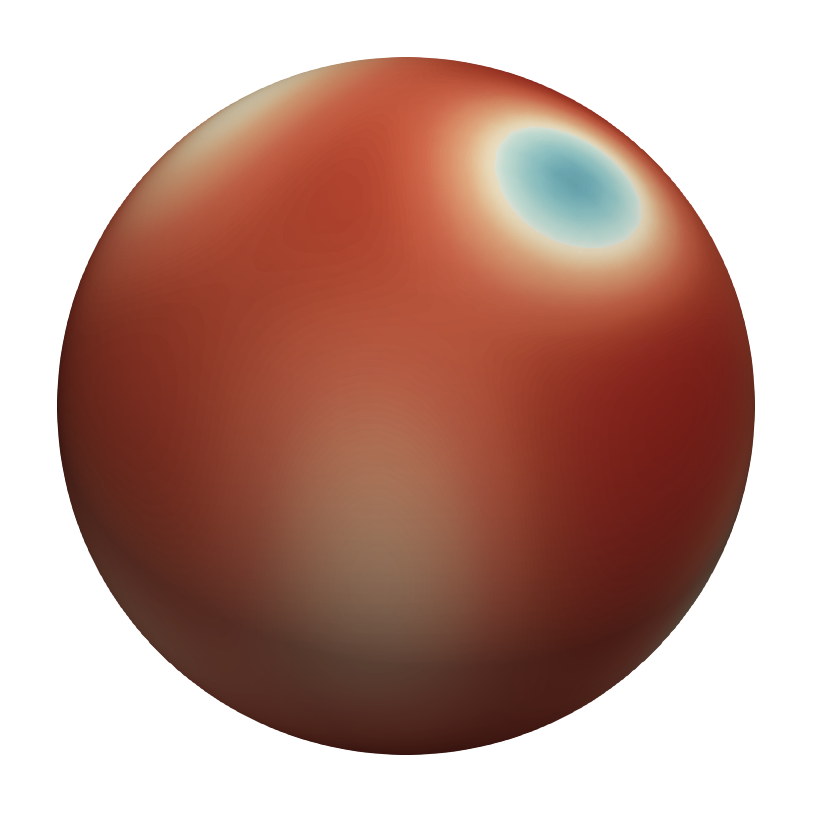}
		\includegraphics[width=0.13\columnwidth]{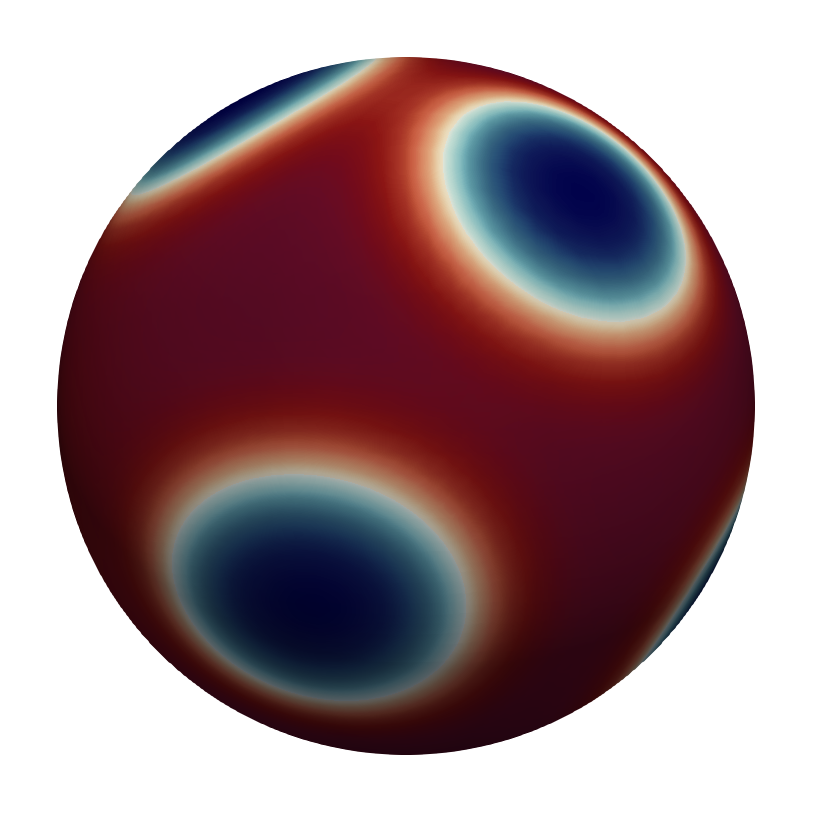}
		\includegraphics[width=0.13\columnwidth]{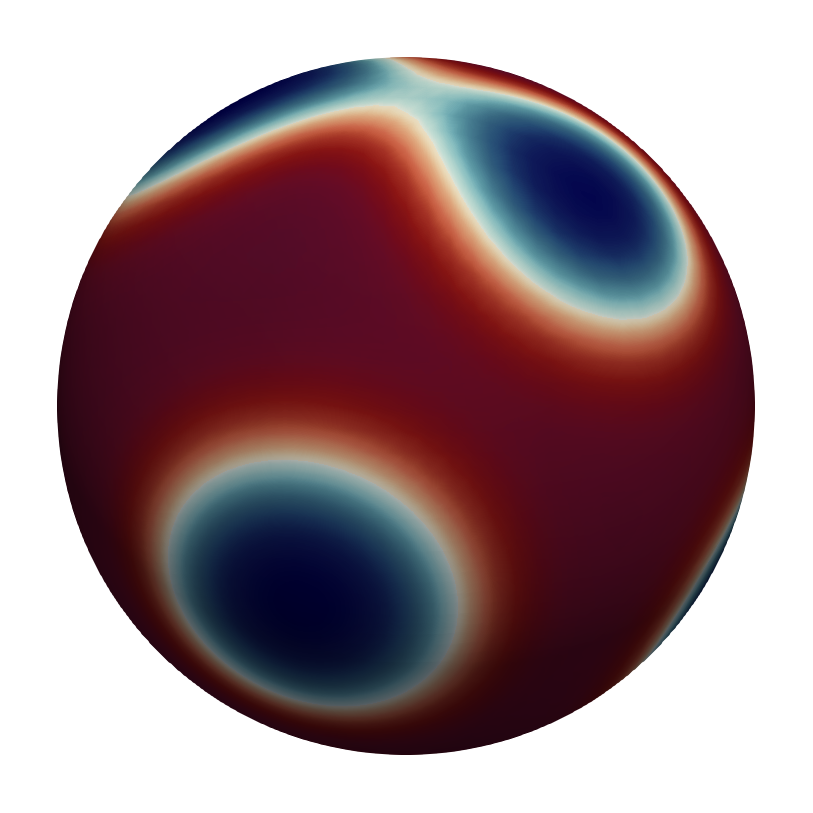}
		\includegraphics[width=0.13\columnwidth]{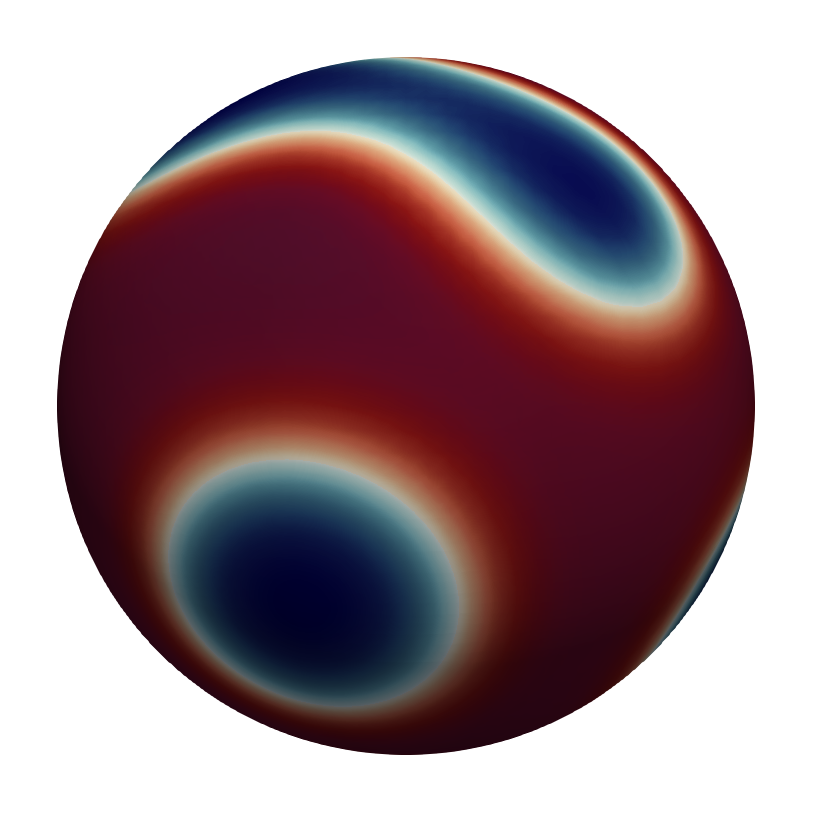}
		\\
		\includegraphics[width=0.4\columnwidth]{./img/colorbar.png}
		
	\end{center}
	\caption{Phase separation  on the sphere, $a=0.7$: evolution of phases computed with the stabilized method in \cite{Yushutin_IJNMBE2019} (top) 
	and the SAV-BDF2 method without time step adaptivity (bottom).}
		\label{fig:a_70}
\end{figure}

Fig.~\ref{fig:mod_e_three_a} displays the decay of modified energy \eqref{eq:mod_e_BDF2}
for the three values of $a$. We see that the decay is more or less rapid depending on the value of $a$.
However, in no case at $t = 25$ the system is close to an energy plateau, which we observed already at $t = 1$ for the simple 
convergence test in Sec~\ref{sec:conv_test}. See the graphs in Fig.~\ref{fig:mod_e}.

\begin{figure}[htb!]
	\centering
	\begin{overpic}[width=.32\textwidth,grid=false]{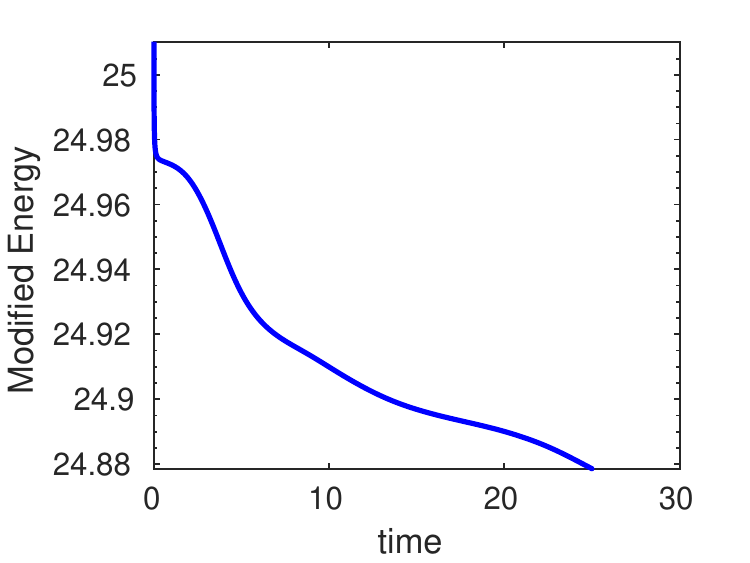}
		\put(45,72){\small{$a = 0.3$}}
	\end{overpic}
	\begin{overpic}[width=.32\textwidth,grid=false]{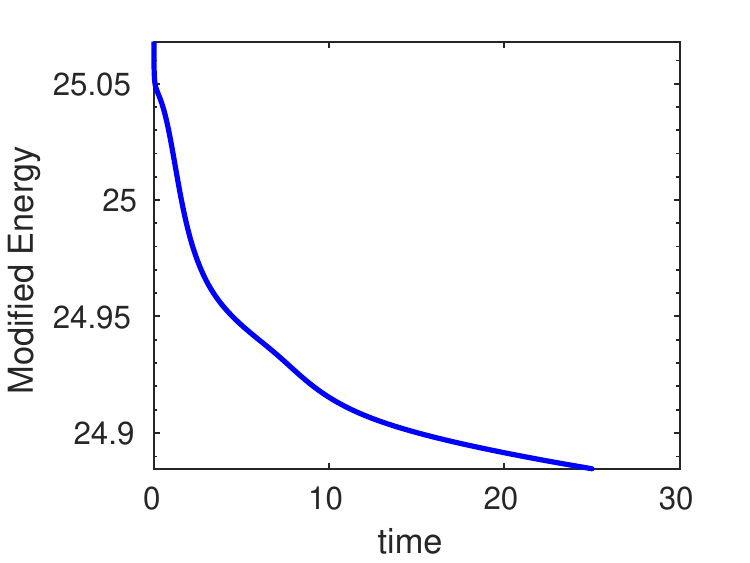}
		\put(45,72){\small{$a = 0.5$}}
	\end{overpic}
	\begin{overpic}[width=.32\textwidth,grid=false]{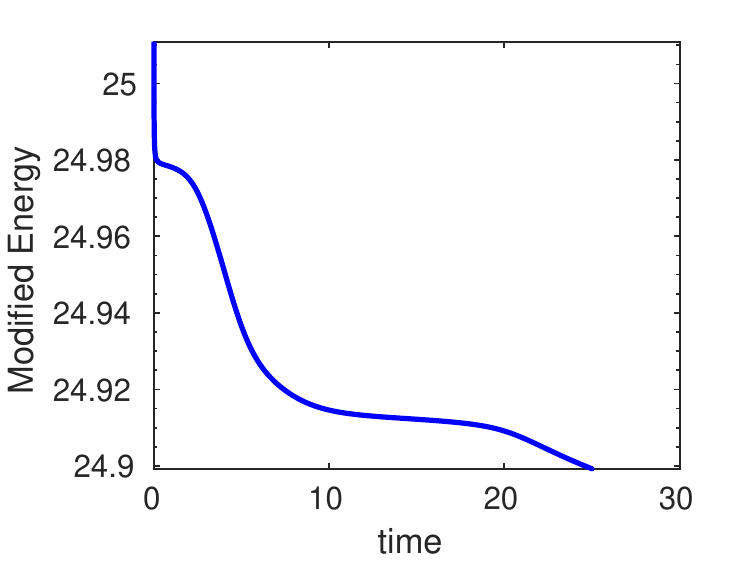}
	\put(45,72){\small{$a = 0.7$}}
	\end{overpic}	
	\caption{Phase separation  on the sphere: decay of modified energy  \eqref{eq:mod_e_BDF2} for $a = 0.3$ (left), 
	$a = 0.5$ (center), and $a = 0.7$ (right).}\label{fig:mod_e_three_a}
\end{figure}

Next, we compare the results obtained with the time-adaptive SAV-BDF2 method to those obtained with the stabilized method in \cite{Yushutin_IJNMBE2019} in its time adaptive version. For this comparison, we select only one representative value of $a$, namely $a=0.5$. In Figure~\ref{fig:a_50_ta}, which illustrates the evolution of phases until reaching the equilibrium configuration, we once again observe no difference in either the spinodal decomposition or the domain ripening between the two methods.

\begin{figure}[htb!]
	\begin{center}
		\begin{overpic}[width=.13\textwidth,grid=false]{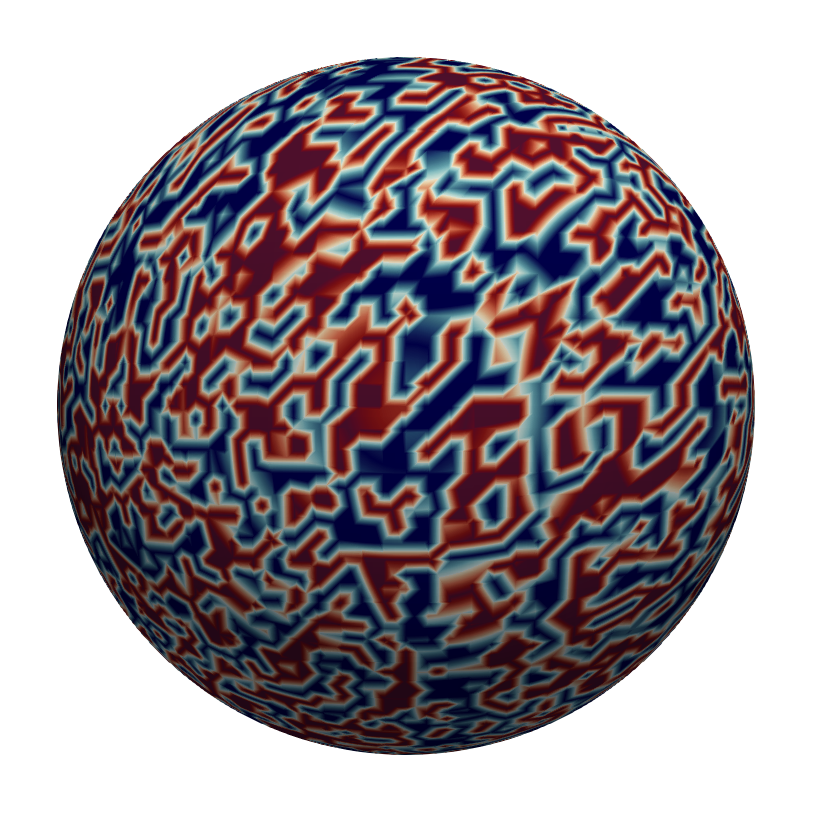}
			\put(30,102){\small{$t = 0$}}
					\put(-70,67){\small{Time}}
		\put(-70,47){\small{adaptive}}
		\put(-70,27){\small{stabilized}}
		\end{overpic}
		\begin{overpic}[width=.13\textwidth,grid=false]{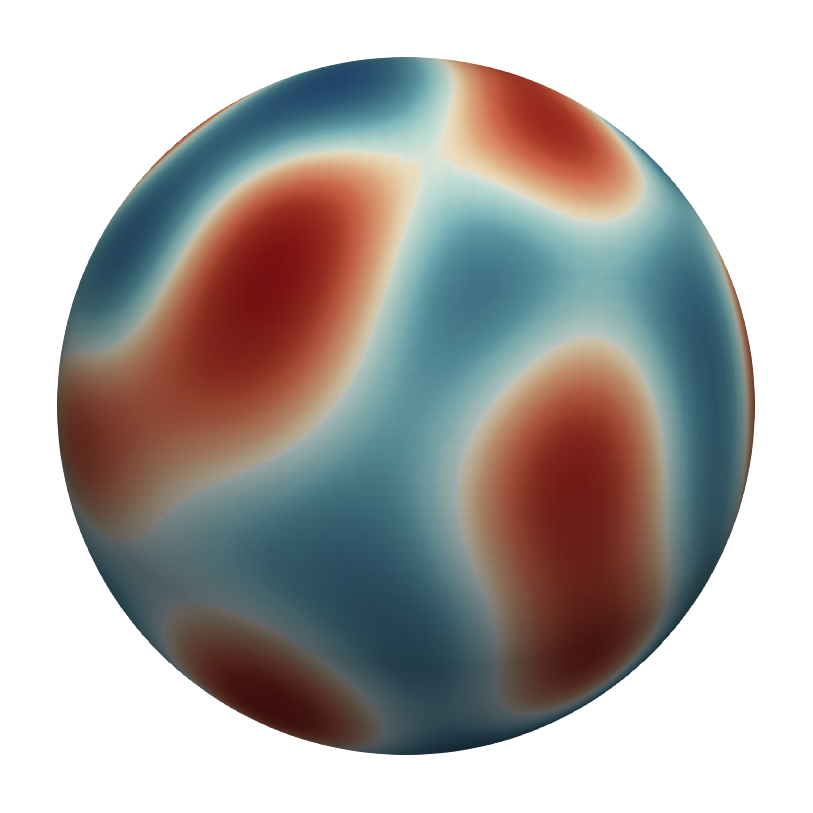}
			\put(30,102){\small{$t = 1$}}
		\end{overpic}
		\begin{overpic}[width=.13\textwidth,grid=false]{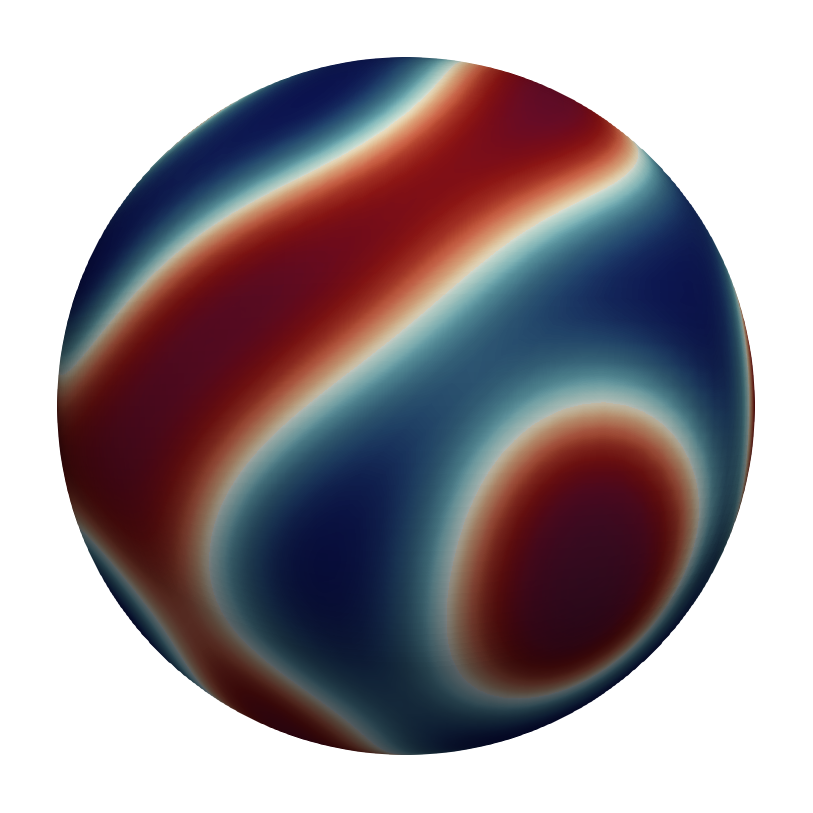}
			\put(30,102){\small{$t = 5$}}
		\end{overpic}
		\begin{overpic}[width=.13\textwidth,grid=false]{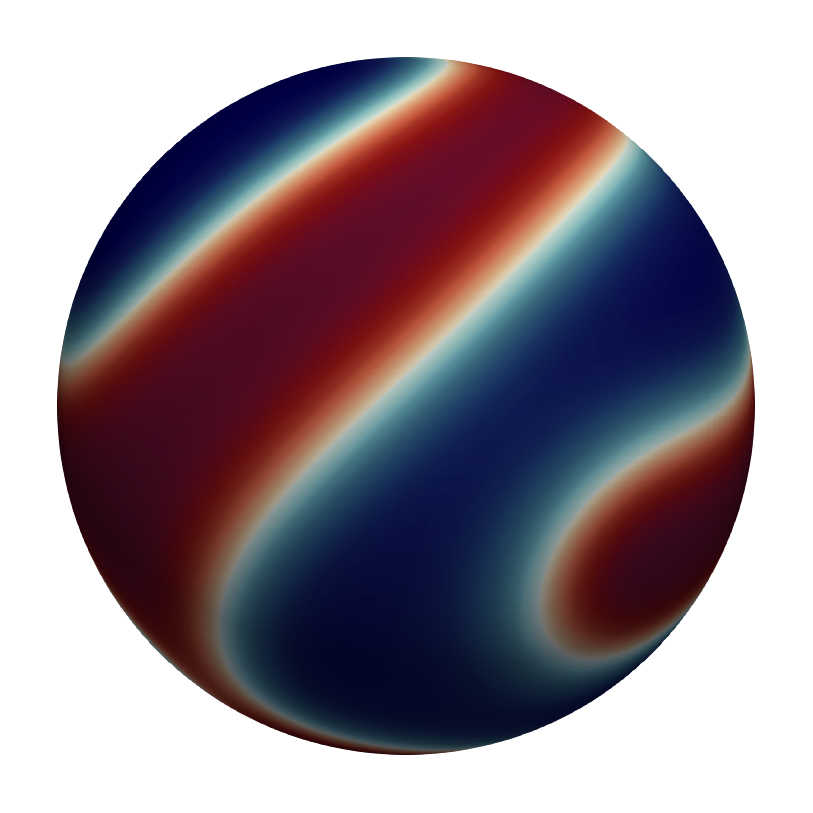}
			\put(30,102){\small{$t = 20$}}
		\end{overpic}
		\begin{overpic}[width=.13\textwidth,grid=false]{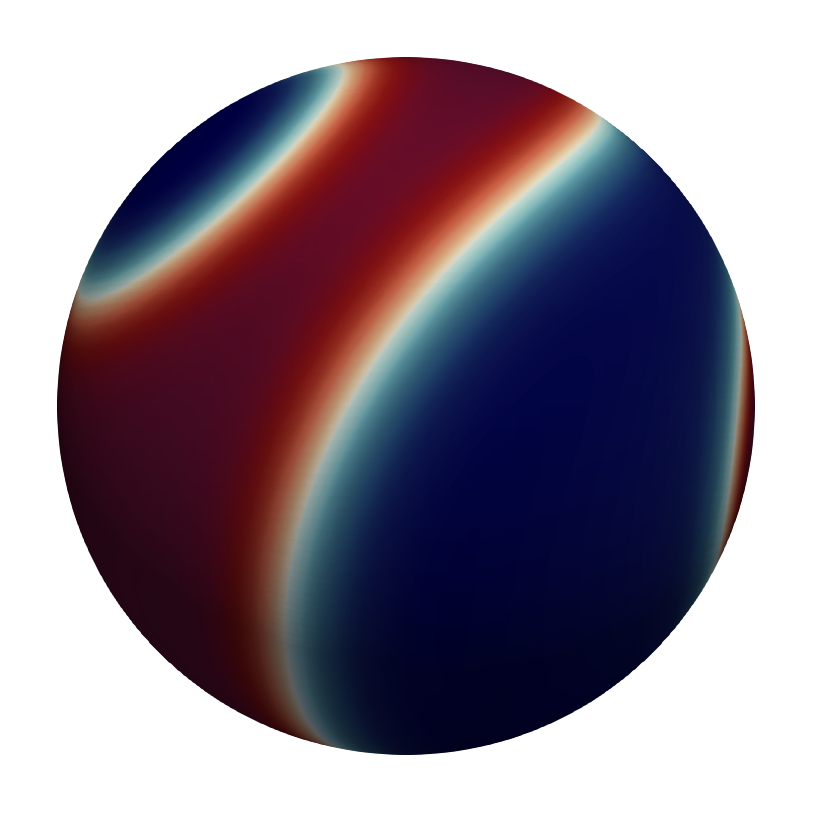}
		\put(30,102){\small{$t = 100$}}
		\end{overpic}
		\begin{overpic}[width=.13\textwidth,grid=false]{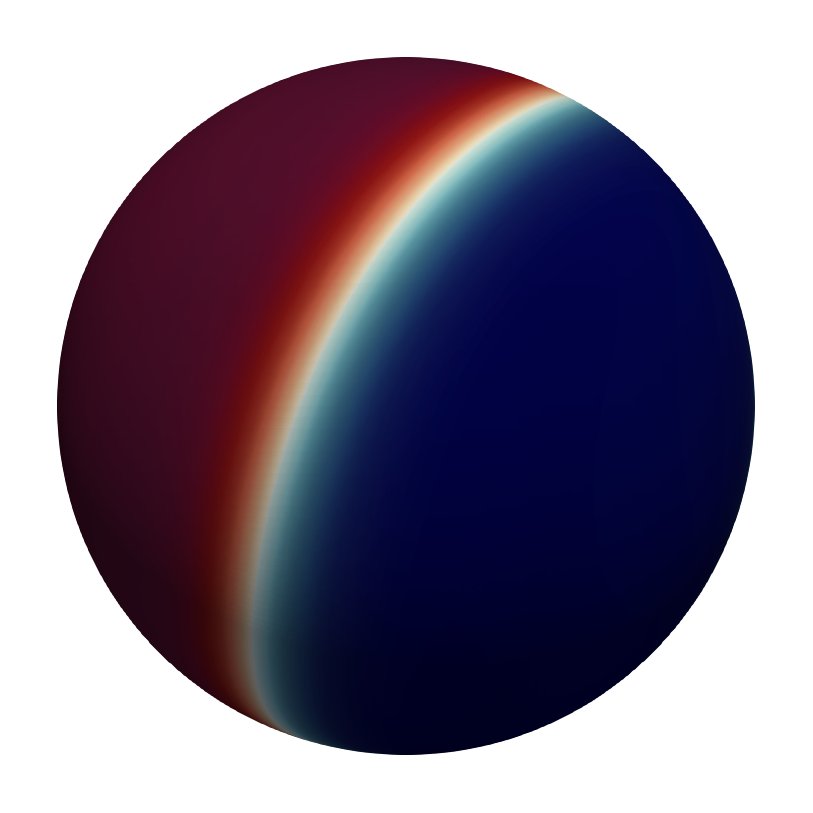}
		\put(30,102){\small{$t = 200$}}
		\end{overpic}\\
		\begin{overpic}[width=.13\textwidth,grid=false]{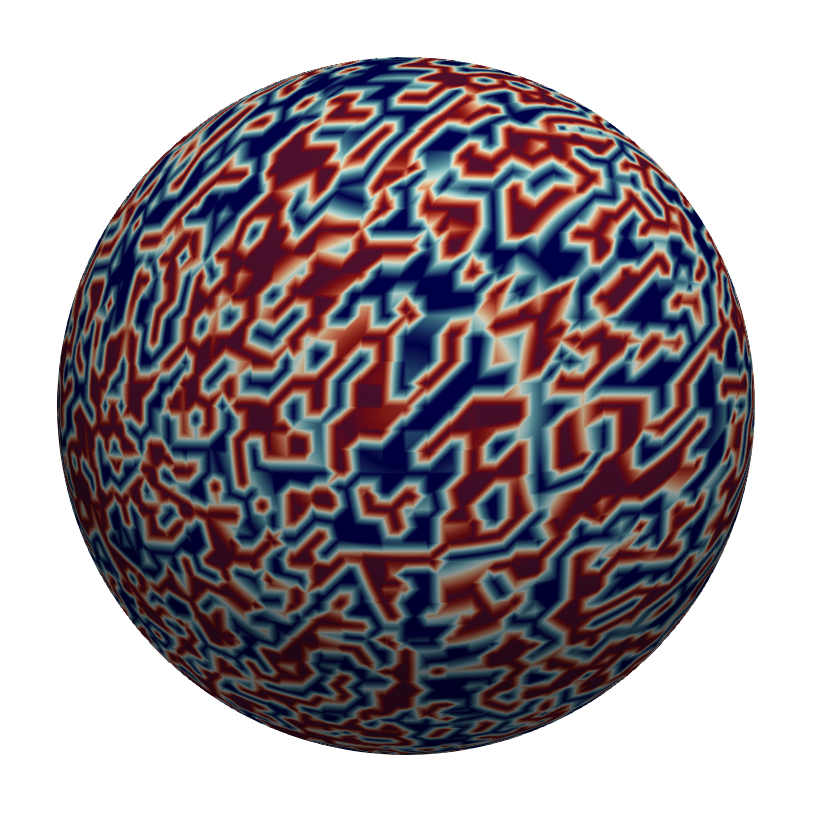}
		\put(-70,67){\small{Time}}
		\put(-70,47){\small{adaptive}}
		\put(-70,27){\small{SAV}}
		\end{overpic}
		\includegraphics[width=0.13\columnwidth]{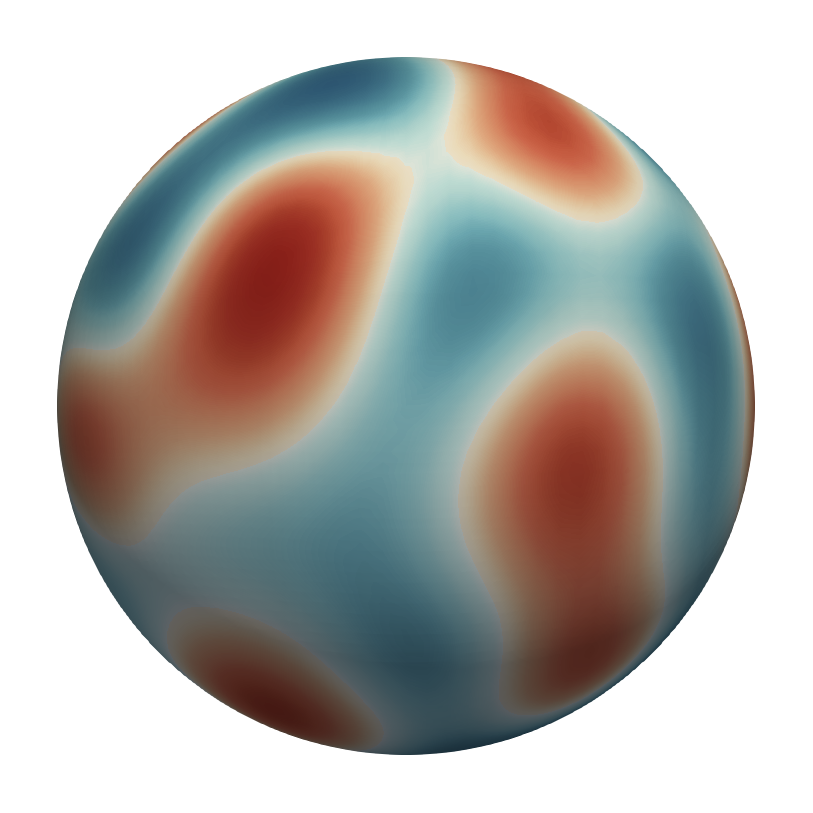}
		\includegraphics[width=0.13\columnwidth]{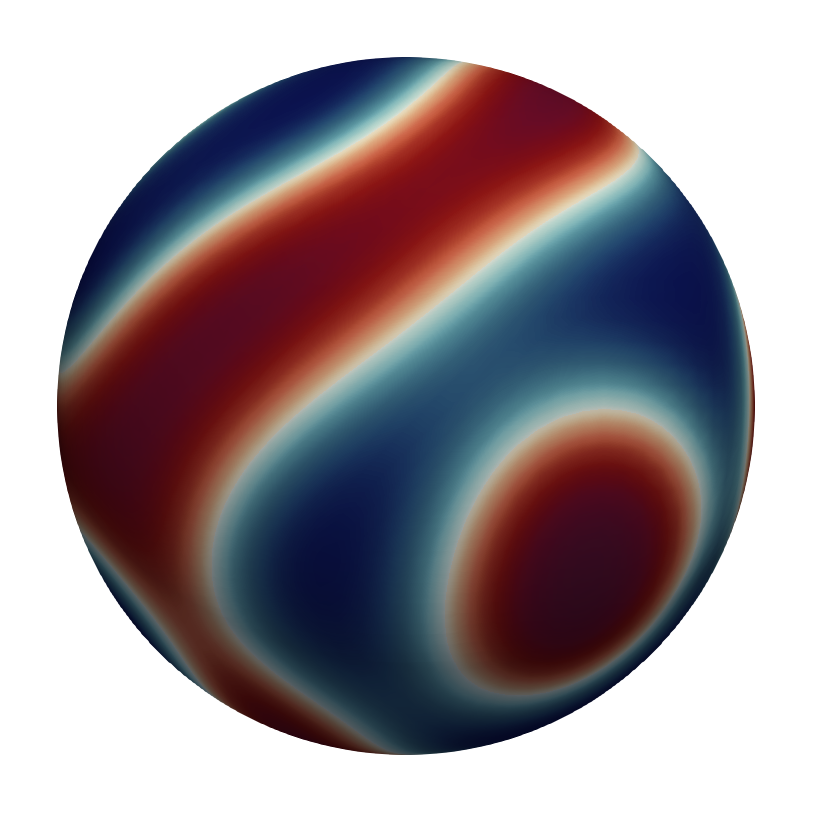}
		\includegraphics[width=0.13\columnwidth]{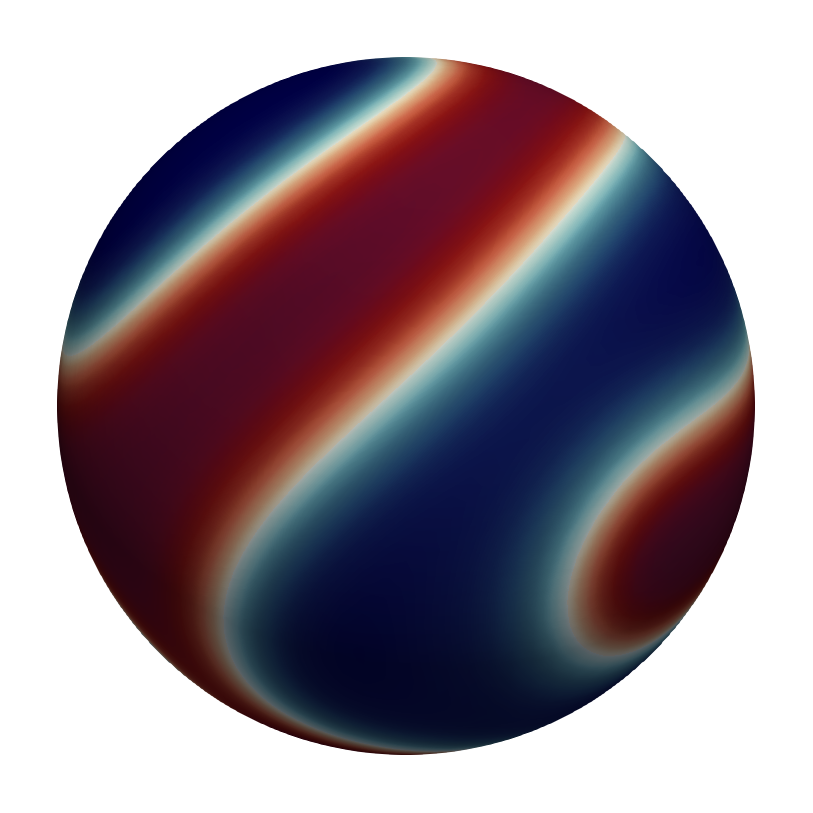}
		\includegraphics[width=0.13\columnwidth]{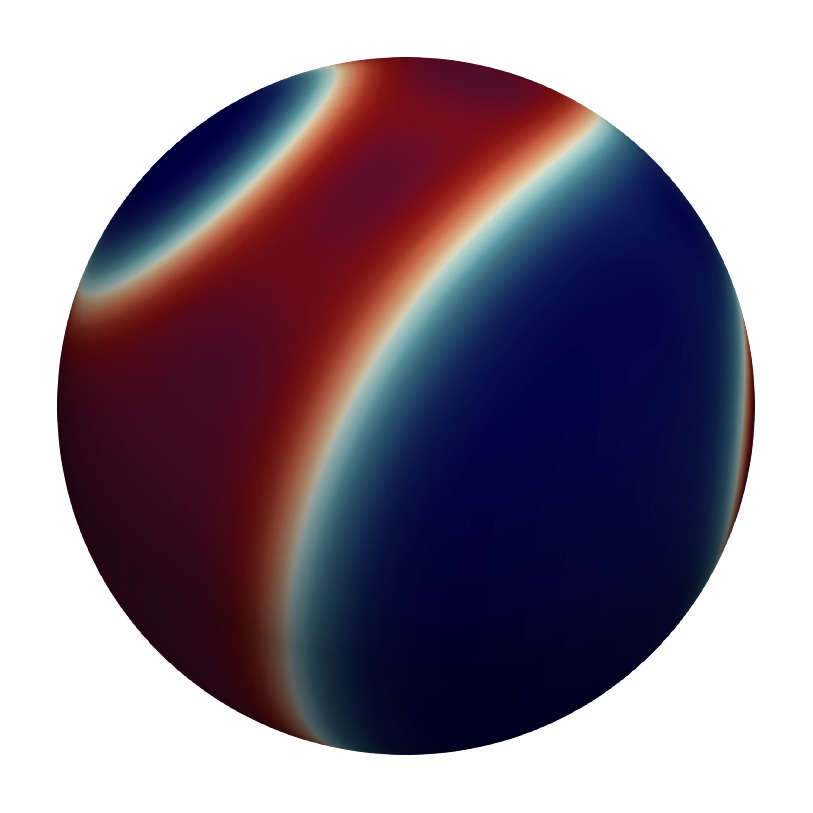}
		\includegraphics[width=0.13\columnwidth]{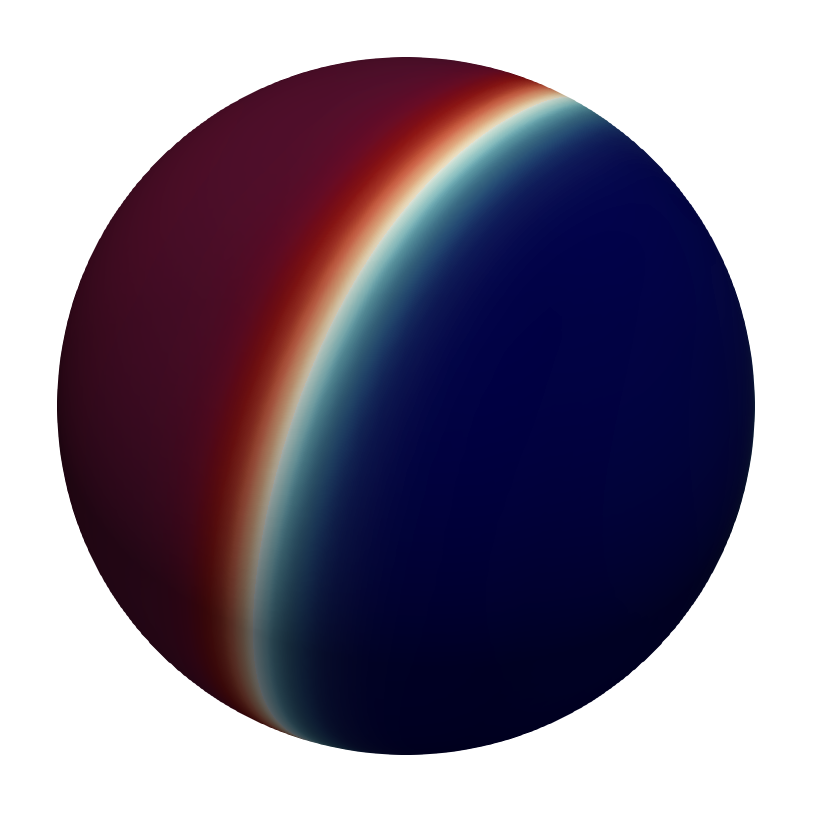}
	\end{center}
\caption{Phase separation  on the sphere, $a=0.5$: evolution of phases computed with the time-adaptive stabilized method in \cite{Yushutin_IJNMBE2019} (top) 
	and the time-adaptive SAV-BDF2 method (bottom).}\label{fig:a_50_ta} 
\end{figure}

A comparison of the time step sizes and time step number over time is shown in Fig.~\ref{fig:dt}. 
From Fig.~\ref{fig:dt} (left), we can see that the time step grows for both methods until approximately $t=50$, after which it fluctuates around $\Delta t=1$. Although the time step sizes are generally comparable for both methods, the SAV method utilizes slightly larger time steps during this initial integration stage. Consequently, time step number $n$ required to integrate the system up to any $t\leq200$ is smaller for the time-adaptive SAV-BDF2 method compared to the time-adaptive stabilized method in \cite{Yushutin_IJNMBE2019}. However, the difference is not significant. See Fig.~\ref{fig:dt} (right).

%

\begin{figure}[htb!]
	\centering
\includegraphics[width=0.4\textwidth]{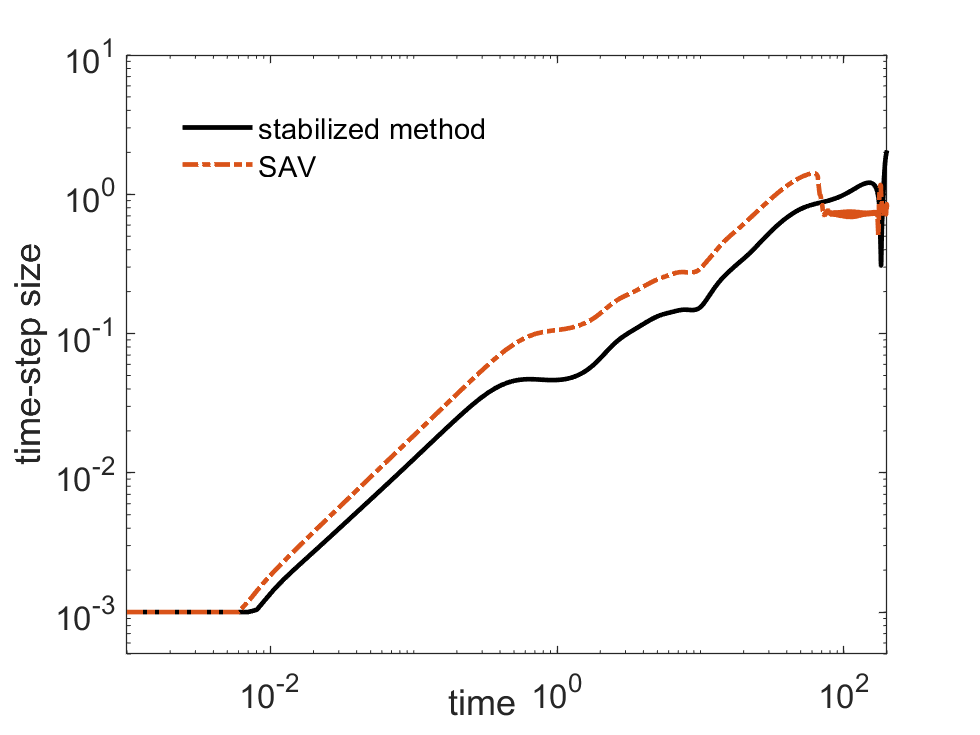}
\includegraphics[width=0.4\textwidth]{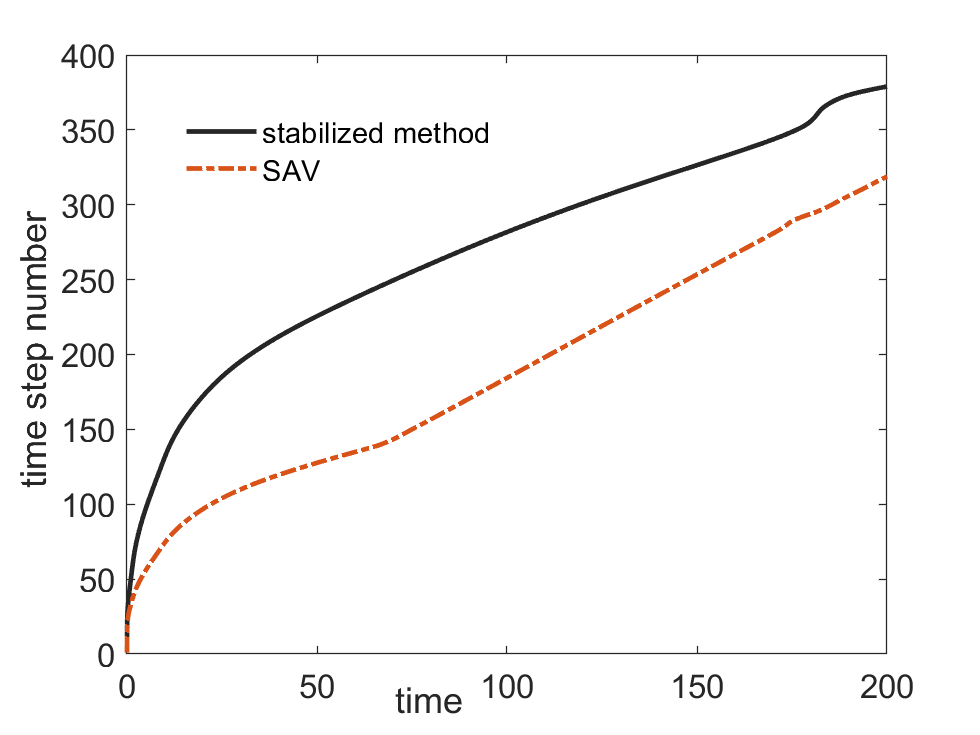}
\caption{Phase separation  on the sphere, $a=0.5$: evolution of the time step size $\Delta t$ (left)
and number of the time steps required at each time instant (right) for 
the time-adaptive stabilized method from \cite{Yushutin_IJNMBE2019} and the BDF2-SAV method with time step adaptivity.}\label{fig:dt}
\end{figure}

We conclude this section with a comment on the computational time. 
All the computations were executed on a machine with an AMD EPYC 7513 32-Core Processor and 512 GB RAM. 
Fig.~\ref{fig:comp_time} reports the computational time needed by the simulations whose results are shown 
in Fig.~\ref{fig:a_50}, \ref{fig:a_30}, and \ref{fig:a_70} to complete the first 100 time steps.
The time required by the stabilized method in \cite{Yushutin_IJNMBE2019}
varies between one half and two thirds of the time required by the SAV method with no time step adaptivity. 
Let us now turn to the simulations in Fig.~\ref{fig:a_50_ta}, i.e., those with time adaptivity. 
The time-adaptive SAV-BDF2 method takes 319 time steps in time interval $(0,200]$ for a total computational time of 
about 41 minutes, 
while the time-adaptive stabilized method in \cite{Yushutin_IJNMBE2019} takes about 9 minutes 
to complete 379 time steps in the same time interval. The simulation with the time-adaptive SAV-BDF2 method
requires less time steps but takes longer overall. As mentioned at the end of Sec.~\ref{sec:impl}, the reason for this 
difference in the computational times is due to the fact that the extra terms introduced by the SAV method make 
the matrices of the associated linear systems dense. 
If one used a finite difference method on uniform grids for space discretization as in \cite{huang2020highly,LiShen2019,Shen2019_SIAMrev}, 
higher computational efficiency could be achieved for the SAV method. Our preference for a finite element method
and non-uniform meshes is for greater geometric flexibility, as shown in the next subsection. 


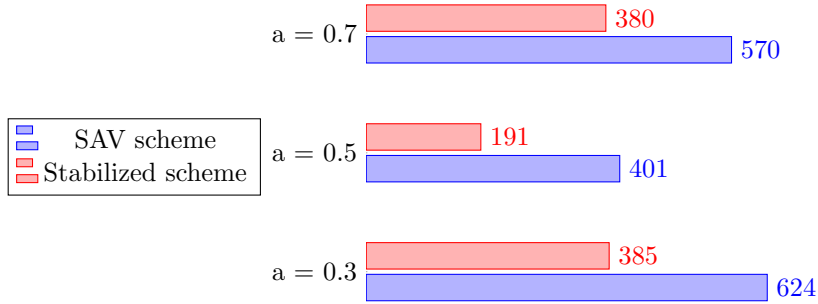
\begin{figure}[htb!]
	\centering
	\begin{tikzpicture}
		\begin{axis}[
			xbar,
			y axis line style = { opacity = 0 },
			axis x line       = none,
			tickwidth         = 0pt,
			enlarge y limits  = 0.4,
			enlarge x limits  = 0.4,
			symbolic y coords = {a = 0.3, a = 0.5, a = 0.7},
			nodes near coords,
			legend style={at={(-0.45,0.4)},anchor=south}
			]
			\addplot coordinates { (624,a = 0.3)         (401,a = 0.5)
				(570,a = 0.7)   };
			\addplot coordinates { (385,a = 0.3)         (191,a = 0.5)
				(380,a = 0.7)     };
			\legend{SAV scheme, Stabilized scheme}
		\end{axis}
	\end{tikzpicture}\caption{Phase separation on the sphere: computational time (in s) needed by the stabilized method in \cite{Yushutin_IJNMBE2019} and  the SAV method with no time step adaptivity to  complete the first 100 time steps of
	the simulations in Fig.~\ref{fig:a_50} ($a = 0.5$), \ref{fig:a_30} ($a = 0.3$), and \ref{fig:a_70} ($a = 0.7$)
}\label{fig:comp_time}
\end{figure}

\subsection{Phase separation on a complex manifold}\label{sec:res_cell}

Because of our interest in phase separation on biological membranes in general, not just lipid vescicles, we need to be able to handle
surfaces that are more complex than the sphere. Here, we consider an idealized cell with surface $\Gamma$ given 
by the zero level set of following function \cite{dziuk2013finite,Yushutin_IJNMBE2019}:
$$\phi(\bx) = \frac{1}{4}x_1^2+x_2^2+\frac{4x_3^2}{(1+\frac{1}{2}\sin(\pi x_1))^2} -1.$$
Fig.~\ref{fig:cell} illustrates a side view of this complex manifold and an angle view of the surface mesh.

\begin{figure}[htb!]
	\centering
	\includegraphics[width=0.4\linewidth]{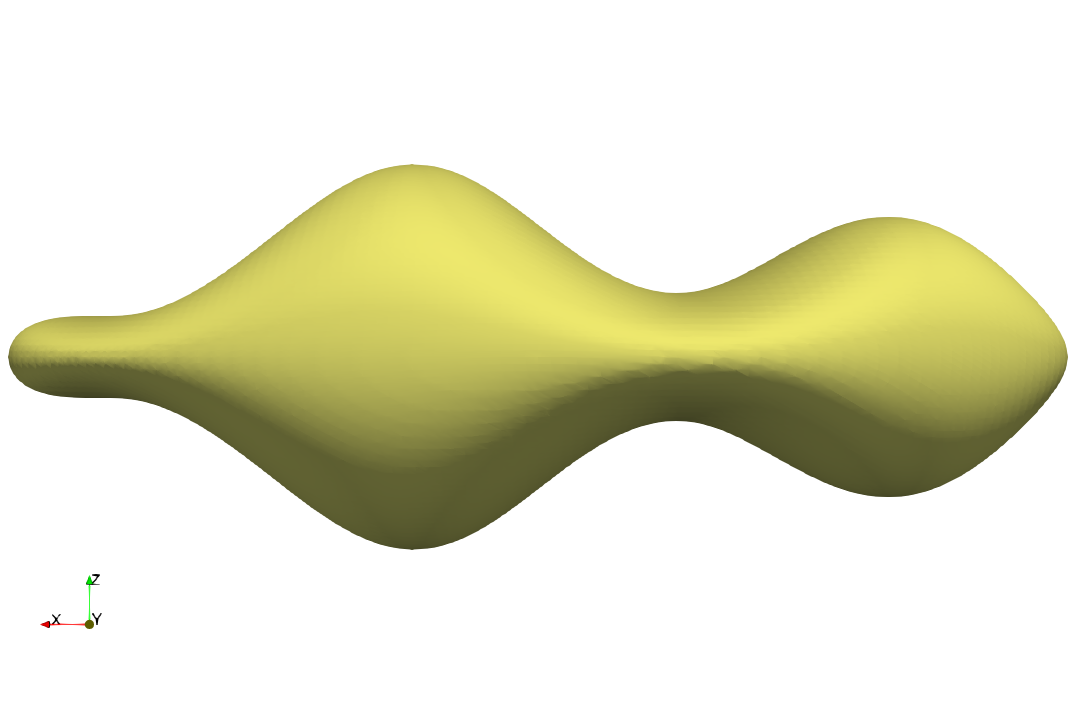}
	\includegraphics[width=0.4\linewidth]{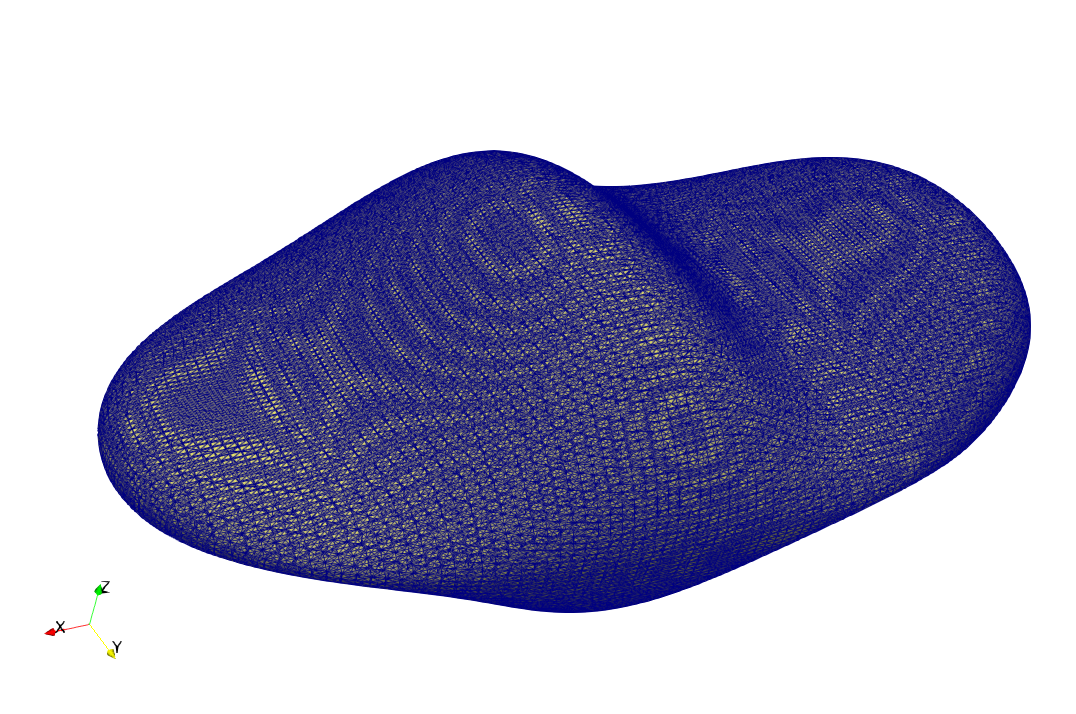}
	\caption{Illustration of the complex manifold}
	\label{fig:cell}
\end{figure}

We embed surface $\Gamma$ in bulk domain $\Omega=[-2,2]\times [-4/3,4/3]\times [-4/3,-4/3]$.  A tetrahedral mesh for $\Omega$ 
is generated in the same way as for the cases in the previous subsection, i.e., by diving $\Omega$ into cubes and then diving the cubes into tetrahedra. The active elements, which are the elements that intersect surface, are further refined for a total of 14298 degrees of freedom. This mesh has a level of refinement comparable to mesh $\ell = 5$ in Sec.~\ref{sec:res_spehre}.
We fix the time step to  $\Delta t = 0.005$ and do not allow for time step adaptivity. 

We set the interface thickness $\epsilon$ to 0.05, like in Sec.~\ref{sec:res_spehre}. Fig.~\ref{fig:a_50_cell} 
compares the evolution of the phases given by SAV- BDF2 method without time step adaptivity 
with the evolution given by the stabilized method in \cite{Yushutin_IJNMBE2019} for $a=0.5$. 
We recall that $a=0.5$ means that 50\% of the idealized cell surface is covered by 
the representative (red) phase and the remaining 50\% is covered by 
the other phases. Just like in the case of the sphere (see Fig.~\ref{fig:a_50}), 
there is no observable difference in the results given by the two methods.

\begin{figure}[htb!]
	\begin{center}
		\begin{overpic}[width=.16\textwidth,grid=false]{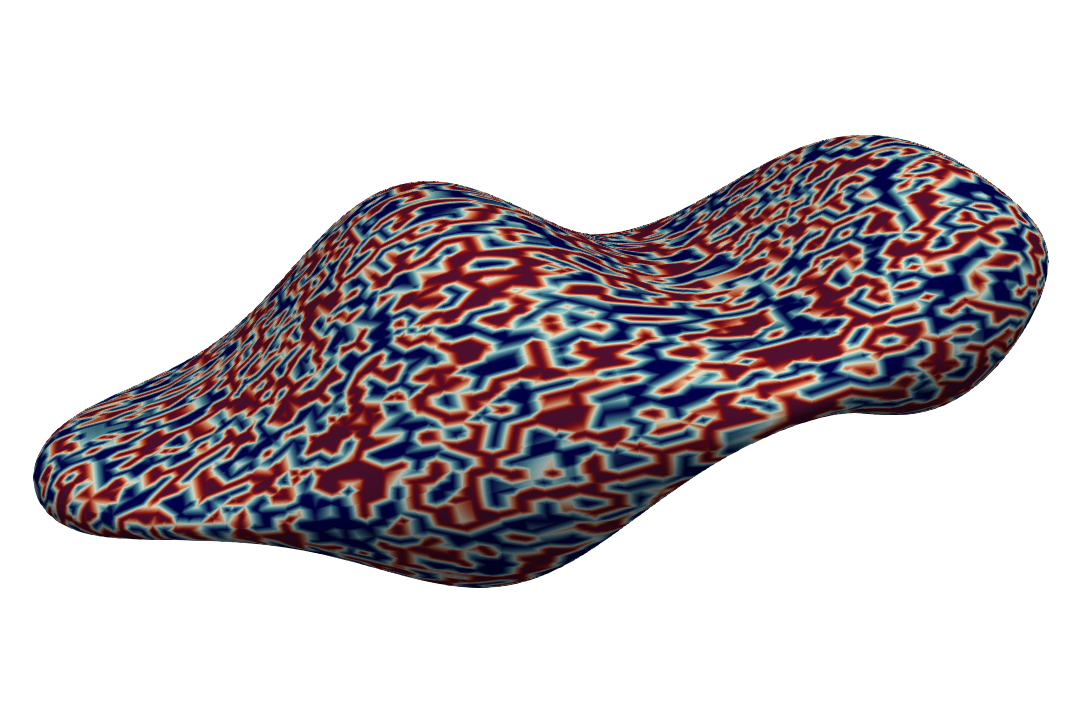}
			\put(30,62){\small{$t = 0$}}
			\put(-50,45){\small{Stabilized}}
		\end{overpic}
		\begin{overpic}[width=.16\textwidth,grid=false]{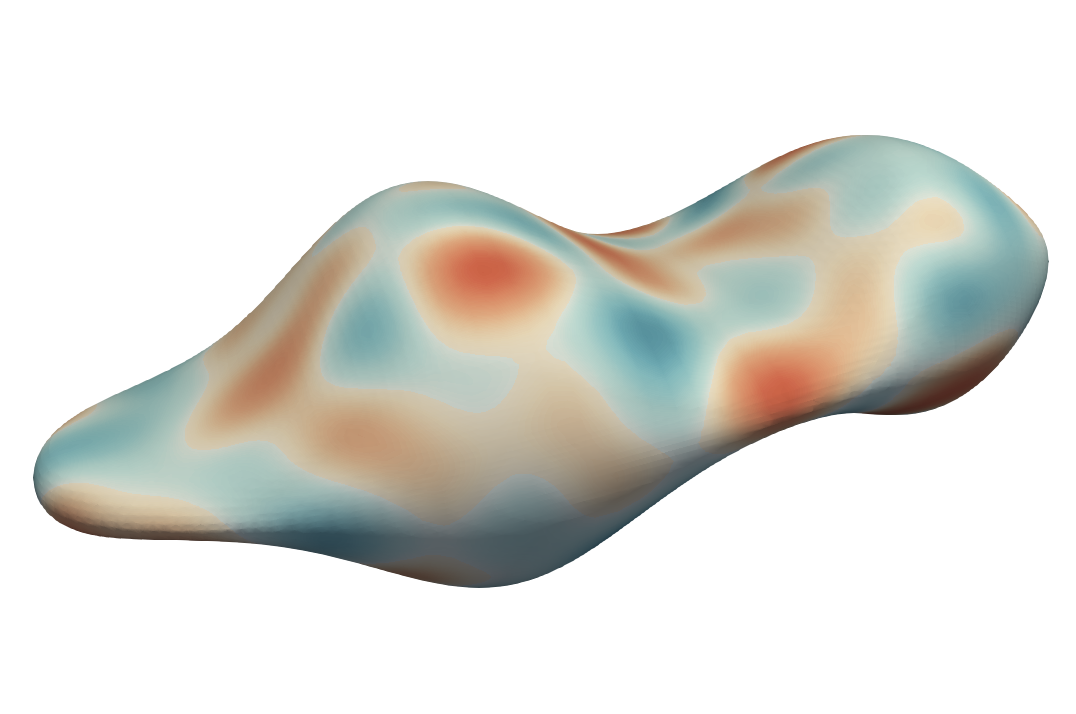}
			\put(30,62){\small{$t = 0.1$}}
		\end{overpic}
		\begin{overpic}[width=.16\textwidth,grid=false]{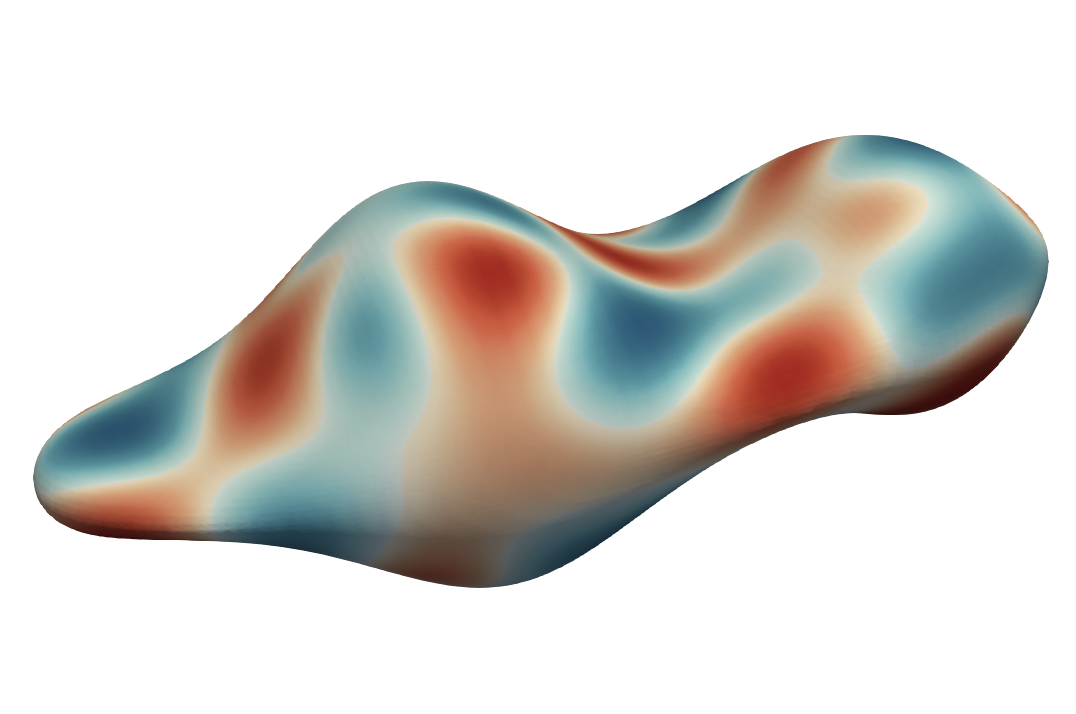}
			\put(30,62){\small{$t = 1$}}
		\end{overpic}
		\begin{overpic}[width=.16\textwidth,grid=false]{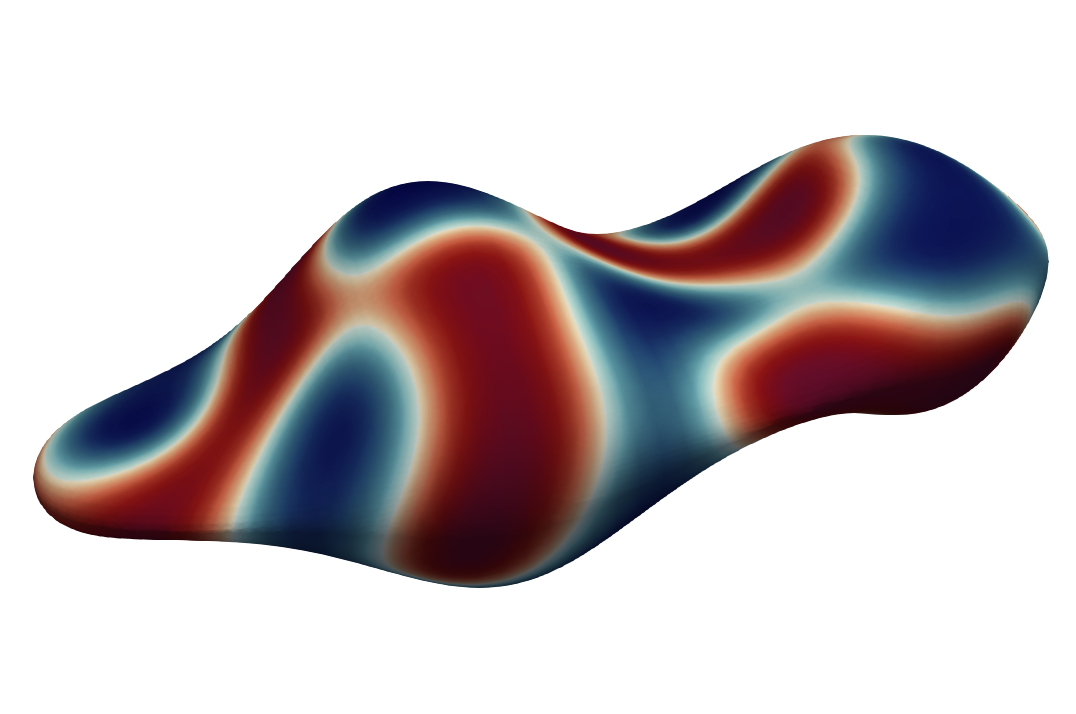}
			\put(30,62){\small{$t = 5$}}
		\end{overpic}
		\begin{overpic}[width=.16\textwidth,grid=false]{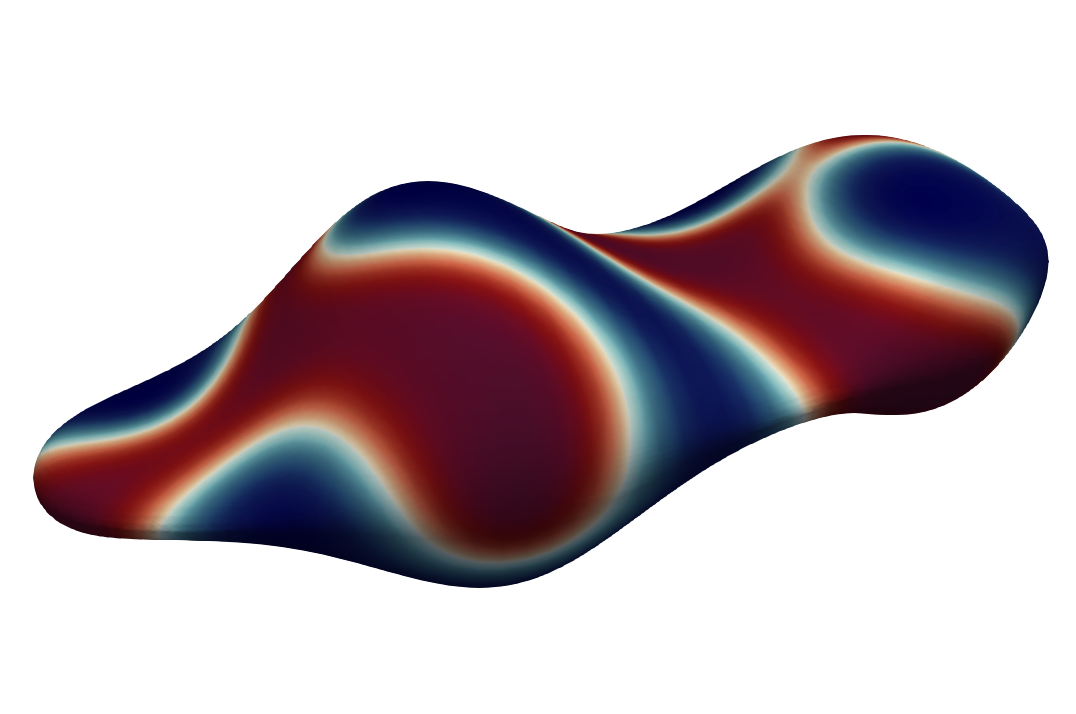}
			\put(30,62){\small{$t = 15$}}
		\end{overpic}\\
		\begin{overpic}[width=.16\textwidth,grid=false]{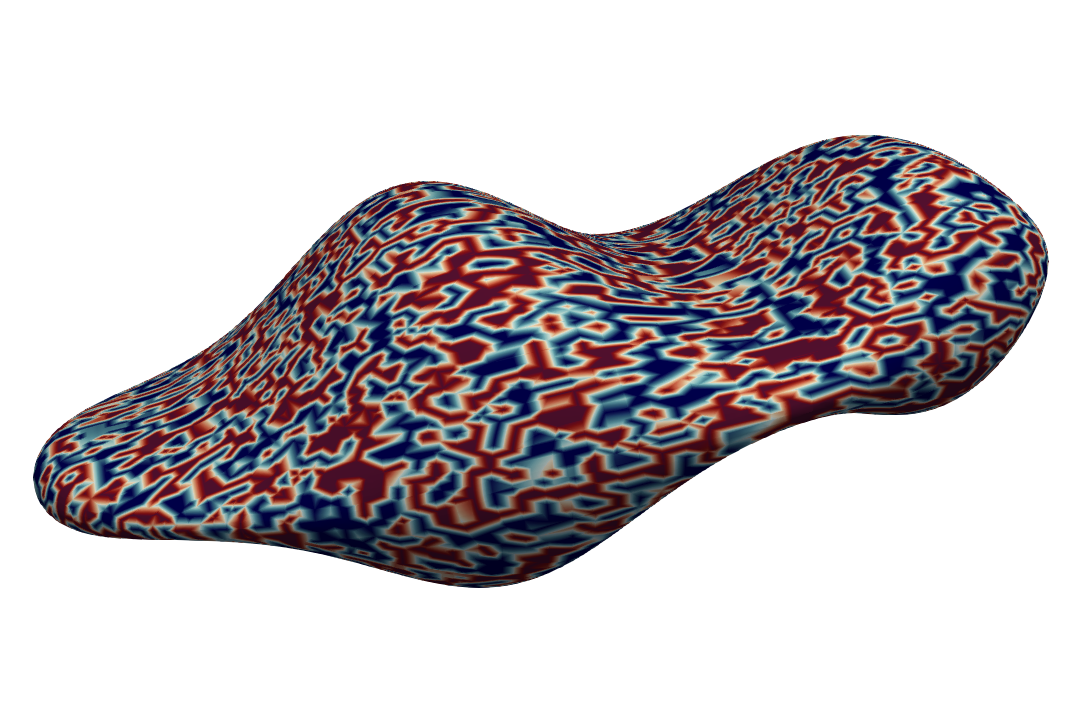}
			\put(-50,35){\small{BDF2}}
			\put(-50,55){\small{SAV}}
		\end{overpic}
		\includegraphics[width=0.16\columnwidth]{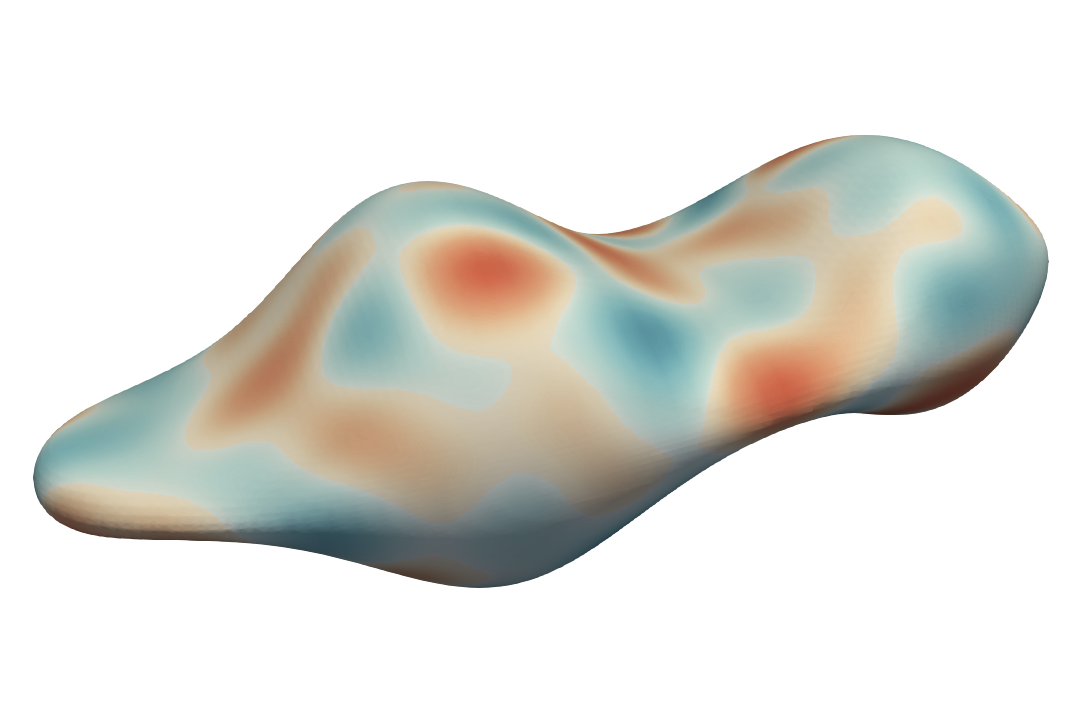}
		\includegraphics[width=0.16\columnwidth]{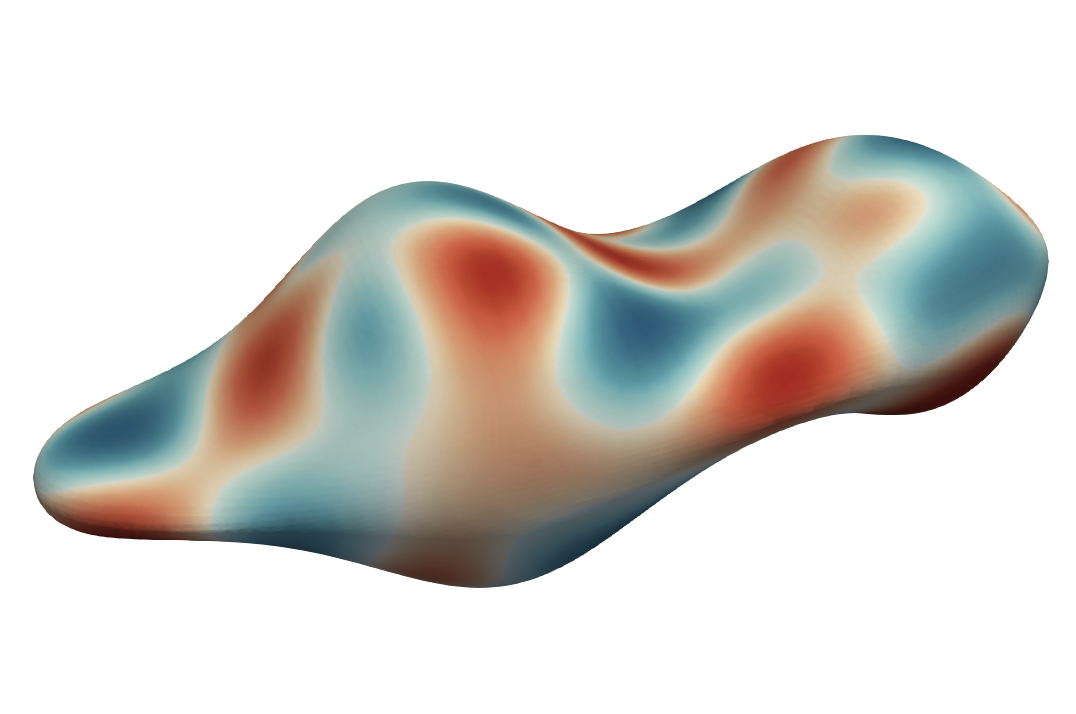}
		\includegraphics[width=0.16\columnwidth]{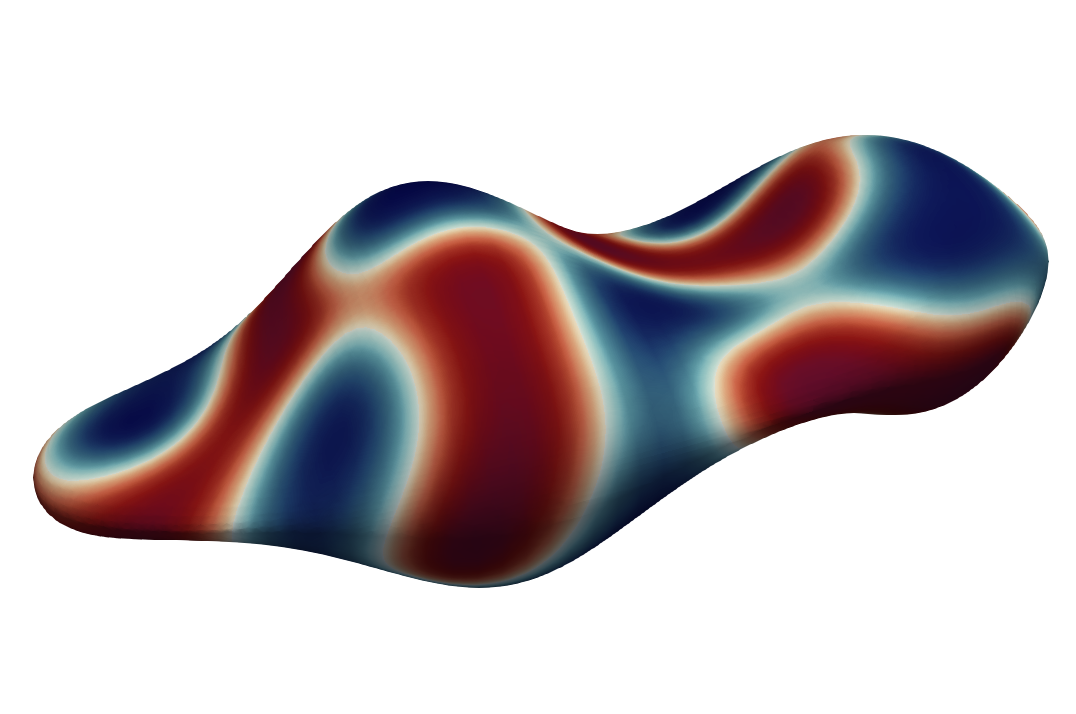}
		\includegraphics[width=0.16\columnwidth]{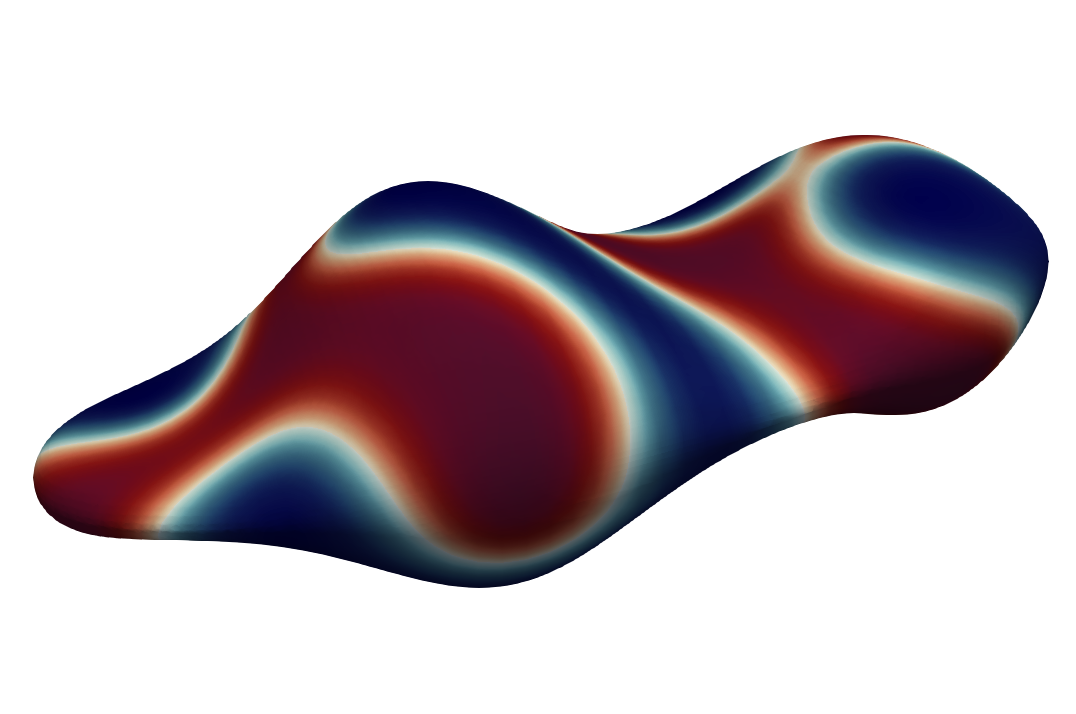}
		\\
		\includegraphics[width=0.4\columnwidth]{./img/colorbar.png}
		
	\end{center}
	\caption{Phase separation on an idealized cell: evolution of phases computed with the stabilized method in \cite{Yushutin_IJNMBE2019} (top) 
		and the SAV method and BDF2 without time step adaptivity (bottom).
	\label{fig:a_50_cell}}
\end{figure}

\section{Conclusion}\label{sec:concl}

The paper introduced and investigated an SAV formulation of the geometrically unfitted trace finite element method for the surface Cahn--Hilliard equations with degenerate mobility. The BDF1 and BDF2 versions of the method were proven to dissipate specific energy, thus conforming to the fundamental property of the continuous problem. The method demonstrated optimal convergence rates for smooth solutions and performed well in predicting phase separation and pattern formation in spherical and more complex shapes. Thus, it proved to be a valuable tool in modeling multicomponent lipid vesicles. A comparison with a semi-explicit mixed trace finite element method formulation with stabilization from~\cite{Shen_Yang2010} shows very similar performance of both methods for the given class of problems. 
Both methods are well-suited for time adaptation. Experiments suggested that SAV method allows for somewhat larger 
time steps when the same adaptive criteria are used for the SAV and semi-explicit stabilized methods. 
The stabilized method requires an additional parameter to be chosen, while the SAV method adds a rank-one dense matrix to the resulting system of algebraic equations, which must be solved at each time step. The availability of a fast algebraic solver for such systems may determine one's preference between these two solid methods.

\section*{Acknowledgments}

This work was partially supported by US National Science Foundation (NSF) through grant DMS-1953535.
M.O.~also acknowledges the support from NSF through DMS-2309197.

\bibliographystyle{abbrv}
\bibliography{literature}

\end{document}